\documentclass[msom,non-blindrev]{informs3} % current default for manuscript submission
\OneAndAHalfSpacedXI % current default line spacing
\usepackage{endnotes}
\usepackage[ruled]{algorithm}
\usepackage{algorithmicx,algpseudocode,amsfonts,subfigure,optidef}
\usepackage{graphicx,hyperref,comment,appendix}
\hypersetup{%
    pdfborder = {0 0 0}
}
\usepackage[capitalize,noabbrev]{cleveref}
\usepackage{tikz}
\usepackage{multirow,makecell}
\usepackage{subfigure,caption,microtype,booktabs} % for professional tables 

\usepackage{amsthm,amsmath,amssymb,soul}
\usepackage{comment,enumitem,diagbox}
\usepackage{subcaption}
\setlist{leftmargin=*}

\newtheorem{theorem}{Theorem}

\newtheorem{lemma}[theorem]{Lemma}
\newtheorem{proposition}[theorem]{Proposition}
\newtheorem{remark}[theorem]{Remark}
\theoremstyle{definition}
\newtheorem{example}[theorem]{Example}

\newtheorem{assumption}[theorem]{Assumption}

\newcommand{\eg}{\textit{e.g.}}

\usepackage{natbib}
\bibpunct[, ]{(}{)}{,}{a}{}{,}%
\def\bibfont{\small}%
\usepackage{bm,comment}

\ECRepeatTheorems

\EquationsNumberedThrough    % Default: (1), (2), ...
\allowdisplaybreaks
\begin{document}
%%%%%%%%%%%%%%%%

%\setcounter{page}{1}
%\DOI{10.1287/xxxx.0000.0000}%

\RUNAUTHOR{}

\RUNTITLE{}

% \TITLE{Bundling Information Goods of Decreasing Value}
\TITLE{Assortment Optimization under the Multinomial Logit Model with Covering Constraints}

\ARTICLEAUTHORS{Omar El Housni \quad \quad Qing Feng \quad \quad Huseyin Topaloglu \\ School of Operations Research and Information Engineering, Cornell Tech, Cornell University  \\ \EMAIL{\{oe46,qf48,ht88\}@cornell.edu} \URL{}}

%} % end of the block

%%%%%%%%%%% Arxiv %%%%%%%%%%%%$

\ABSTRACT{ We consider an assortment optimization problem under the multinomial logit choice model with general covering constraints. In this problem, the seller offers an assortment that should contain a minimum number of products from multiple categories. We refer to these constraints as covering constraints. Such constraints are common in practice, as retailers often need to meet service-level agreements with suppliers or simply wish to impose diversity in the assortment to meet different customers’ needs. We consider both the deterministic version, where the seller decides on a single assortment, and the randomized version, where they choose a distribution over assortments. In the deterministic case, we provide a $1/(\log K+2)$-approximation algorithm, where $K$ is the number of product categories, matching the problem's hardness up to a constant factor. For the randomized setting, we show that the problem is solvable in polynomial time via an equivalent linear program. We also extend our analysis to multi-segment assortment optimization with covering constraints, where there are $m$ customer segments, and an assortment is offered to each. In the randomized setting, the problem remains polynomially solvable. In the deterministic setting, we design a $(1 - \epsilon) / (\log K + 2)$-approximation algorithm for constant $m$ and a $1 / (m (\log K + 2))$-approximation for general $m$, which matches the problem’s hardness up to a logarithmic factor. Finally,  we conduct a numerical experiment using real data from an online electronics store, categorizing products by price range and brand. Our findings demonstrate that, in practice, it is feasible to enforce a minimum number of representatives from each category while incurring a relatively small revenue loss. Moreover, we observe that the optimal expected revenue in both deterministic and randomized settings is often comparable, and the randomized setting's optimal solution typically involves only a few assortments.}

\KEYWORDS{Approximation algorithms, Multinomial logit model, Assortment optimization, Covering constraints} 
%\HISTORY{This paper was first submitted on April 12, 1922 and has been with the authors for 83 years for 65 revisions.}

\maketitle
%%%%%%%%%%%%%%%%%%%%%%%%%%%%%%%%%%%%%%%%%%%%%%%%%%%%%%%%%%%%%%%%%%%%%%

% Samples of sectioning (and labeling) in OPRE
% NOTE: (1) \section and \subsection do NOT end with a period
%       (2) \subsubsection and lower need end punctuation
%       (3) capitalization is as shown (title style).
%
%\section{Introduction.}\label{intro} %%1.
%\subsection{Duality and the Classical EOQ Problem.}\label{class-EOQ} %% 1.1.
%\subsection{Outline.}\label{outline1} %% 1.2.
%\subsubsection{Cyclic Schedules for the General Deterministic SMDP.}
%  \label{cyclic-schedules} %% 1.2.1
%\section{Problem Description.}\label{problemdescription} %% 2.
% Text of your paper here

\section{Introduction}

Assortment optimization is a fundamental problem in the revenue management literature. In this class of problems, the decision maker, hereafter referred to as the seller, selects an assortment (a subset of products) from the entire product universe to offer to customers. The goal is to maximize a context-specific objective, such as revenue, profit, or total market share. These problems typically rely on discrete choice models that capture the purchase probability of each product when a particular assortment is offered. Among the various choice models established in the literature, the multinomial logit (MNL) model is arguably the most prevalent due to its simplicity in computing purchase probabilities, its tractability relative to more complex models, and its strong predictive power. Over the past few decades, numerous studies have explored assortment optimization under the MNL model, both in the unconstrained setting and with various practical constraints.

In the context of assortment optimization, we refer to {\em covering constraints} as conditions that require the selection of at least  certain numbers of products from different, possibly overlapping, product categories. These constraints are common in today’s retail industry but remain relatively unexplored in the academic literature. For example, online retailers may need to adhere to service-level agreements with suppliers, mandating that a minimum number of products from specific brands or suppliers be offered to customers. {External regulations may also mandate the seller to guarantee a certain number of products from specific categories to be offered to customers. For example, according to a recent news article\footnote{https://variety.com/2022/digital/global/netflix-30-europe-content-quota-avms-1235286587}, Netflix is required by the European Commission to include a minimum amount of local original productions in their content list. Another example is that in 2023, the French government launched the ``anti-inflation basket" initiative, pushing retailers to offer at least a certain number of everyday products in reduced prices.\footnote{https://www.reuters.com/world/europe/french-government-says-has-deal-anti-inflation-shopping-basket-2023-03-06/} Another common reason of introducing covering constraints is to ensure product diversity, where retailers guarantee a minimum number of products with particular features (e.g., price range, color, size) to be included in their offerings.} In essence, given different product categories, covering constraints guarantee that each product category is sufficiently represented in the assortment, making the assortment more diverse and better able to meet customer preferences for products with certain attributes. This helps to engage a wider market segment. To achieve these objectives, sellers may impose covering constraints based on either product brands or features, ensuring that a minimum number of products from each category is {offered} to customers.

Although covering constraints are common in business practice, the assortment optimization problem with such constraints remains largely unexplored in the literature. Early works in this area, such as \cite{barre2023assortment} and \cite{lu2023simple}, focus on assortment customization across multiple customers, introducing covering constraints at the level of individual products. These studies require each product to be visible to a certain number of customers, but they do not address the more general scenario where covering constraints apply to product categories, each containing multiple products. When the product categories are disjoint, covering constraints are totally unimodular, allowing the problem to be solved in polynomial time, as shown in \cite{sumida2021revenue}. However, in practice, categories often overlap, with a single product belonging to multiple categories (for example, an iPhone can be part of both the smartphone category and the Apple product category). In such cases, covering constraints are not necessarily totally unimodular, making it more challenging to solve the assortment optimization problem with these general constraints.

\subsection{Main Contributions}

We consider assortment optimization under the MNL model with covering constraints. The problem can be classified along two dimensions: the first dimension is deterministic versus randomized, and the second is single-segment versus multi-segment. In the deterministic setting, the seller must offer a single assortment deterministically, whereas in the randomized setting, the seller can randomize over multiple assortments. In the single-segment setting, all customers belong to the same segment and their choices follow the same MNL model. In contrast, in the multi-segment setting, customers are divided into multiple segments, each following a potentially different MNL model, and the seller customizes the assortments offered to each segment. In what follows, we will provide a precise definition of each class of these problems and present our respective contributions for each.

\vspace{3mm}

\noindent\textbf{Single-Segment Problem.} We begin with the single-segment assortment optimization problem with covering constraints, addressing both the deterministic and randomized versions. The main results of the single-segment problem are summarized as follows:

% \noindent\underline{\textit{Deterministic single-segment problem:}} In \cref{sec:deterministic_setting}, we consider the deterministic single-segment assortment optimization problem under the MNL model with covering constraints. In this setting, the seller is given multiple categories of products, each associated with a minimum threshold given exogenously. The seller has to decide an assortment such that the number of products in the assortment from each category exceeds the minimum threshold of the corresponding category. We first prove that it is NP-hard to approximate the problem within a factor of $(1+\epsilon)/\log K$ for any $\epsilon>0$ by a reduction from the minimum set cover problem, where $K$ is the number of categories. Then we introduce an approximation algorithm for the problem, which is also our main technical contribution of the paper. In our algorithm, we first find an assortment that approximately minimizes the sum of preference weights subject to the covering constraints. This is achieved by approximately solving a weighted set cover problem using greedy algorithm. Then we find the optimal expansion of the assortment obtained from the first step. We prove that the proposed algorithm is a $1/(\log K+2)$-approximation to the problem. The approximation ratio we obtained matches the hardness results up to a constant, and is asymptotically tight for large $K$.

\noindent\underline{\textit{Deterministic Single-Segment Problem:}} In \cref{sec:deterministic_setting}, we consider the deterministic single-segment assortment optimization problem under the MNL model with covering constraints. In this setting, the seller is given multiple product categories, each with an exogenously specified minimum threshold. The seller must select an assortment such that the number of products from each category in the assortment  exceeds the minimum threshold for that category. We first show that it is NP-hard to approximate the problem within a factor of \((1+\epsilon)/\log K\) for any \(\epsilon > 0\), using a reduction from the minimum set cover problem, where \(K\) is the number of categories (Theorem \ref{thm:hardness_single_segment}). Our main technical contribution is the design of a novel algorithm that provides a \(1/(\log K + 2)\)-approximation to this problem (Theorem \ref{thm:single_segment}). Therefore, our  approximation ratio matches the hardness result up to a constant, and is asymptotically tight for large \(K\).  Our algorithm consists of two steps. Given an MNL model in which each product is associated with a preference weight, the first step is to find an assortment that approximately minimizes the sum of the preference weights, subject to the covering constraints. This is done by approximately solving a weighted set cover problem using a greedy algorithm. The second step is to determine the \textit{optimal expansion} of the assortment obtained in the first step, i.e., an assortment that maximizes the total expected revenue and includes the assortment from the first step. %The analysis of our algorithm is presented in Theorem \ref{thm:single_segment}.

%\noindent\underline{\textit{Randomized single-segment problem:}} In \cref{sec:randomized_setting}, we consider the randomized single-segment assortment optimization problem with covering constraints. In this setting, the seller decides a distribution over assortments, and the covering constraints need to be satisfied in expectation. In other words, the expected number of products from each category that is visible to customers should exceed the minimum threshold of the category. We show that the randomized single-segment assortment optimization can be solved in polynomial time by solving an equivalent linear program. To show this, we first prove that the optimal solution to the problem is supported on a sequence of nested assortments. Then we establish a representation of product visibility using purchase probabilities of products when the distribution of assortments is supported on a sequence of nested assortments. Based on the results above, we can reformulate the covering constraints as constraints on purchase probabilities and rewrite the problem as a linear program with purchase probabilities as decision variables. We further show that the optimal solution in the randomized setting randomizes over at most $\min\{K+1,n\}$ assortments, and that the ratio of optimal revenue between the randomized and deterministic settings can be arbitrarily large in the worst case.

\noindent\underline{\textit{Randomized Single-Segment Problem:}} In \cref{sec:randomized_setting}, we address the randomized single-segment assortment optimization problem with covering constraints. In this case, the seller selects a distribution over assortments, and the covering constraints must be satisfied in expectation. That is, the expected number of products from each category {  offered} to customers must exceed the minimum threshold of the category. We demonstrate that the randomized single-segment assortment optimization problem can be solved in polynomial time by solving a novel linear program (Theorem~\ref{thm:lp_randomized_single_segment}). To establish our linear program formulation, we first demonstrate that the optimal solution is supported by a sequence of nested assortments. We then show that the covering constraints can be reformulated as linear constraints on purchase probabilities, allowing us to rewrite the problem as a linear program with purchase probabilities as decision variables. 
Furthermore, while the randomized assortment could theoretically randomize over an exponential number of assortments, we prove that there exists an optimal solution in the randomized setting that involves randomizing over at most  \(\min\{K+1,n\}\) assortments. This makes the randomization more practical, as it requires a relatively small number of assortments.
Finally, we show that, in the worst case, the optimal expected revenue in the randomized setting can exceed that of the deterministic setting by an arbitrarily large ratio. However, the practical gap between the two is much smaller, as demonstrated in the numerical experiments discussed in Section \ref{sec:numerical}.

\vspace{3mm}

\noindent\textbf{Multi-Segment Problem.} We further extend our results to the case where we have multiple customer segments and can personalize assortments for each segment, while verifying the covering constraints. Specifically, we assume there are $m$ customer segments, each associated with an arrival probability and a potentially different MNL choice model. The covering constraints are introduced in expectation, based on the arrival probabilities of customer segments. That is, the expected number of products from each category {  offered} to customers must exceed the category's minimum threshold, where the expectation is taken over the arrival probabilities of the customer segments.

In \cref{sec:multi_segment}, we study the deterministic setting for the multi-segment problem. We find that the complexity of the deterministic multi-segment problem depends on the number of customer segments $m$, and thus we distinguish between two cases: the case of constant (small) $m$ and the case of general (large) $m$.

\noindent\underline{\textit{Constant number of customer segments:}} 
For a constant number of customer segments, we introduce a $(1-\epsilon)/(\log K+2)$-approximation algorithm with runtime polynomial in the input parameters for constant $m$. Similar to our algorithm for the deterministic single-segment problem, our approach consists of the following two steps. First, we find a set of assortments for all customer segments that minimizes the {\em customer-segment-weighted sum} of preference weights, subject to covering constraints. This is achieved by approximately solving a weighted set cover problem across all customer segments using a greedy algorithm. Then, we determine the optimal expansion of the assortments obtained in the first step for each customer segment. We show that, with appropriate weights across different customer segments, the algorithm yields a $1/(\log K+2)$-approximation to the problem. 
Since the appropriate weights for each customer segment depend on the true optimal assortments, which are unknown, we use a grid search to guess the appropriate weights. This results in a $(1-\epsilon)/(\log K+2)$-approximation algorithm with a runtime that is polynomial in the number of products and categories, and in $1/\epsilon$, but exponential in $m$ (Theorem \ref{thm:multi-segement-constant-m}). The approximation ratio we obtain in the constant $m$ case also matches the hardness result for the single-segment problem up to a constant.

\noindent\underline{\textit{General number of customer segments:}} 
The runtime of our algorithm in the previous setting (Theorem \ref{thm:multi-segement-constant-m}) remains exponential in $m$, making it suitable for small $m$ but impractical for general $m$. To overcome this limitation, we examine the complexity of the multi-segment assortment problem for arbitrary $m$ and investigate approximation algorithms with runtimes that are polynomial in both in $m$ and the other problem inputs. However, for general $m$, we prove that the problem is NP-hard to approximate within a factor of $\Omega(1/m^{1-\epsilon})$ for any $\epsilon > 0$ (Theorem \ref{thm:hardness_multi_segment}), using a reduction from the maximum independent set problem. We then propose a straightforward algorithm that achieves a $1/m(\log K + 2)$-approximation (Theorem \ref{thm:algo-multi-segement-general-m}). This approximation ratio matches the hardness result up to a logarithmic factor of $K$ in the general case.

\noindent\underline{\textit{Randomized multi-segment problem:}} 
The randomized setting of the multi-segment problem is analogous to that of the single-segment problem. As in the randomized single-segment problem, we can solve the randomized multi-segment problem in polynomial time by formulating it as an equivalent linear program. Specifically, we use the same representation of {  offer probabilities of products} via purchase probabilities, introduced in the randomized single-segment problem, and reformulate the multi-segment problem as a linear program where the decision variables are the purchase probabilities for each customer segment. The detailed analysis of the multi-segment problem in the randomized setting is provided in Appendix \ref{sec:multi_segment_randomized}.

In summary, we classify the problems across two dimensions: deterministic versus randomized, and single-segment versus multi-segment. We provide approximation algorithms and establish hardness results for the deterministic problems, and show that the randomized problems are polynomially solvable. A summary of our results is presented in \cref{table:summary_results}.

%In summary, we classify the problems we consider in two dimensions, deterministic versus randomized, and single-segment versus multi-segment. In the deterministic setting, we provide an $1/(\log K+2)$ approximation algorithm for the single-segment problem, and prove that the single-segment problem is NP-hard to approximate within a factor of $(1+\epsilon)/\log K$ for any $\epsilon>0$; we provide an $(1-\epsilon)/(\log K+2)$ approximation algorithm for the multi-segment problem with constant $m$; we prove that the multi-segment problem with general $m$ is NP-hard to approximate within a factor of $\Omega(1/m^{1-\epsilon})$, and we provide a $1/m(\log K+2)$ approximation algorithm for the problem. In the randomized setting, both the single-segment problem and the multi-segment problem can be solved in polynomial time by solving an equivalent linear program. A summary of our algorithm results is provided in \cref{table:summary_results}.

\begin{table}[!ht]
\centering
\begin{tabular}{|c|c|c|c|c|}
\hline
\multicolumn{2}{|c|}{}&\multicolumn{2}{c|}{Deterministic}&\multirow{2}{*}{Randomized}\\
\cline{3-4}
\multicolumn{2}{|c|}{}&Approximation&Hardness
&~\\
\hline
\multicolumn{2}{|c|}{Single-Segment}&$1/(\log K+2)$&$(1+\epsilon)/\log K$&\multirow{3}{*}{\makecell{Polynomial-time\\solvable}}\\
\cline{1-4}
\multirow{2}{*}{\makecell{Multi-\\Segment}}&Constant $m$&$(1-\epsilon)/(\log K+2)$&$(1+\epsilon)/\log K$&~\\
\cline{2-4}
~&General $m$&$1/m(\log K+2)$&$\Omega(1/m^{1-\epsilon})$&~\\
\hline
\end{tabular}
\caption{Summary of our algorithmic results}
\label{table:summary_results}
\end{table}

\noindent\textbf{Numerical Studies.} In \cref{sec:numerical}, we conduct a numerical study using real data to address the following questions: (i) What is the impact on the seller's revenue due to the introduction of covering constraints in practice? (ii) What is the revenue gap between the deterministic and randomized settings? (iii) How many assortments does the optimal solution randomize over in the randomized setting?
We use the E-commerce dataset \cite{ECommerceData}, which contains transaction records from April 2020 to November 2020 from a large online home appliances and electronics store. Our numerical experiments are performed across $13$ product types. For each product type, we calibrate an MNL model based on the offered assortments and purchase records. Products are classified into categories by price and brand, with a uniform minimum number of representatives enforced across all categories. 
We solve both the deterministic and randomized versions of the problem under the MNL model, using preference weights calibrated from the data and applying the minimum representative constraints for each price and brand category. The results show that introducing covering constraints leads to only a minimal revenue loss compared to the unconstrained optimum in both deterministic and randomized settings, with losses typically under $5\%$. Moreover, the optimal revenues in the deterministic and randomized settings are close in practice. In many cases, the optimal randomized solution offers a single assortment, while in most other cases, it randomizes over at most three assortments. {  We also observe that as we enforce stronger covering constraints, the total expected revenue from high revenue products decreases, while the total expected revenue from low revenue products decreases. This suggests that in general, low revene products would benefit from introducing stonger covering constraints.}

{  In Appendix~\ref{sec:numerical_synthetic}, we further conduct numerical experiments on synthetic instances to test the performance of our approximation algorithm of the deterministic single-segment problem. The results show that under the synthetic instances, our algorithm can often achieve a much better approximation ratio than its theoretical approximation ratio guarantee. We also compare the performance of Algorithm~\ref{alg:single_segment}~with two heuristics, one of which returns the union of locally revenue-ordered assortments of all categories and the unconstrained optimal assortment, and the other returns an optimal expansion of the union of locally revenue-ordered assortments. The results show that our algorithm outperforms the two heuristics under a broad range of test instances.}

\subsection{Related Literature}

Our work is related to three streams of literature: the MNL model, assortment optimization problem under the MNL model, and assortment optimization problem with visibility and fairness constraints.

\noindent\textbf{Multinomial Logit Model:} For decades, the MNL model has been arguably one of the most popular choice model in modeling customer choices. The MNL model was first introduced by~\cite{luce1959individual} and followed by~\cite{mcfadden1972conditional}, and has received attention in economics and operations research literature due to its simplicity in computing choice probabilities, its predictive power, and its significant computational tractability in corresponding decision problems. Many generalized variants of the MNL have also been established, such as the nested logit model (see \citealt{williams1977formation}, \citealt{mcfadden1980econometric}), mixture of MNL models (see \citealt{mcfadden2000mixed}), and the general attraction model (see \citealt{gallego2015general}).

\noindent\textbf{Assortment Optimization under the MNL Model:} Assortment optimization under the MNL model has been studied extensively in the revenue management literature. \cite{talluri2004revenue} are among the first to consider assortment optimization under the MNL model and they showed that the unconstrained optimal assortment under the MNL model is an assortment containing all products whose revenue is above a certain threshold, often referred to as a revenue-ordered assortment. Thus, the unconstrained assortment optimization problem under the MNL model can be solved in polynomial time by finding the revenue-ordered assortment that maximizes the expected revenue. \cite{rusmevichientong2010dynamic} established a polynomial-time algorithm that solves the assortment optimization problem under the  MNL with cardinality constraints. \cite{desir2022capacitated} further considered the assortment optimization problem under the MNL model with capacity constraints. They showed that the capacitated assortment optimization under the MNL model is NP-hard, and they provided an FPTAS for the problem. As an extension of \cite{rusmevichientong2010dynamic}, \cite{sumida2021revenue} further considered the assortment optimization under the MNL model with totally unimodular constraints. They showed that the problem can be reformulated into an equivalent linear program, thus can be solved in polynomial time. Other research works that consider assortment optimization problems under different settings of the MNL model  include \cite{mahajan2001stocking}, \cite{gao2021assortment}, \cite{el2023joint}, \cite{housni2023maximum}.

\noindent\textbf{Assortment Optimization with Visibility and Fairness Constraints:} 
In recent years, there has been growing interest in assortment optimization with fairness and visibility constraints. The works of \cite{barre2023assortment} and \cite{lu2023simple} are the most closely related to ours, as they are among the first to explore assortment optimization under visibility constraints (also referred to as fairness constraints in \cite{lu2023simple}). Both papers address an assortment optimization  aiming to ensure that each product is visible to a minimum number of customers, and therefore achieving some notion of fairness across the offered products. Specifically, \cite{barre2023assortment} study deterministic assortment optimization over a finite set of customers, requiring each product to be offered to at least a certain number of customers and the assortment can be customized to each customer. On the other hand, \cite{lu2023simple} adopt a randomized approach, ensuring that the probability of each product being visible to customers exceeds a certain threshold.
To the best of our knowledge, our paper is the first to investigate assortment optimization under the MNL model with general covering constraints. {  Our covering constraints are more general than  those in \cite{barre2023assortment} and \cite{lu2023simple}, as they capture the special case where each category corresponds to a single product.
In particular, our randomized version of the problem is a strict generalization of the baseline problem in \cite{lu2023simple}. %, while our deterministic version generalizes \cite{barre2023assortment} in the case of a single customer segment. 
Moreover, our problem is significantly more challenging than the baseline models in \cite{barre2023assortment} and \cite{lu2023simple} which are solvable in polynomial time, while we demonstrate that it is NP-hard to approximate our problem within a factor of $(1+\epsilon)/\log K$ for any $\epsilon>0$, even in the deterministic setting with a single customer. Both \cite{barre2023assortment} and \cite{lu2023simple} also consider extensions beyond their baseline models that enforces cardinality constraints over the offered products. Both extensions are shown to be NP-hard, and are not captured by our model.}

Broadly, our work also contributes to the literature on assortment optimization with fairness constraints. Another notable work in this area is \cite{chen2022fair}, which takes a different approach by defining fairness constraints as ensuring similar outcomes (e.g., revenue, market share, visibility) for products with similar features.

\section{Deterministic Single-Segment Assortment Optimization}\label{sec:deterministic_setting}

In this section, we formally introduce the deterministic single-segment assortment optimization problem with covering constraints. We study the complexity of the problem and prove that it is NP-hard to approximate within a factor of $(1+\epsilon)/\log K$ for any $\epsilon>0$, where $K$ is the number of categories. Then we establish an approximation algorithm that gives an approximation ratio of $1/(\log K+2)$, thereby matching the hardness of the problem.

Let $\mathcal{N}=\{1,2,\dots,n\}$ be the universe of products. Each product $i\in\mathcal{N}$ has a revenue $r_i$. The seller chooses an assortment, i.e., a subset of products $S\subseteq\mathcal{N}$, to offer to customers. The option of not purchasing any product is represented symbolically as product $0$, and refereed to as the no-purchase option.

We assume that  customers make choices based on the MNL model. Under this choice model, each product $i\in\mathcal{N}$ is associated with a preference weight $v_i>0$ capturing the attractiveness of the product. Without loss of generality, the preference weight of the no-purchase option is normalized to be $v_0=1$. Under the MNL model, if assortment $S$ is offered, the probability that a customer purchases product $i$ for some $i\in S$ is given by
\begin{equation*}
\phi(i,S)=\dfrac{v_i}{1+\sum_{j\in S}v_j}.
\end{equation*}
The probability of no-purchase is given by $\phi(0,S)={1}/({1+\sum_{i\in S}v_i})$. The expected revenue gained from a customer that is offered assortment $S$ is given by
\begin{equation}
R(S)=\sum_{i\in S}r_i\phi(i,S)=\dfrac{\sum_{i\in S}r_iv_i}{1+\sum_{i\in S}v_i}.\label{eq:revenue_mnl}
\end{equation}

We assume that the products in the universe are categorized into $K$ possibly overlapping categories $C_1,C_2,\dots,C_K$, where each $C_k$ is a subset of $\mathcal{N}$. Each category $C_k$ is associated with a minimum threshold $\ell_k$, which is an exogenous integer parameter between  $0$ and $ |C_k|$. The offered assortment should include at least $\ell_k$ products that belongs to category $C_k$ for each $k$. In other words, the offered assortment $S$ must satisfy $|S\cap C_k|\geq \ell_k$ for all $k\in\{1,2,\dots,K\}$. We refer to these constraints as covering constraints.

Our goal is to find an assortment that maximizes the expected revenue subject to covering constraints. We refer to the problem as the {\em Deterministic Assortment Optimization with Covering constraints}, briefly~\eqref{prob:deterministic_single_segment}, and it can be formulated as follows:

\begin{equation}
\label{prob:deterministic_single_segment}
\tag{DAOC}
\begin{aligned}
&\max_{S\subseteq\mathcal{N}} && R(S)\\
&\textup{s.t.} && |S\cap C_k|\geq \ell_k,\ \forall k\in\{1,2,\dots,K\}.\\
\end{aligned}
\end{equation}

\subsection{Complexity of \ref{prob:deterministic_single_segment}}

For some special cases, the constraints in~\eqref{prob:deterministic_single_segment}~are totally unimodular, thus \eqref{prob:deterministic_single_segment} can be solved in polynomial time via an equivalent linear program as shown in \cite{sumida2021revenue}. For example, the constraints in~\eqref{prob:deterministic_single_segment}~are totally unimodular if the categories are disjoint from each other. We further prove in Lemma~\ref{lemma:tu_constraints} in Appendix \ref{sec:tu_constraints}  that the covering constraints are still totally unimodular even if the categories can be divided into two groups, where categories from the same group are disjoint from each other. %This occurs for example when products are categorized by two features, and each category consists of all products with one particular feature. 
This occurs, for example, when products are categorized by two features, where each product has one value for the first feature and one value for the second, and each category corresponds to all products that share the same value for one of the two features.
However, in general, not all covering constraints are totally unimodular. In fact, in the next theorem, we show that it is NP-hard to approximate~\eqref{prob:deterministic_single_segment}~within a factor of $(1+\epsilon)/\log K$ for any $\epsilon>0$. 

\begin{theorem}\label{thm:hardness_single_segment}
	It is NP-hard to approximate~\eqref{prob:deterministic_single_segment}~within a factor of $(1+\epsilon)/\log K$ for any $\epsilon>0$ unless $P=NP$.
\end{theorem}

The proof of \cref{thm:hardness_single_segment}  uses  a reduction from the minimum set cover problem and is  provided in Appendix~\ref{sec:thm:hardness_single_segment}. {  As a side note, we would like to remark that due to the complexity of covering constraints, there is no clear structure for the optimal solution of~\eqref{prob:deterministic_single_segment}. It is a well known property that under the MNL model, given any assortment $S$, adding a product with revenue higher than $R(S)$ will cause the expected revenue to increase, while adding a product with revenue lower than $R(S)$ will cause the expected revenue to decrease. Therefore it is easy to see that the optimal assortment $S^*$ of \eqref{prob:deterministic_single_segment}~contains all products with revenue higher than the optimal value of \eqref{prob:deterministic_single_segment}, since adding these products will not reduce the overall expected revenue and does not violate covering constraints. Furthermore, any product in $S^*$ whose revenue is lower than the optimal value is in the optimal assortment to tightly satisfy certain covering constraints, since otherwise removing the product will increase the overall expected revenue. However, beyond these trivial properties, the structure of the optimal solution of \eqref{prob:deterministic_single_segment} remains largely unclear due to the complexity of covering constraints, presenting challenges to solving \eqref{prob:deterministic_single_segment}.}

\subsection{Approximation Algorithm for  \ref{prob:deterministic_single_segment}}

Our main technical contribution in this section is to provide a $1/(\log K+2)$-approximation algorithm for~\eqref{prob:deterministic_single_segment}, which matches the hardness of the problem given in Theorem \ref{thm:hardness_single_segment} up to a constant factor. Our algorithm consists of two steps. In the first step, we obtain an assortment $\hat{S}$ that approximately minimizes the sum of preference weights subject to the covering constraints, i.e., $\hat{S}$ is an approximate solution to the following weighted set cover problem

\begin{equation}
\label{step1}
\begin{aligned}
&\min_{S\subseteq\mathcal{N}} && \sum_{i \in S} v_i\\
&\textup{s.t.} && |S\cap C_k|\geq \ell_k,\ \forall k\in\{1,2,\dots,K\}.\\
\end{aligned}
\end{equation}
The assortment $\hat{S}$ is obtained by solving the weighted set cover problem using the standard greedy algorithm. That is, we start by setting $\hat{S}=\varnothing$. In each iteration, we add to $\hat{S}$ the product outside $\hat{S}$ that minimizes the ratio of preference weight over the number of categories it belongs to and whose covering constraints are not yet satisfied. In other words, the greedy algorithm adds the product $i$ outside $\hat S$ that minimizes $v_i/c_i$ where $c_i=|\{k\in\{1,2,\dots,K\}:\,i\in C_k,\,|C_k\cap \hat{S}|<\ell_k\}|$. Here, we define $v_i/c_i$ to be $\infty$ if $c_i=0$, thus we only add products that satisfy $c_i>0$.
The iterations terminate when $\hat{S}$ satisfies all covering constraints $|\hat{S}\cap C_k|\geq \ell_k$ for all $k\in\{1,2,\dots,K\}$. In the second step, we  expand the assortment $\hat{S}$ optimally such that the expected revenue is maximized. 
Here, an optimal expansion of $\hat{S}$ refers to an assortment  with the maximum expected revenue over all assortments that contain $\hat S$. Lemma 3.3 in \cite{barre2023assortment} shows that the optimal expansion of $\hat{S}$ is the union of $\hat{S}$ and a revenue-ordered assortment, thus the optimal expansion of $\hat{S}$ can be solved in polynomial time. The summary of our algorithm is provided in~\cref{alg:single_segment}.

\begin{algorithm}[!ht]
\SingleSpacedXI
\caption{Approximation algorithm for~\eqref{prob:deterministic_single_segment}}
\label{alg:single_segment}
\begin{algorithmic}
\State Initialize $\hat{S}\leftarrow\varnothing$
\While{there exists $k\in\{1,2,\dots,K\}$ such that $|\hat{S}\cap C_k|<\ell_k$}
\State Set $c_i=|\{k\in\{1,2,\dots,K\}:\,i\in C_k,\,|C_k\cap \hat{S}|<\ell_k\}|$ for all $i\in\mathcal{N}\backslash\hat{S}$
\State Set $i^*=\argmin_{i\in \mathcal{N}\backslash\hat{S}} v_i/c_i$, update $\hat{S}\leftarrow \hat{S}\cup\{i^*\}$
\EndWhile
\State Set $\bar{S}=\mathrm{argmax}_{S\supseteq\hat{S}}R(S)$
\State \textbf{return} $\bar{S}$
\end{algorithmic}
\end{algorithm}

{  Before proceeding to the performance guarantee of Algorithm \ref{alg:single_segment}, we would like to first provide some intuition for Algorithm \ref{alg:single_segment}. Intuitively, Algorithm \ref{alg:single_segment} first finds an assortment $\hat{S}$ that satisfies the covering constraints while having minimal impact on the overall expected revenue. Then based on $\hat{S}$, the algorithm adds more products to the assortment to maximize revenue. From the formula of expected revenue given in \eqref{eq:revenue_mnl}, one can see that the sum of preference weights appears in the denominator. A large sum of preference weights in $\hat{S}$ would result in a large denominator in \eqref{eq:revenue_mnl}, which potentially has a negative impact on the total expected revenue. On the other hand, if $\sum_{i\in\hat{S}}v_i$ is small, then the denominator in \eqref{eq:revenue_mnl} will not be excessively large, and the covering constraints will be satisfied with minimal impact on the overall expected revenue. Therefore in the first step, we find an assortment $\hat{S}$ by approximately minimizing the sum of preference weights subject to covering constraints. In the second step, we optimally expand $\hat{S}$ by adding more products to the assortment such that the expected revenue is satisfied.}

In the next theorem, we show our main result in this section where  we prove that the assortment returned by \cref{alg:single_segment} is a $1/(\log K+2)$ approximation to~\eqref{prob:deterministic_single_segment}. %Combining this result and the result in Theorem~\ref{thm:hardness_single_segment} we conclude that the approximation ratio obtained by \cref{alg:single_segment} matches the hardness results in \cref{thm:hardness_single_segment}.

\begin{theorem}\label{thm:single_segment}
Algorithm~\ref{alg:single_segment}~returns a $1/(\log K+2)$-approximation for~\eqref{prob:deterministic_single_segment}.
\end{theorem}

To prove Theorem~\ref{thm:single_segment}, we first introduce an auxiliary lemma where we show  that if the sum of preference weights of an assortment $S_1$ is less than  $\alpha$ times the sum of preference weights of an assortment $S_2$, then the expected revenue of an optimal expansion of $S_2$ is at least $1/(\alpha+1)$ times the expected revenue of $S_1$. 
We use $\bar{R}(S)$ to refer to the expected revenue for an optimal expansion of assortment $S$, i.e., 
$$\bar{R}(S)=\max_{S'\supseteq S}R(S').$$
Our result is given in the following lemma. %Lemma~\ref{lemma:optimal_expansion}
%In other words, $\bar{R}(S)$ is defined as the expected revenue for the optimal expansion of assortment $S$. 

\begin{lemma}\label{lemma:optimal_expansion}
    Let $S_1,S_2\subseteq\mathcal{N}$ such that $\sum_{i\in S_1}v_i\leq \alpha\sum_{i\in S_2}v_i$. Then
	\begin{equation*}
		\bar{R}(S_1)\geq \dfrac{1}{\alpha+1}R(S_2).
	\end{equation*}
\end{lemma}
\begin{proof}
	Let $S'$ be an optimal solution of the following knapsack problem.
	\begin{equation*}
        \begin{aligned}
		&\max_{S\subseteq\mathcal{N}} &&{\sum_{i\in S}r_iv_i}\\
		&\text{s.t.}&&{\sum_{i\in S}v_i\leq \sum_{i\in S_2}v_i.}
        \end{aligned}
	\end{equation*}
We claim that $R(S_1\cup S')\geq R(S_2)/(\alpha+1)$. In fact, since $S_2$ is also a feasible solution to the knapsack problem, we have by optimality of $S'$ that $\sum_{i\in S'}r_iv_i\geq \sum_{i\in S_2}r_iv_i$. Therefore,
	\begin{align*}
		R(S_1\cup S')=&\dfrac{\sum_{i\in S_1\cup S'}r_iv_i}{1+\sum_{S_1\cup S'}v_i}\geq\dfrac{\sum_{i\in S'}r_iv_i}{1+\sum_{i\in S_1}v_i+\sum_{i\in S'}v_i}\\\geq&\dfrac{\sum_{i\in S_2}r_iv_i}{1+(\alpha+1)\sum_{i\in S_2}v_i}\geq\dfrac{\sum_{i\in S_2}r_iv_i}{(\alpha+1)(1+\sum_{i\in S_2}v_i)}=\dfrac{1}{\alpha+1}R(S_2).
	\end{align*}
	By definition $\bar{R}(S_1)=\max_{S\supseteq S_1}R(S)$, and since $S_1 \subseteq S_1 \cup S'$, we get
	\begin{equation*}
		\bar{R}(S_1)\geq R(S_1\cup S')\geq \dfrac{1}{\alpha+1}R(S_2).\qedhere
	\end{equation*}
\end{proof}

Building on  Lemma~\ref{lemma:optimal_expansion}, we are able to complete the proof of Theorem~\ref{thm:single_segment}.

\begin{proof}[Proof of Theorem~\ref{thm:single_segment}]
Let $S^*$ be an  optimal solution to~\eqref{prob:deterministic_single_segment}, and let  $\tilde{S}$ be an optimal solution to weighted set cover problem \eqref{step1}. Since $S^*$ is also a feasible solution to Problem \eqref{step1}, we have $\sum_{i\in S^*} v_i\geq \sum_{i\in\tilde{S}}v_i$. Additionally, our approximate solution $\hat{S}$ for Problem \eqref{step1} was obtained using greedy algorithm, hence by~\cite{lovasz1975ratio}, we know that $\hat{S}$ gives a $H_K$-approximation to Problem \eqref{step1}, where $H_K$ is the harmonic sum defined as $H_K=\sum_{k=1}^K 1/k$. Note that  $\log K \leq H_K \leq \log K +1$. Therefore
\begin{equation*}
\sum_{i\in\hat{S}}v_i\leq H_K\cdot \sum_{i\in \tilde{S}}v_i\leq H_K\cdot \sum_{i\in S^*}v_i.
\end{equation*}
By applying Lemma~\ref{lemma:optimal_expansion} using $\hat{S}$, $S^*$ and $\alpha= H_K$, we obtain
$    \bar{R}(\hat{S}) \geq   {R}(S^*)  /(H_K+1)  $. Recall $\bar{S}$ is the assortment returned by Algorithm \ref{alg:single_segment}, therefore
\begin{equation*}
R(\bar{S}) = \bar{R}(\hat{S})    \geq \dfrac{1}{H_K+1}R(S^*)\geq\dfrac{1}{\log K+2}\cdot R(S^*),
\end{equation*}
thus $\bar{S}$ is a $1/(\log K+2)$-approximation of~\eqref{prob:deterministic_single_segment}.
\end{proof}

{  
\begin{remark}[Extensions to Cardinality Constraint]
In Appendix~\ref{sec:daoc_cardinality}, we further consider an extension of \eqref{prob:deterministic_single_segment}, where besides covering constraints, the assortment should also satisfy a cardinality constraint. For this problem, we establish a bicriteria approximation algorithm that returns an assortment that satisfies the covering constraints, attains an expected revenue at least $\Omega(1/\log K)$ fraction of the optimal expected revenue, and satisfies the covering constraints up to a factor of $O(\log K)$.
\end{remark}
}

\section{Randomized Single-Segment Assortment Optimization}\label{sec:randomized_setting}

In this section, we consider the randomized single-segment assortment optimization problem under the MNL model with covering constraints. In this setting, customer choices still follow the MNL model. From the seller's side, instead of deterministically offering a single assortment to all customers as in \cref{sec:deterministic_setting}, in the randomized setting the seller decides a distribution over assortments $q(\cdot)$. Whenever a customer comes to the seller, the seller randomly draws an assortment $S$ with probability  $q(S)$ and offers the assortment $S$ to the customer. Similar to the deterministic setting, the products are categorized into  categories $C_1,C_2,\dots,C_K$, where $C_k\subseteq\mathcal{N}$ for all $k\in\{1,2,\dots,K\}$. Each category $C_k$ is associated with a minimum threshold $\ell_k$. The covering constraints in the randomized setting are introduced in expectation. In particular, we require that the expected number of offered products that belong to $C_k$ under the probability mass function $q(\cdot)$ should be at least $\ell_k$. In other words, the probability mass function $q(\cdot)$ must satisfy $\sum_{S\subseteq\mathcal{N}}|S\cap C_k|q(S)\geq \ell_k$ for all $k\in\{1,2,\dots,K\}$. We refer to this problem as the \textit{Randomized Assortment Optimization with Covering constraints}, briefly~\eqref{prob:randomized_single_segment}, and its formal definition is provided by

\begin{equation}
\label{prob:randomized_single_segment}
\tag{RAOC}
\begin{aligned}
& \max_{q(\cdot)} && \sum_{S\subseteq\mathcal{N}} R(S)q(S) \\
& \text{s.t.} && \sum_{S\subseteq\mathcal{N}}|S\cap C_k|q(S)\geq \ell_k,\ \forall k\in\{1,2,\dots,K\}, \\
&&& \sum_{S\subseteq\mathcal{N}}q(S)=1,\\
&&& q(S)\geq 0,\ \forall S\subseteq\mathcal{N}. \\
\end{aligned}
\end{equation}

Here, $R(S)$ is the expected revenue by offering assortment $S$ given by~\eqref{eq:revenue_mnl}, and the decision variable $q(\cdot)$ is a probability mass function over assortments. The validity of $q(\cdot)$ as a probability mass function is guaranteed by the last two constraints of~\eqref{prob:randomized_single_segment}.
Furthermore, the covering constraints in~\eqref{prob:randomized_single_segment} ensure that for each $k$,  the expected number of products in $C_k$  offered to the customer, which is given  by $\sum_{S\subseteq\mathcal{N}}|S\cap C_k|q(S)$, should be greater or equal than the  threshold $\ell_k$. % is thus the first set of constraints in~\eqref{prob:randomized_single_segment}~means the covering constraints are satisfied in expectation. 

Note that \eqref{prob:randomized_single_segment} contains an exponential number of variables. Our main contribution in this section is to show that, surprisingly, \eqref{prob:randomized_single_segment}~can be solved in polynomial time by solving an equivalent compact linear program that has a polynomial number of variables and constraints. We present our  efficient linear program in Section \ref{sec:LP-random}. Furthermore, we study the value of using randomized assortments as opposed to deterministic assortments in Section \ref{sec:value_randomized}.

\subsection{Linear Program for \ref{prob:randomized_single_segment}} \label{sec:LP-random}

In this section, we show that \eqref{prob:randomized_single_segment}~can be solved in polynomial time by solving an equivalent efficient linear program. Our linear program is given by:
\begin{equation}\label{prob:lp_randomized}\tag{RAOC-LP}
\begin{aligned}
&\max_{\mathbf{x},\mathbf{y}}&&{\sum_{i\in\mathcal{N}}r_iv_ix_i}\\
&\textup{s.t.}&&{x_0+\sum_{i\in\mathcal{N}}v_ix_i=1,}\\
&&&{x_i\leq x_0,\ \forall i\in\mathcal{N},}\\
&&&{y_{ij}\leq x_i,\ y_{ij}\leq x_j,\ \forall i,j\in\mathcal{N},}\\
&&&{\sum_{i\in C_k}\left(x_i+\sum_{j\in\mathcal{N}}v_jy_{ij}\right)\geq \ell_k,\ \forall k\in\{1,2,\dots,K\}.}
\end{aligned}
\end{equation}

In our formulation above, $x_0$ can be interpreted as the probability of no-purchase, and $x_i$ as the ratio of purchase probability of product $i$ over $v_i$. Then the purchase probability of product $i$ is $v_ix_i$. Summing over all purchase probabilities gives the first constraint $x_0+\sum_{i\in\mathcal{N}}v_ix_i=1$. Furthermore, since under the MNL model, in any assortment the purchase probability of product $i$ is at most $v_i$ times the no-purchase probability, the second set of constraints $x_i\leq x_0$ for all $i\in\mathcal{N}$ is valid under the MNL model. We will show later in this section
%in the forthcoming Lemma~\ref{lemma:visibility_representation} 
that the last two sets of constraints imply that the distribution over offered assortments satisfies the covering constraints in our problem. We would like to note that a closely related linear program formulation was established  in \cite{cao2023revenue} for another problem which considers assortment optimization and network revenue management under mixture of MNL and independent demand models. Specifically,  similar to~\eqref{prob:lp_randomized}, \cite{cao2023revenue}  also use $x_i+\sum_{j\in\mathcal{N}}v_jy_{ij}$  to capture {  whether} product $i$ {  is} offered for each $i\in\mathcal{N}$ {  when the seller deterministically offers a single assortment}, where $y_{ij}= \min\{x_i,x_j\}$ for all $i,j \in \mathcal{N}$.

%formulate the revenue from the independent demand segment as a linear combination of probabilities of each product being offered, and similar to~\eqref{prob:lp_randomized}, they also use $x_i+\sum_{j\in\mathcal{N}}v_jy_{ij}$ to capture the probability of product $i$ being offered for each $i\in\mathcal{N}$.

%{\color{blue} In fact, the separation subproblem takes the form $\max_{S\subseteq\mathcal{N}}R(S)+\sum_{k=1}^K\gamma_k|S\cap C_k|$, which can be solved in polynomial time using the algorithm provided in~\cite{cao2023revenue}.}

%Consider an optimal solution $(\mathbf{x}^*,\mathbf{y}^*)$ of \eqref{prob:lp_randomized}. Consider $\{i_1,\dots,i_n\}$  the permutation of $\{1,\dots,n\}$ such that $x^*_{i_1}\geq x^*_{i_2}\geq\dots x^*_{i_n}$. For  $p \in \{1, \ldots, n\}$, let  $S_p$  denote the set  $\{i_1,i_2,\dots,i_p\}$. We construct a  distribution over assortments as follows:

Consider an optimal solution $(\mathbf{x}^*,\mathbf{y}^*)$ of \eqref{prob:lp_randomized}. Suppose without loss of generality that  $x^*_1 \geq x^*_{2}\geq\dots x^*_{n}$. For  $p \in \{1, \ldots, n\}$, let  $S_p$  denote the set  $\{1,2,\dots,p\}$. We construct a  distribution over assortments as follows:

\begin{equation}
q^*(S)=\begin{cases}
\left(1+\sum_{j=1}^p v_{j}\right)(x^*_{p}-x^*_{{p+1}}),\ &\text{if}\ S=S_p\ \text{for some}\ p\in\{1,2,\dots,n\},\\
1-\sum_{p=1}^n q^*(S_p),\ &\text{if}\ S=\varnothing,\\
0,\ &\text{otherwise},
\end{cases}\label{eq:solution_recover}
\end{equation}
with the notation $x^*_{n+1}=0$. % The construction of the distribution of assortments $q^*(\cdot)$ using $\mathbf{x}^*$ follows from Theorem 1 of \cite{topaloglu2013joint}, where the authord present the in general how to construct the a distribtuion of assrtoemtn if we know the purchase probailityes.
The construction of the distribution of assortments \( q^*(\cdot) \) using \( \mathbf{x}^* \) follows from Theorem 1 of \cite{topaloglu2013joint}, where the author {  presents}, in general, how to construct a distribution of assortments if we know the purchase probabilities. {  Similar construction of distribution of assortments is also established in Section 4.2 of \cite{lei2022joint}.}

In the next theorem, we show that we can obtain an optimal solution to~\eqref{prob:randomized_single_segment} using an optimal solution to~\eqref{prob:lp_randomized}.

\begin{theorem}\label{thm:lp_randomized_single_segment}
For an optimal solution $(\mathbf{x}^*,\mathbf{y}^*)$ to~\eqref{prob:lp_randomized}, define $q^*(\cdot)$ using~\eqref{eq:solution_recover}, then $q^*(\cdot)$ is an optimal solution to~\eqref{prob:randomized_single_segment}.
\end{theorem}
To prove Theorem~\ref{thm:lp_randomized_single_segment}, we introduce two lemmas. First, in Lemma~\ref{lemma:nested_support}, we prove that there exists an optimal solution to~\eqref{prob:randomized_single_segment}~that is supported on a sequence of nested assortments.

\begin{lemma}\label{lemma:nested_support}
There exists an optimal solution $q(\cdot)$ to~\eqref{prob:randomized_single_segment}~and a sequence of assortments $\{S_i\}_{i=1}^m$ for some $m$, such that $S_1\supseteq S_2\supseteq\dots\supseteq S_m$ and $q(S)=0$ for all $S\notin \{S_i\}_{i=1}^m$.
\end{lemma}

\begin{proof}
By introducing new variables $p_i$ for all $i \in {\cal N}$,    we rewrite~\eqref{prob:randomized_single_segment} as follows
    \begin{equation}\label{prob:randomized_single_segment_2}
    \begin{aligned}
    &\max_{q(\cdot),\mathbf{p}}&&{\sum_{S\subseteq\mathcal{N}}R(S)q(S)}\\
    &\text{s.t.}&&{\sum_{S\ni i}q(S)\geq p_i,\ \forall i\in\mathcal{N},}\\
    &&&{\sum_{i\in C_k}p_i\geq \ell_k,\ \forall k\in\{1,2,\dots,K\},}\\
    &&&{\sum_{S\subseteq\mathcal{N}}q(S)=1,}\\
    &&&{q(S)\geq 0,\ \forall S\subseteq\mathcal{N}.}
    \end{aligned}
    \end{equation}
    Let $(q^*(\cdot),\mathbf{p}^*)$ be an optimal solution to~\eqref{prob:randomized_single_segment_2}. By fixing $\mathbf{p}$ in~\eqref{prob:randomized_single_segment_2}~to  $\mathbf{p}^*$ and only optimizing over $q(\cdot)$, we derive the following problem
    \begin{equation}\label{prob:randomized_single_segment_3}
    \begin{aligned}
    &\max_{q(\cdot)}&&{\sum_{S\subseteq\mathcal{N}}R(S)q(S)}\\
    &\text{s.t.}&&{\sum_{S\ni i}q(S)\geq p_i^*,\ \forall i\in\mathcal{N},}\\
    &&&{\sum_{S\subseteq\mathcal{N}}q(S)=1,}\\
    &&&{q(S)\geq 0,\ \forall S\subseteq\mathcal{N}.}
    \end{aligned}
    \end{equation}
    For any optimal solution $\hat{q}(\cdot)$ of~\eqref{prob:randomized_single_segment_3}, $(\hat{q}(\cdot),\mathbf{p}^*)$ is an optimal solution to~\eqref{prob:randomized_single_segment_2}, thus $\hat{q}(\cdot)$ is also an optimal solution to~\eqref{prob:randomized_single_segment}.

    Problem \eqref{prob:randomized_single_segment_3} is similar to the fair assortment optimization problem considered in~\cite{lu2023simple}. By Corollary 1 of~\cite{lu2023simple}, there exists an optimal solution $\hat{q}(\cdot)$ to~\eqref{prob:randomized_single_segment_3} such that all assortments offered with a positive probability are nested. In other words, there exists an optimal solution $q(\cdot)$ to~\eqref{prob:randomized_single_segment}~and a sequence of assortments $\{S_i\}_{i=1}^m$ for some $m$, such that $S_1\supseteq S_2\supseteq\dots\supseteq S_m$ and $q(S)=0$ for all $S\notin \{S_i\}_{i=1}^m$.
\end{proof}

Our second lemma  is given in Lemma~\ref{lemma:visibility_representation}, where we show that if the assortments offered with positive probability are nested, then we can derive a nice algebraic expression for the  probability of each product being included in the offered assortments. Specifically, we assume that the seller randomizes only over a nested sequence of assortments $S_1\supseteq S_2\supseteq\dots \supseteq S_m$ for some $m$, and offers assortment $S_\ell$ with probability $q_\ell$ for each $\ell\in\{1,2,\dots,m\}$. We define $x_i$ as the ratio of purchase probability of product $i$ over preference weight $v_i$, given by
\begin{equation}
	x_i=\sum_{\ell=1}^mq_\ell\dfrac{\bm{1}[i\in S_\ell]}{1+\sum_{j\in S_\ell}v_j}.\label{eq:purchase_probability_ratio}
\end{equation}
Here, we use $\bm{1}[i\in S]$ to denote the indicator function of $i\in S$, which takes value $1$ if $i\in S$ and takes value $0$ if $i\notin S$. We show that for each $i\in\mathcal{N}$, the probability that product $i$ is included in the offered assortment, i.e.,  \mbox{$\sum_{\ell=1}^m q_\ell\cdot \bm{1}[i\in S_\ell]$}, can be expressed nicely using $\{x_i\}_{i\in\mathcal{N}}$. We  would like to note that a similar expression also appears in the proof of Theorem 6.1 in \cite{cao2023revenue}. 
%Similar result can also be found in the proof of Theorem 6.1 in \cite{cao2023revenue}.}

\begin{lemma}\label{lemma:visibility_representation}
	Consider a sequence of nested  assortments $S_1\supseteq S_2\supseteq\dots\supseteq S_m$, and a probability $q_\ell$ of offering assortment $S_\ell$ for each $\ell\in\{1,2,\dots,m\}$. For $i\in \mathcal{N}$, let  $x_i$ be defined as in~\eqref{eq:purchase_probability_ratio}. Then for any $i\in\mathcal{N}$,
	\begin{equation}
		\sum_{\ell=1}^m q_\ell\cdot \bm{1}[i\in S_\ell]=x_i+\sum_{j\in\mathcal{N}}v_j\min\{x_i,x_j\}.\label{eq:visibility_rep}
	\end{equation}
\end{lemma}

\begin{proof}
	%Define $n=|\mathcal{N}|$, and 
 For simplicity, assume without loss of generality that $m=n$ (if $m<n$ we could complement the sequence with additional assortments with zero probability). We further assume that \mbox{$S_\ell=\{\ell,\ell+1,\dots,n\}$} for all $\ell\in\{1,2,\dots,n\}$. In this case, $i\in S_\ell$ if and only if $i\geq \ell$. Then we get
    \begin{equation*}
    x_i=\sum_{\ell=1}^i \dfrac{q_\ell}{1+\sum_{j\in S_\ell}v_j},
    \end{equation*}
    which also implies that $x_1\leq x_2\leq\dots\leq x_n$. In this case, the left hand side of \eqref{eq:visibility_rep} equals to $\sum_{\ell=1}^i q_\ell$, and the right hand side of \eqref{eq:visibility_rep} equals to
    \begin{align*}
    x_i+\sum_{j\in\mathcal{N}}v_j\min\{x_i,x_j\}=&x_i+\sum_{j=1}^i v_j \min\{x_i,x_j\}+\sum_{j=i+1}^n v_j \min\{x_i,x_j\}\\
    =&x_i+\sum_{j=1}^i v_jx_j+\sum_{j=i+1}^n v_jx_i=\sum_{j=1}^i v_jx_j+\left(1+\sum_{j=i+1}^n v_j\right)x_i.
    \end{align*}
    Therefore it suffices to verify that for all $i\in \{1,2,\dots,n\}$,
	\begin{equation*}
		\sum_{\ell=1}^i q_\ell=\sum_{j=1}^i v_jx_j+\left(1+\sum_{j=i+1}^n v_j\right)x_i.
	\end{equation*}
	We have
	\begin{align*}
		\sum_{j=1}^i v_jx_j+\left(1+\sum_{j=i+1}^n v_j\right)x_i=&\sum_{j=1}^i\sum_{\ell=1}^j \dfrac{v_j q_\ell}{1+\sum_{p=\ell}^n v_p} 
  + \sum_{\ell=1}^i q_\ell\dfrac{1+\sum_{j=i+1}^nv_j}{1+\sum_{p=\ell}^n v_p}         \\
=& \sum_{\ell=1}^i\sum_{j=\ell}^i \dfrac{v_j q_\ell}{1+\sum_{p=\ell}^n v_p} 
  + \sum_{\ell=1}^i q_\ell\dfrac{1+\sum_{j=i+1}^nv_j}{1+\sum_{p=\ell}^n v_p}         \\
  =& \sum_{\ell=1}^i q_\ell  \dfrac{\sum_{j=\ell}^i v_j + 1+\sum_{j=i+1}^nv_j }{1+\sum_{p=\ell}^n v_p} = \sum_{\ell=1}^i q_\ell.
%  
%  
	%	=& \sum_{j=1}^i\sum_{\ell=j}^iq_j\dfrac{v_\ell}{1+\sum_{p=j}^n v_p}
% + \sum_{j=1}^i q_j\dfrac{1+\sum_{\ell=i+1}^nv_\ell}{1+\sum_{\ell=j}^n v_\ell}\\
		% =&  \sum_{ell=1}^i q_j\dfrac{\sum_{\ell=j}^i v_\ell}{1+\sum_{p=j}^n v_p}
  % +
  % \sum_{j=1}^i q_j\dfrac{1+\sum_{\ell=i+1}^nv_\ell}{1+\sum_{\ell=j}^n v_j}\\
		% =&\sum_{j=1}^i q_j\dfrac{1+\sum_{\ell=j}^n v_\ell}{1+\sum_{\ell=j}^n v_\ell}=\sum_{j=1}^i q_j.
	\end{align*}
    Thus we conclude that the equation in the lemma holds.
\end{proof}

%In Lemma~\ref{lemma:visibility_representation}, $x_i$ is the ratio of purchase probability over preference weight of product $i$. Lemma~\ref{lemma:visibility_representation}~states that if $q(\cdot)$ is supported on a sequence of nested assortments, then the visibility of product $i$ is given by $x_i+\sum_{j\in\mathcal{N}}v_j\min\{x_i,x_j\}$. Therefore the last set of constraints in~\eqref{prob:lp_randomized}~is equivalent to the covering constraints, \textit{i.e.}, the first set of constraints in~\eqref{prob:randomized_single_segment}.

Building  on Lemmas~\ref{lemma:nested_support}~and~\ref{lemma:visibility_representation}, we are ready to give the proof of Theorem~\ref{thm:lp_randomized_single_segment}.

\begin{proof}[Proof of Theorem~\ref{thm:lp_randomized_single_segment}]
We prove the theorem in two steps: first we prove that the optimal objective value of~\eqref{prob:randomized_single_segment}~is smaller or equal to that of~\eqref{prob:lp_randomized}, then we prove that the optimal objective value of~\eqref{prob:randomized_single_segment}~is greater or equal to that of~\eqref{prob:lp_randomized}.

\underline{\textbf{Step 1.} \eqref{prob:randomized_single_segment}$\leq$\eqref{prob:lp_randomized}:} By Lemma~\ref{lemma:nested_support}, there exists an optimal solution $q^*(\cdot)$ to~\eqref{prob:randomized_single_segment}~that is supported on a sequence of nested assortments $S_1\supseteq S_2\supseteq\dots\supseteq S_m$ for some $m$. We define $q_\ell=q^*(S_\ell)$ for each $\ell\in\{1,2,\dots,m\}$. For each $i\in\mathcal{N}$, We define $x^*_i$ using~\eqref{eq:purchase_probability_ratio} and we define $x^*_0$ as
\begin{equation*}
x^*_0=\sum_{\ell=1}^m q_\ell\dfrac{1}{1+\sum_{j\in S_\ell}v_j}.
\end{equation*}
It is clear that $x^*_i\leq x^*_0$ for all $i\in\mathcal{N}$. We further have
\begin{equation*}
x^*_0+\sum_{i\in\mathcal{N}}v_ix^*_i=\sum_{\ell=1}^m q_\ell\left(\dfrac{1}{1+\sum_{j\in S_\ell}v_j}+\dfrac{\sum_{j\in S_\ell}v_j}{1+\sum_{j\in S_{\ell}}v_j}\right)=\sum_{\ell=1}^m q_\ell=1.
\end{equation*}
By definition of $x^*_i$ we get
\begin{equation*}
\sum_{i\in\mathcal{N}}r_iv_ix^*_i=\sum_{\ell=1}^mq_\ell\dfrac{\sum_{i\in S_\ell}r_iv_i}{1+\sum_{i\in S_\ell}v_i}=\sum_{\ell=1}^m q_\ell R(S_\ell)=\sum_{S\subseteq\mathcal{N}}R(S)q^*(S).
\end{equation*}
We further set $y^*_{ij}=\min\{x^*_i,x^*_j\}$ for all $i,j\in\mathcal{N}$, then by Lemma~\ref{lemma:visibility_representation}~we have that for any $i\in\mathcal{N}$,
\begin{equation*}
\sum_{\ell=1}^m q_\ell\cdot \bm{1}[i\in S_\ell]=x^*_i+\sum_{j\in\mathcal{N}}v_j\min\{x^*_i,x^*_j\}=x^*_i+\sum_{j\in\mathcal{N}}v_jy^*_{ij}.
\end{equation*}
Therefore for any $k\in \{1,2,\dots,K\}$,
\begin{equation*}
\sum_{i\in C_k}\left(x^*_i+\sum_{j\in \mathcal{N}}v_iy^*_{ij}\right)=\sum_{i\in C_k}\sum_{\ell=1}^m q_\ell\cdot \bm{1}[i\in S_\ell]=\sum_{\ell=1}^m |S_\ell\cap C_k|q_\ell=\sum_{S\subseteq\mathcal{N}}|S\cap C_k|q^*(S)\geq \ell_k,
\end{equation*}
where the last inequality is due to the feasibility of $q^*(\cdot)$. Therefore $(\mathbf{x}^*,\mathbf{y}^*)$ constructed above is a feasible solution to~\eqref{prob:randomized_single_segment}, and $\sum_{i\in\mathcal{N}}r_iv_ix^*_i$ is equal to the optimal value of~\eqref{prob:randomized_single_segment}. Hence, the optimal value for~\eqref{prob:randomized_single_segment}~is smaller or equal to the optimal value for~\eqref{prob:lp_randomized}.

\underline{\textbf{Step 2.} \eqref{prob:randomized_single_segment}$\geq$\eqref{prob:lp_randomized}:} Let $(\mathbf{x}^*,\mathbf{y}^*)$ be an optimal solution to~\eqref{prob:lp_randomized}. Assume without loss of generality that $x^*_1\geq x^*_2\geq\dots\geq x^*_n$.  We  also assume without loss of generality that $y^*_{ij}=\min\{x^*_i,x^*_j\}$, which is the optimal choice for $y_{ij}$ given the constraints of \eqref{prob:lp_randomized}.
 We define $q^*(\cdot)$ using~\eqref{eq:solution_recover}. First, let us prove that $q^*(\cdot)$ is a probability mass function. By the assumption that $x^*_1\geq x^*_2\geq \dots\geq x^*_n$ we have $q^*(S_i)\geq 0$ for all $i\in\{1,2,\dots,n\}$. By definition of $q^*(\cdot)$ it is also obvious that $\sum_{S\subseteq\mathcal{N}}q^*(S)=1$. Then it suffices to prove that $q^*(\varnothing)\geq 0$, or equivalently $\sum_{i=1}^n q^*(S_i)\leq 1$. We have
\begin{align*}
\sum_{i=1}^n q^*(S_i)=\sum_{i=1}^n \left(1+\sum_{j=1}^i v_j\right)(x^*_i-x^*_{i+1}) &=\sum_{i=1}^n(x^*_i-x^*_{i+1})+\sum_{i=1}^n\sum_{j=1}^i x^*_iv_j-\sum_{i=1}^n\sum_{j=1}^i x^*_{i+1}v_j\\
=&x^*_1+ x^*_1v_1+ \sum_{i=2}^n\sum_{j=1}^i x^*_iv_j-\sum_{i=2}^n\sum_{j=1}^{i-1}x^*_iv_j\\
=&x^*_1+  \sum_{i=1}^n x^*_iv_i \\
\leq & x^*_0+\sum_{i=1}^n v_ix^*_i=1,
\end{align*}
where the inequality follows from the second constraint in \eqref{prob:lp_randomized} and the last equality holds from the first constraint in \eqref{prob:lp_randomized}.
Therefore $q^*(\cdot)$ is a probability mass function supported on a sequence of nested assortments \mbox{$\{\varnothing\}\cup\{S_i\}_{i\in\{1,2,\dots,n\}}$}. Next, we show that for any $i \in \cal{N}$,
%Then we prove that for any $i\in\{1,2,\dots,n\}$,
\begin{equation*}
x^*_i=\sum_{\ell=1}^n q^*(S_\ell)\dfrac{\bm{1}[i\in S_\ell]}{1+\sum_{j\in S_\ell}v_j}.
\end{equation*}
In fact, since $S_i=\{1,2,\dots,i\}$, it suffices to prove that
\begin{equation*}
x^*_i=\sum_{\ell=i}^n \dfrac{q^*(S_\ell)}{1+\sum_{j\in S_\ell}v_j},
\end{equation*}
which is true because from Equation~\eqref{eq:solution_recover}~we have
\begin{equation*}
\dfrac{q^*(S_\ell)}{1+\sum_{j\in S_\ell}v_j}=\dfrac{1+\sum_{j=1}^\ell v_j}{1+\sum_{j=1}^\ell v_j}(x^*_\ell-x^*_{\ell+1})=x^*_\ell-x^*_{\ell+1},
\end{equation*}
and by taking the sum we get
\begin{equation*}
\sum_{\ell=i}^n\dfrac{q^*(S_\ell)}{1+\sum_{j\in S_l}v_j}=\sum_{\ell=i}^n (x^*_\ell-x^*_{\ell+1})=x^*_i-x^*_{n+1}=x^*_i.
\end{equation*}
From Lemma~\ref{lemma:visibility_representation}, we have 
\begin{equation*}
\sum_{\ell=1}^n q^*(S_\ell)\cdot \bm{1}[i\in S_\ell]=x^*_i+\sum_{j\in\mathcal{N}}v_i\min\{x^*_i,x^*_j\}=x^*_i+\sum_{j\in\mathcal{N}}v_jy^*_{ij}.
\end{equation*}
Then for any $k\in\{1,2,\dots,K\}$, we get
\begin{align*}
\sum_{S\subseteq\mathcal{N}}|S\cap C_k|q^*(S)=\sum_{\ell=1}^n \sum_{i\in C_k}q^*(S_\ell)\cdot \bm{1}[i\in S_\ell]
=\sum_{i\in C_k}\left(x^*_i+\sum_{j\in\mathcal{N}}v_jy^*_{ij}\right)\geq \ell_k,
\end{align*}
where the last inequality follows from the last constraint in \eqref{prob:lp_randomized}. Therefore $q^*(\cdot)$ is a feasible solution to~\eqref{prob:randomized_single_segment}. We further have that for every $i\in\{1,2,\dots,n\}$,
\begin{equation*}
x^*_i=\sum_{\ell=i}^n (x^*_\ell-x^*_{\ell+1})=\sum_{\ell=i}^n \dfrac{q^*(S_\ell)}{1+\sum_{j\in S_\ell}v_j}=\sum_{\ell=1}^n\dfrac{q^*(S_\ell)\cdot\bm{1}[i\in S_\ell]}{1+\sum_{j\in S_\ell}v_j},
\end{equation*}
which implies that
\begin{equation*}
\sum_{i\in\mathcal{N}}r_iv_ix^*_i=\sum_{\ell=1}^n q^*(S_\ell)\dfrac{\sum_{i\in S_\ell}r_iv_i}{1+\sum_{i\in S_\ell}v_i}=\sum_{\ell=1}^n q^*(S_\ell)R(S_\ell)=\sum_{S\subseteq\mathcal{N}}R(S)q^*(S).
\end{equation*}
This concludes that the optimal value for~\eqref{prob:lp_randomized}~is smaller or equal to the optimal value for~\eqref{prob:randomized_single_segment}. Combining the two parts of the proof we have that~\eqref{prob:randomized_single_segment}~is equivalent to~\eqref{prob:lp_randomized}, and an optimal solution $q^*(\cdot)$ to~\eqref{prob:randomized_single_segment}~can be recovered from an optimal solution $(\mathbf{x}^*,\mathbf{y}^*)$ to~\eqref{prob:lp_randomized}~using~\eqref{eq:solution_recover}.
\end{proof}

{  Before we proceed, we would like to compare our results in the randomized problem with the ones in \cite{lu2023simple}. The baseline problem in \cite{lu2023simple} enforces a minimum offer probability for each individual products, and their fairness constraints can be considered as a special case of covering constraints where each category contains one single product. Both our (RAOC) and the baseline fair assortment optimization problem in \cite{lu2023simple} are polynomial time solvable, but our algorithm is fundamentally different from the one in \cite{lu2023simple}. Specifically, our (RAOC) is solved by reformulating the problem as an equivalent linear program, while \cite{lu2023simple} establishes an algorithm that directly solves their baseline problem without solving a linear program. Nevertheless, since our (RAOC) is a generalization of the baseline problem in \cite{lu2023simple}, our algorithm for (RAOC) can also solve the baseline problem in \cite{lu2023simple} exactly in polynomial time. We would also like to remark that beyond their baseline model, \cite{lu2023simple} also consider an extension that enforces cardinality constraints over the offered products. This extension is proven to be NP-hard, and is not captured by the randomized problem in our paper.}

\subsection{Value of Randomized Assortments}\label{sec:value_randomized}

It is clear that the optimal solution to~\eqref{prob:randomized_single_segment}, which randomizes over multiple assortments, achieves an expected revenue that is greater than or equal to the optimal expected revenue from~\eqref{prob:deterministic_single_segment}, since deterministically offering the optimal assortment from~\eqref{prob:deterministic_single_segment} is also feasible for~~\eqref{prob:randomized_single_segment}.
This naturally raises two questions. First, is implementing the optimal solution to~\eqref{prob:randomized_single_segment} practical, given that there could be an exponential number of assortments to randomize over? Second, how large can the gap in optimal expected revenue between~\eqref{prob:randomized_single_segment} and~\eqref{prob:deterministic_single_segment} be?

We address the first question in Proposition~\ref{prop:num_assortment}, where we show that the optimal solution to~\eqref{prob:randomized_single_segment} randomizes over at most $\min\{K+1,n\}$ assortments. In the forthcoming \cref{sec:numerical}, we further use numerical experiments to demonstrate that in practice, the optimal solution to~\eqref{prob:randomized_single_segment} offers a single assortment deterministically in many cases and randomizes over only a few assortments in most of the remaining cases of our experiments.

% It is obvious that the optimal solution of~\eqref{prob:randomized_single_segment} which randomizes over multiple assortments  achieves an expected revenue that is greater or equal to the expected revenue given by ~\eqref{prob:deterministic_single_segment}, since deterministically offering the optimal assortment of~\eqref{prob:deterministic_single_segment}~is also feasible to~\eqref{prob:randomized_single_segment}. Two further questions naturally arise. First, is implementing the optimal solution to~\eqref{prob:randomized_single_segment} practical, as there are potentially exponential number of assortments to randomize over? Second, how large can the gap in optimal expected revenue between~\eqref{prob:deterministic_single_segment}~and~\eqref{prob:randomized_single_segment}~be?

% We answer the first question in Proposition~\ref{prop:num_assortment}, where we show that the optimal solution to~\eqref{prob:randomized_single_segment}~randomizes over at most $\min\{K+1,n\}$ assortments. In the forthcoming \cref{sec:numerical}, we further use numerical experiments to show that in practice the optimal solution to~\eqref{prob:randomized_single_segment}~offers deterministically a single assortment in many cases, and only randomizes over only a few assortments in most of the other cases.

\begin{proposition}\label{prop:num_assortment}
There exists an optimal solution $q^*(\cdot)$ to~\eqref{prob:randomized_single_segment}~that offers at most $\min\{K+1,n\}$ assortments with positive probability.
\end{proposition}

\begin{proof}
The number of decision variables in the original form of~\eqref{prob:randomized_single_segment}~is 
$2^n$. There are $K+1$ equality and inequality constraints in~\eqref{prob:randomized_single_segment}~besides the constraints $q(S)\geq 0$ for all $S\subseteq\mathcal{N}$. Therefore any basic feasible solution of~\eqref{prob:randomized_single_segment}~has at least $2^n-K-1$ elements being $0$, thus has at most $K+1$ nonzero elements. By taking a basic feasible solution that is optimal to~\eqref{prob:randomized_single_segment}, we conclude that there exists an optimal solution $q^*(\cdot)$ to~\eqref{prob:randomized_single_segment}~ that has at most $K+1$ nonzero elements, and this optimal solution randomizes over at most $K+1$ assortments. 

Furthermore, by Lemma~\ref{lemma:nested_support}, there exists an optimal solution such that the assortments offered with positive probabilities are nested. This optimal solution is supported on at most $n$ assortments. Therefore we conclude that there exists an optimal solution that offers at most $\min\{K+1,n\}$ assortments with positive probability.
\end{proof}

We answer our second question with Proposition~\ref{prop:integrality_gap}, where we show that the ratio between the optimal expected revenue of~\eqref{prob:randomized_single_segment}~over that of~\eqref{prob:deterministic_single_segment}~can be arbitrarily large in the worst case. However, we would like to point out that the example in the proof of Proposition~\ref{prop:integrality_gap} is  a pathological instance. In the forthcoming numerical experiments in \cref{sec:numerical}, we numerically show that in practice the optimal revenue in the randomized and deterministic setting can be close to each other.

\begin{proposition}\label{prop:integrality_gap}
For any $M>0$, there exists an instance with only one category such that the optimal value of~\eqref{prob:randomized_single_segment}~is at least $M$ times larger than that of~\eqref{prob:deterministic_single_segment}.
\end{proposition}

\begin{proof}
Consider any $M>1$. Suppose $\mathcal{N}=\{1,2,3\}$, $K=1$, $C_1=\{1,2,3\}$, $\ell_1=2$. The revenue and preference weights are set to be
\begin{equation*}
r_1=4M^2,\ v_1=1/M,\ r_2=r_3=1/2,\ v_2=v_3=8M.
\end{equation*}
It is easy to verify that the optimal assortment to~\eqref{prob:deterministic_single_segment}~is $\{1,2\}$, and the optimal expected revenue of~\eqref{prob:deterministic_single_segment}~is
\begin{equation*}
R(\{1,2\})=\dfrac{r_1v_1+r_2v_2}{1+v_1+v_2}=\dfrac{4M+4M}{1+1/M+8M}=\dfrac{8M}{8M+1/M+1}<1.
\end{equation*}
Furthermore, $q(\{1,2,3\})=1/2$, $q(\{1\})=1/2$ is a feasible solution to~\eqref{prob:randomized_single_segment}, and the corresponding expected revenue is
\begin{equation*}
\dfrac{1}{2}R(\{1,2,3\})+\dfrac{1}{2}R(\{1\})>\dfrac{1}{2}R(\{1\})=\dfrac{1}{2}\dfrac{4M}{1+1/M}>M>M\cdot R(\{1,2\}).
\end{equation*}
Therefore in our instance, the optimal expected revenue to~\eqref{prob:randomized_single_segment} is at least $M$ times larger than the optimal expected revenue to~\eqref{prob:deterministic_single_segment}.
\end{proof}

{  
\begin{remark}[Extension to Assortment Customization over a Finite Number of Customers]
As an extension of the randomized single segment problem, another problem of potential practical interest is the assortment customization problem over $T$ homogeneous customers. In this problem, There are $T$ customers, and the choice of each customer follows the same MNL model. Instead of randomizing over assortments and satisfying covering constraints in expectation as in \eqref{prob:randomized_single_segment}, in this problem, the seller needs to deterministically customize an assortment for each of the $T$ customers, and the covering constraints should also be deterministically satisfied. We remark that a similar setting is also considered in \cite{barre2023assortment}. In Appendix~\ref{sec:assortment_customization}, we provide a rounding algorithm based on an optimal solution of \eqref{prob:randomized_single_segment} for the assortment customization problem, and prove that the algorithm returns a $1-(nv_{\max}+1)/Tv_{\min}$-approximation algorithm of the problem. The approximation ratio of the algorithm converges to $1$ as the number of customers $T$ gets large, suggesting that the algorithm is asymptotically optimal.
\end{remark}
}

\section{Multi-Segment Assortment Optimization}\label{sec:multi_segment}

In this section, we extend our discussion to a model with multiple customer segments. We assume that there are $m$ customer segments indexed by $\mathcal{M}=\{1,2,\dots,m\}$. Each customer segment $j\in\mathcal{M}$ is associated with an arrival probability $\theta_j$ as well as a MNL model with preference weights $v_{ij}>0$  where $v_{ij}$ is the preference weight of a customer in segment $j$ for product $i\in\mathcal{N}$. For each customer segment $j\in\mathcal{M}$, the preference weight of the no-purchase option is normalized to be $v_{0j}=1$. The expected revenue gained by offering assortment $S$ to a customer from segment $j$ is given by
\begin{equation*}
R_j(S)=\dfrac{\sum_{i\in S}r_iv_{ij}}{1+\sum_{i\in S}v_{ij}}.
\end{equation*}

Similar to previous sections, the products are categorized into categories $C_1,C_2,\dots,C_K$, and each category is associated with a minimum threshold $\ell_k$. 
In the deterministic version of the multi-segment problem, the decision of the seller is to offer an assortment $S_j$  to each customer segment $j\in\mathcal{M}$ in order to maximize total expected revenue $\sum_{j\in\mathcal{M}}\theta_jR_j(S_j)$.
Our covering constraint requires that the expected number of products that we offer from category $C_k$  should be at least $\ell_k$, where expectation is taken over  the arrival probabilities of customer segments. In particular, the expected number of products from category $C_k$ {  offered} to customers is given by $\sum_{j\in\mathcal{M}}\theta_j|S_j\cap C_k|$, and the covering constraints can be written as $\sum_{j\in\mathcal{M}}\theta_j|S_j\cap C_k|\geq \ell_k$.

We would like to mention that the randomized version of the multi-segment model, where we decide a distribution of assortment for each segment instead of a single assortment, is analogous to the single-segment model, and can be solved in polynomial time using a linear program similar to~\eqref{prob:lp_randomized}. Thus we defer the complete discussion of the randomized multi-segment problem to Appendix~\ref{sec:multi_segment_randomized}. In the rest of this section, we focus on the deterministic multi-segment problem. We refer to this problem as~\eqref{prob:deterministic_multi_segment}, and its formal definition is provided by

%In the deterministic setting, the expected number of products from category $C_k$ visible to customers is $\sum_{j\in\mathcal{M}}\theta_j|S_j\cap C_k|$, thus the covering constraints can be written as $\sum_{j\in\mathcal{M}}\theta_j|S_j\cap C_k|\geq \ell_k$.

\begin{equation}\label{prob:deterministic_multi_segment}\tag{m-DAOC}
\begin{aligned}
&\max_{S_j\subseteq\mathcal{N},\,\forall j\in\mathcal{M}}&&{\sum_{j\in\mathcal{M}}\theta_jR_j(S_j)}\\
&\text{s.t.}&&{\sum_{j\in\mathcal{M}}\theta_j|S_j\cap C_k|\geq \ell_k,\ \forall k\in\{1,2,\dots,K\}.}
\end{aligned}
\end{equation}

{  We would like to remark that the covering constraints in~\eqref{prob:deterministic_multi_segment}~is introduced across all customer segments rather than over individual customer segment. Since the covering constraints do not specify the lower bounds of $|S_j\cap C_k|$ for every individual customer segment $j$, \eqref{prob:deterministic_multi_segment} cannot be solved by running Algorithm~\ref{alg:single_segment} individually for every customer segment, thus a new algorithm is needed to solve the problem.} We will show that the complexity of~\eqref{prob:deterministic_multi_segment}~depends on $m$, the number of customer segments. In the following, we distinguish two different cases: constant (small) $m$ and general (large) $m$.

\subsection{Constant Number of Customer Segments}
When trying to solve~\eqref{prob:deterministic_multi_segment}, it is tempting to directly adopt the idea of \cref{alg:single_segment} in the deterministic single segment setting: first, use greedy algorithm to find a sequence of assortments $\{\hat{S}_j\}_{j\in\mathcal{M}}$ that approximately minimizes $\sum_{j\in\mathcal{M}}\sum_{i\in S_j}v_{ij}$  subject to the covering constraints; second, find the optimal expansion $S_j$ of $\hat{S}_j$ for each customer segment. However, we use the following example to show that the approach mentioned above can perform arbitrarily poorly.

\begin{example}\label{example:multi_segment}
Suppose $\mathcal{M}=\{1,2\}$, $\theta_1=\theta_2=1/2$, $\mathcal{N}=\{1,2,3\}$. The revenue and preference weights of the products are defined as
\begin{equation*}
r_1=M/\epsilon,\ r_2=r_3=\epsilon/M;\ v_{11}=v_{12}=\epsilon,\ v_{21}=v_{22}=2M,\ v_{31}=v_{32}=M,
\end{equation*}
where $\epsilon<1/10$ and $M>10$. There is only one category $C_1=\{2,3\}$ and we set $\ell_1=1$. In this setting, the output of greedy algorithm that minimizes $\sum_{j\in\mathcal{M}}\sum_{i\in S_j}v_{ij}$ subject to covering constraints is $\hat{S}_1=\hat{S}_2=\{3\}$, and the corresponding optimal expansion is $S_1=S_2=\{1,3\}$. The corresponding revenue is
\begin{equation*}
\theta_1R_1(\{1,3\})+\theta_2R_2(\{1,3\})=\dfrac{\epsilon+M}{1+M+\epsilon}<1.
\end{equation*}
However, $S_1=\{1,2,3\}$, $S_2=\{1\}$ is also a feasible assortment, and its revenue is at least $\theta_2R_2(\{1\})=M/(2+2\epsilon)>5M/11$. By setting $M\to\infty$ we have that the approach mentioned above can perform arbitrarily poorly. \hfill$\Box$
\end{example}

Example~\ref{example:multi_segment} shows that in the multi-segment problem, simply minimizing the sum of preference weights across all customer segments subject to covering constraints is not always a good idea. Note that minimizing the weighted sum  $\sum_{j\in\mathcal{M}}\sum_{i\in S_j}\theta_jv_{ij}$ can be as bad as minimizing $\sum_{j\in\mathcal{M}}\sum_{i\in S_j}v_{ij}$ because in Example~\ref{example:multi_segment}~these two approaches are equivalent as $\theta_j$ are equal. %Instead, it could be a better solution to satisfy covering constraints with fewer customer segments, and leaving other segments with weaker or no covering constraints, thus gaining higher revenue in those segments.

To tackle this problem, we adopt a slightly different approach. In particular, in the first step, instead of minimizing the sum of preference weights, we add a weight $\gamma_j$ for customer segment $j$ and consider the following weighted set cover problem
\begin{equation}\label{prob:weighted_set_cover_correct_weight}
\begin{aligned}
&\min_{S_j\subseteq\mathcal{N},\,\forall j\in\mathcal{M}}&&{\sum_{j\in\mathcal{M}}\sum_{i\in S_j}\gamma_jv_{ij}}\\
&\textup{s.t.}&&{\sum_{j\in\mathcal{M}}\theta_j|S_j\cap C_k|\geq \ell_k,\ \forall k\in\{1,2,\dots,K\}.}
\end{aligned}
\end{equation}
We use greedy algorithm to find the approximate solution $\{\hat{S}_j\}_{j\in\mathcal{M}}$ to~\eqref{prob:weighted_set_cover_correct_weight}. In the second step, we find the optimal expansion $S_j$ of $\hat{S}_j$ for each customer segment $j$.

In the following, we prove that if the weights are set to be $\gamma_j=\theta_jR_j(S_j^*)/\sum_{i\in S_j^*}v_i$, where $\{S_j^*\}_{j\in\mathcal{M}}$ is the optimal solution to~\eqref{prob:deterministic_multi_segment}, then weights will strike the right balance across different customer segments and output an $1/(\log K+2)$-approximation to the problem.

\begin{lemma}\label{prop:correct_weights}
Suppose $\gamma_j=\theta_jR_j(S_j^*)/\sum_{i\in S_j^*}v_i$, and $\{\hat{S}_j\}_{j \in {\cal M}}$ is an $\alpha$-approximation to~\eqref{prob:weighted_set_cover_correct_weight}. Let $S_j=\argmax_{S\supseteq\hat{S}_j}R_j(S)$ be the optimal expansion of $\hat{S}_j$, then $\{S_j\}_{j\in\mathcal{M}}$ is a $1/(\alpha+1)$-approximation to~\eqref{prob:deterministic_multi_segment}.
\end{lemma}

\begin{proof}
First, a direct application o Lemma~\ref{lemma:optimal_expansion}, using the sets $\hat{S}_j$ and $S^*_j$ implies
\begin{equation*}
\bar{R}_j(\hat{S}_j)=R_j(S_j)\geq \dfrac{1}{V_{j}(\hat{S}_j)/V_{j}(S_j^*)+1}R_j(S_j^*),
\end{equation*}
where we use the notation $V(S)=\sum_{i \in S} v_i$.
Taking the sum over all $j\in\mathcal{M}$ with weights $\{\theta_j\}_{j\in\mathcal{M}}$, we get
\begin{equation*}
\sum_{j\in\mathcal{M}}\theta_jR_j(S_j)\geq \sum_{j\in\mathcal{M}}\dfrac{\theta_jR_j(S_j^*)}{V_{j}(\hat{S}_j)/V_{j}(S_j^*)+1}.
\end{equation*}
Since $1/(x+1)$ is a convex function, by Jensen's inequality we have
\begin{align*}
\sum_{j\in\mathcal{M}}\dfrac{\theta_jR_j(S_j^*)}{V_{j}(\hat{S}_j)/V_{j}(S_j^*)+1}&\geq\left(\sum_{j\in\mathcal{M}}\theta_jR_j(S_j^*)\right)\cdot\dfrac{1}{(\sum_{j\in\mathcal{M}}\theta_jR_j(S_j^*)V_j(\hat{S}_j)/V_{j}(S_j^*))/(\sum_{j\in\mathcal{M}}\theta_jR_j(S_j^*))+1}\\
&=\left(\sum_{j\in\mathcal{M}}\theta_jR_j(S_j^*)\right)\cdot \dfrac{1}{(\sum_{j\in\mathcal{M}}\gamma_jV_j(\hat{S}_j))/(\sum_{j\in\mathcal{M}}\theta_jR_j(S_j^*))+1}.
\end{align*}
Since $\{S_j^*\}_{j\in\mathcal{M}}$ is a feasible solution to~\eqref{prob:weighted_set_cover_correct_weight}, and $\{\hat{S}_j\}_{j\in\mathcal{M}}$ is an $\alpha$-approximation to~\eqref{prob:weighted_set_cover_correct_weight}, using the definition of $\{\gamma_j\}_{j\in\mathcal{M}}$ we have
\begin{equation*}
\sum_{j\in\mathcal{M}}\gamma_jV_j(\hat{S}_j)\leq \alpha\sum_{j\in\mathcal{M}}\gamma_jV_j(S_j^*)=\alpha\sum_{j\in\mathcal{M}}\theta_j\dfrac{R_j(S_j^*)}{V_j(S_j^*)}V_j(S_j^*)=\alpha\sum_{j\in\mathcal{M}}\theta_jR_j(S_j^*).
\end{equation*}
Therefore
\begin{equation*}
\sum_{j\in\mathcal{M}}\theta_jR_j(S_j)\geq \left(\sum_{j\in\mathcal{M}}\theta_jR_j(S_j^*)\right)\cdot\dfrac{1}{(\alpha\sum_{j\in\mathcal{M}}R_j(S_j^*))/(\sum_{j\in\mathcal{M}}R_j(S_j^*))+1}=\dfrac{1}{\alpha+1}\sum_{j\in\mathcal{M}}\theta_jR_j(S_j^*),
\end{equation*}
which implies that $\{S_j\}_{j\in\mathcal{M}}$ is a $1/(\alpha+1)$-approximation to~\eqref{prob:deterministic_multi_segment}.
\end{proof}

Although Lemma~\ref{prop:correct_weights} provides a set of weights  which strike the right balance across different customer segments, the precise weights defined in Lemma~\ref{prop:correct_weights} depend on the optimal assortments, which is obviously unknown. For constant number of customer segments, we can take guesses for the correct weights $\{\gamma_j\}_{j\in\mathcal{M}}$ via a grid search, resulting in an algorithm with runtime exponential in the number of customer segments $m$. In particular, we define $v=\min\{v_{ij}:v_{ij}>0, i \in {\cal N}, j \in {\cal M}\}$, $V=\max\{v_{ij}: i \in {\cal N}, j \in {\cal M}\}$, $r=\min\{r_i:i\in\mathcal{N}\}$, $R=\max\{r_i:i\in\mathcal{N}\}$. For each $j\in\mathcal{M}$, we define the geometric grid $\Gamma_{\epsilon,j}$ as
\begin{equation*}
\Gamma_{\epsilon,j}=\left\{\dfrac{\theta_jr v}{(1+nV)^2}\cdot (1+\epsilon)^\ell\right\}_{\ell=1}^L,
\end{equation*}
where $L=\lceil \log_{1+\epsilon}((1+nV)^2R/rv)\rceil$. This grid starts from a lower bound of $\gamma_j$, finishes at an upper bound of $\gamma_j$ and has multiplicative steps of $1+\epsilon$.
We further define $\Gamma_{\epsilon,\mathcal{M}}=\prod_{j\in\mathcal{M}}\Gamma_{\epsilon,j}$.

Our algorithm is summarized in \cref{alg:constant_segments}. Note that the runtime of \cref{alg:constant_segments} is polynomial in the number of products and categories and $1/\epsilon$ but exponential in the number of customer segments.

\begin{algorithm}[!ht]
\caption{Multi-segment assortment optimization with constant number of segments}
\SingleSpacedXI
\label{alg:constant_segments}
\begin{algorithmic}
\For{$\hat{\boldsymbol{\gamma}}\in\Gamma_{\epsilon,\mathcal{M}}$}
\State Initialize $\hat{S}_{j,\hat{\boldsymbol{\gamma}}}\leftarrow \varnothing$ for all $j\in\mathcal{M}$
\While{there exists $k\in\{1,2,\dots,K\}$ such that $\sum_{j\in\mathcal{M}}\theta_j|\hat{S}_{j,\hat{\boldsymbol{\gamma}}}\cap C_k|<\ell_k$}
\State Set $c_{ij}=\theta_j|\{k\in\{1,2,\dots,K\}:i\in C_k,\,\sum_{j\in\mathcal{M}}\theta_j|\hat{S}_{j,\hat{\boldsymbol{\gamma}}}\cap C_k|<\ell_k\}|$
\State Set $(\hat{i},\hat{j})=\argmax_{i\in\mathcal{N},j\in\mathcal{M}}\hat{\gamma}_jv_{ij}/c_{ij}$
\State Update $\hat{S}_{\hat{j},\hat{\boldsymbol{\gamma}}}\leftarrow \hat{S}_{\hat{j},\hat{\boldsymbol{\gamma}}}\cup\{\hat{i}\}$
\EndWhile
\State Let $\bar{S}_{j,\hat{\boldsymbol{\gamma}}}$ be the optimal expansion of $\hat{S}_{j,\hat{\boldsymbol{\gamma}}}$, i.e.,
\begin{equation*}
\bar{S}_{j,\hat{\boldsymbol{\gamma}}}=\underset{S\supseteq \hat{S}_{j,\hat{\boldsymbol{\gamma}}}}{\mathrm{argmax}}\,R_j(S)
\end{equation*}
\EndFor
\State \textbf{return} $\{S_j\}_{j\in\mathcal{M}}$ that maximizes the expected revenue over $\{\{\bar{S}_{j,\hat{\boldsymbol{\gamma}}}\}_{j\in\mathcal{M}}:\hat{\boldsymbol{\gamma}}\in\Gamma_{\epsilon,\mathcal{M}}\}$.
\end{algorithmic}
\end{algorithm}

In the following, we prove that \cref{alg:constant_segments} is a $(1-\epsilon)/(\log K+2)$-approximation to~\eqref{prob:deterministic_multi_segment}. In other words, for constant number of customer segments, \cref{alg:constant_segments} still achieves the same approximation ratio as \cref{alg:single_segment}, and matches the hardness result of~\eqref{prob:deterministic_single_segment} provided in Theorem~\ref{thm:hardness_single_segment} up to a constant.

\begin{theorem} \label{thm:multi-segement-constant-m}
Algorithm~\ref{alg:constant_segments}~returns a $(1-\epsilon)/(\log K+2)$-approximation to~\eqref{prob:deterministic_multi_segment}. The runtime of Algorithm~\ref{alg:constant_segments}~is $O(m^2n^2K\cdot(\log((1+nV)^2R/rv))^m/\epsilon^m)$.
\end{theorem}

\begin{proof}

By definition of $\Gamma_{\epsilon,\mathcal{M}}$, there exists  $\hat{\boldsymbol{\gamma}}\in\Gamma_{\epsilon,\mathcal{M}}$ such that $\gamma_j\leq\hat{\gamma}_j\leq (1+\epsilon)\gamma_j$ for all $j\in\mathcal{M}$. Note that $\{\hat{S}_{j,\hat{\boldsymbol{\gamma}}}\}_{j\in\mathcal{M}}$ are obtained in Algorithm~\ref{alg:constant_segments} using a greedy algorithm. Therefore,  $\{\hat{S}_{j,\hat{\boldsymbol{\gamma}}}\}_{j\in\mathcal{M}}$ give $H_K$-approximation for the equivalent of problem~\eqref{prob:weighted_set_cover_correct_weight} where we replace the coefficients   ${\boldsymbol{\gamma}}$ by $\hat{\boldsymbol{\gamma}}$.
%$\hat{\boldsymbol{\gamma}}$, 
%Under this particular $\hat{\boldsymbol{\gamma}}$, $\{\hat{S}_{j,\hat{\boldsymbol{\gamma}}}\}_{j\in\mathcal{M}}$ is a $H_K$-approximation for~\eqref{prob:weighted_set_cover_correct_weight}. 
Moreover, since $\gamma_j \leq \hat{\gamma}_j \leq (1+\epsilon){\gamma}_j$ for all $j\in\mathcal{M}$, then $\{\hat{S}_{j,\hat{\boldsymbol{\gamma}}}\}_{j\in\mathcal{M}}$ gives a $(1+\epsilon)H_K$-approximation for~\eqref{prob:weighted_set_cover_correct_weight}. 
Therefore, from Lemma \ref{prop:correct_weights}, we obtain that the optimal expansion of  $\{\hat{S}_{j,\hat{\boldsymbol{\gamma}}}\}_{j\in\mathcal{M}}$ gives $ \dfrac{1}{(1+\epsilon)H_K+1}$-approximation to~\eqref{prob:deterministic_multi_segment}.
%\begin{equation*}
%\sum_{j\in\mathcal{M}}\theta_jR_j(\bar{S}_j)\geq \dfrac{1}{(1+\epsilon)H_K+1}\sum_{j\in\mathcal{M}}\theta_jR_j(S_j^*)\geq \dfrac{1-\epsilon}{\log K+2}\sum_{j\in\mathcal{M}}\theta_jR_j(S_j^*),
%\end{equation*}
This implies that the output of Algorithm~\ref{alg:constant_segments}~is a $(1-\epsilon)/(\log K+2)$-approximation to~\eqref{prob:deterministic_multi_segment}.

We further analyze the runtime of Algorithm~\ref{alg:constant_segments}. In the inner loop, the runtime for computing $\sum_{j\in\mathcal{M}}\theta_j|\hat{S}_{j,\hat{\boldsymbol{\gamma}}}\cap C_k|$ for each $k\in\{1,\dots,K\}$ is $O(mn)$. Therefore the runtime for identifying all categories $k\in\{1,2,\dots,K\}$ such that $\sum_{j\in\mathcal{N}}\theta_j|\hat{S}_{j,\hat{\boldsymbol{\gamma}}}\cap C_k|<\ell_k$ is $O(mnK)$. The runtime for computing $c_{ij}$ for each $i\in\mathcal{N}$ and $j\in\mathcal{M}$ is $O(K)$, thus the runtime for computing $c_{ij}$ for all $i\in\mathcal{N}$ and $j\in\mathcal{M}$ is $O(mnK)$. Given $\{c_{ij}\}_{i\in\mathcal{N},j\in\mathcal{M}}$, the runtime for computing $(\hat{i},\hat{j})$ is $O(mn)$. Therefore the total runtime of an iteration in the inner loop is $O(mnK)$. Since there are at most $mn$ iterations for computing $\{\hat{S}_{j,\hat{\boldsymbol{\gamma}}}\}_{j\in\mathcal{M}}$, for each $\hat{\boldsymbol{\gamma}}\in\Gamma_{\epsilon,\mathcal{M}}$ the total runtime for computing $\{\hat{S}_{j,\hat{\boldsymbol{\gamma}}}\}_{j\in\mathcal{M}}$ is $O(m^2n^2K)$. By Lemma 3.3 of~\cite{barre2023assortment}, given $\hat{S}_{j,\hat{\boldsymbol{\gamma}}}$, $\bar{S}_{j,\hat{\boldsymbol{\gamma}}}$ can be computed in $O(n)$ time for each $j\in\mathcal{M}$ and $\hat{\boldsymbol{\gamma}}\in\Gamma_{\epsilon,\mathcal{M}}$. Therefore the runtime for computing $\{\bar{S}_{j,\hat{\boldsymbol{\gamma}}}\}_{j\in\mathcal{M}}$ is $O(mn)$. Then we conclude that the total runtime for an iteration of the outer loop is $O(m^2n^2K)$. Finally, by definition $|\Gamma_{\epsilon,\mathcal{M}}|=L^m$ and we have
\begin{equation*}
L=O(\log_{1+\epsilon}((1+nV)^2R/rv))=O\left(\dfrac{\log((1+nV)^2R/rv)}{\log(1+\epsilon)}\right)=O\left(\dfrac{\log((1+nV)^2R/rv)}{\epsilon}\right).
\end{equation*}
We conclude that $L^m=O((\log((1+nV)^2R/rv))^m/\epsilon^m)$. Since the number of iterations of the outer loop is $L^m$, the total runtime for Algorithm~\ref{alg:constant_segments}~is $O(m^2n^2K\cdot(\log((1+nV)^2R/rv))^m/\epsilon^m)$.
\end{proof}

\subsection{General Number of Customer Segments}

In this section, we proceed to the case of general (large) number of customer segments. We show that \eqref{prob:deterministic_multi_segment} with general $m$ is harder than either~\eqref{prob:deterministic_single_segment}~or~\eqref{prob:deterministic_multi_segment}~in the constant $m$ case. Specifically, we prove in Theorem~\ref{thm:hardness_multi_segment} that with general $m$, \eqref{prob:deterministic_multi_segment} is NP-hard to approximate within a factor of $\Omega(1/m^{1-\epsilon})$ for any $\epsilon>0$. In particular, Theorem \ref{thm:hardness_multi_segment} shows that for general $m$, \eqref{prob:deterministic_multi_segment} does not admit an approximation that is sublinear in $m$. The proof of \cref{thm:hardness_multi_segment} is provided in Appendix~\ref{sec:thm:hardness_multi_segment} and uses reduction from the maximum independent set problem.

\begin{theorem}\label{thm:hardness_multi_segment}
For any $\epsilon>0$, it is NP-hard to approximate~\eqref{prob:deterministic_multi_segment}~within a factor of $\Omega(1/m^{1-\epsilon})$ even when each customer segment has an equal probability, unless $NP\subseteq BPP$.
\end{theorem}

In view of this hardness result, we propose an alternative algorithm that yields a $1/m(\log K + 2)$-approximation. 
Our approximation ratio matches the hardness result provided in Theorem~\ref{thm:hardness_multi_segment}~up to a logarithmic factor of $K$. The algorithm is provided in~\cref{alg:multi_segment_general}. Specifically, \cref{alg:multi_segment_general} simply offers the full universe of products to all except for one customer segment. The assortment offered to the remaining customer segment is determined by running \cref{alg:single_segment} on that customer segment such that the (remaining) covering constraints are satisfied. In particular, if we offer the full universe of products to all customer segments except $\ell \in {\cal M}$, for each category $k$, we update the threshold of products to be covered in that category to $\ell_k^{(\ell)}$ according to Equation \eqref{eq:segment_lower_bound}, and then we solve the single customer segment problem for segment $\ell$ using the updated thresholds. We run the algorithm for all $\ell \in {\cal M}$ and pick the best solution.
In the following, we prove that \cref{alg:multi_segment_general} returns a $1/m(\log K+2)$-approximation to~\eqref{prob:deterministic_multi_segment}.

\begin{algorithm}[!ht]
\caption{Multi-segment assortment optimization with general number of segments}
\SingleSpacedXI
\label{alg:multi_segment_general}
\begin{algorithmic}
\For{$\ell\in\mathcal{M}$}
\State For each $k\in\{1,2,\dots,K\}$, set
\begin{equation}
\ell_k^{(\ell)}=\left\lceil \dfrac{\ell_k-\sum_{j\in\mathcal{M}\backslash\{\ell\}}\theta_{j}|C_k|}{\theta_\ell}\right\rceil\label{eq:segment_lower_bound}
\end{equation}
\State Run~\cref{alg:single_segment}~on customer segment $\ell$ with minimum thresholds on categories being $\{\ell_k^{(\ell)}\}_{k=1}^K$. Let $\bar{S}^{(\ell)}$ be the output of~\cref{alg:single_segment}
\State Set $S_\ell^{(\ell)}=\bar{S}^{(\ell)}$, $S_j^{(\ell)}=\mathcal{N}$ for all $j\in\mathcal{M}\backslash\{\ell\}$
\EndFor
\State \textbf{return} $\{S_j\}_{j\in\mathcal{M}}$ that maximizes the expected revenue over $\{\{S_j^{(\ell)}\}_{j\in\mathcal{M}}:\ell\in\mathcal{M}\}$
\end{algorithmic}
\end{algorithm}

\begin{theorem} \label{thm:algo-multi-segement-general-m}
Algorithm~\ref{alg:multi_segment_general}~returns a $1/m(\log K+2)$-approximation to~\eqref{prob:deterministic_multi_segment}.
\end{theorem}

\begin{proof}
We first verify that for every $\ell\in\mathcal{M}$, $\{S_j^{(\ell)}\}_{j\in\mathcal{M}}$ is a feasible solution to~\eqref{prob:deterministic_multi_segment}. By definition of $S_\ell^{(\ell)}$ we have $|S_{\ell}^{(\ell)}\cap C_k|\geq \ell_k^{(\ell)}$ for all $k\in\{1,2,\dots,K\}$. Therefore for any $k\in\{1,2,\dots,K\}$,
\begin{equation*}
\sum_{j\in\mathcal{M}}\theta_j|S_j^{(\ell)}\cap C_k|\geq \theta_\ell\ell_k^{(\ell)}+|C_k|\sum_{j\in\mathcal{M}\backslash\{\ell\}} \theta_j\geq \ell_k-|C_k|\sum_{j\in\mathcal{M}\backslash\{\ell\}}\theta_j+|C_k|\sum_{j\in\mathcal{M}\backslash\{\ell\}}\theta_j=\ell_k,
\end{equation*}
which implies that for every $\ell\in\mathcal{M}$, $\{S_j^{(\ell)}\}_{j\in\mathcal{M}}$ is a feasible solution to~\eqref{prob:deterministic_multi_segment}.

We use $\{S_j^*\}_{j\in\mathcal{M}}$ to denote the optimal solution to~\eqref{prob:deterministic_multi_segment}, and define $j^*$ as \mbox{$j^*=\mathrm{argmax}\{\theta_jR_j(S_j^*):j\in\mathcal{M}\}$}. Then we have
\begin{equation*}
\theta_{j^*}R_{j^*}(S_{j^*}^*)\geq \dfrac{1}{m}\sum_{j\in\mathcal{M}}\theta_jR_j(S_j^*).
\end{equation*}
We claim that $S_{j^*}^*$ is a feasible solution to the following weighted set cover problem
\begin{equation}
\begin{aligned}
&\max_{S\subseteq\mathcal{N}}&&R_{j^*}(S)\\
&\textup{s.t.}&&{|S\cap C_k|\geq \ell_k^{(j^*)},\ \forall k\in\{1,2,\dots,K\}.}
\end{aligned}\label{prob:best_segment}
\end{equation}
In fact, since $\{S_j^*\}_{j\in\mathcal{M}}$ is a feasible solution to~\eqref{prob:deterministic_multi_segment}, for every $k\in\{1,2,\dots,K\}$ we have \mbox{$\sum_{j\in\mathcal{M}}\theta_j|S_j^*\cap C_k|\geq \ell_k$}, thus for every $k\in\{1,2,\dots,K\}$,
\begin{equation*}
|S_{j^*}^*\cap C_k|\geq \dfrac{1}{\theta_{j^*}}\left(\ell_k-\sum_{j\in\mathcal{M}\backslash\{j^*\}}\theta_j|S_j^*\cap C_k|\right)\geq \dfrac{\ell_k-\sum_{j\in\mathcal{M}\backslash\{j^*\}}\theta_j|C_k|}{\theta_{j^*}},
\end{equation*}
taking the ceiling implies that $|S_{j^*}^*-C_k|\geq \ell_k^{(j^*)}$ for all $k\in\{1,2,\dots,K\}$, thus $S_{j^*}^*$ is feasible to~\eqref{prob:best_segment}.

Since $\bar{S}^{(j^*)}$ is obtained by approximately solving~\eqref{prob:best_segment}~using~\cref{alg:single_segment}, by~\cref{thm:single_segment}~we have that $\bar{S}^{(j^*)}$ is a $1/(\log K+2)$-approximation to~\eqref{prob:best_segment}. Thus we get
\begin{equation*}
\sum_{j\in\mathcal{M}}\theta_jR_j(S_j^{(j^*)})\geq \theta_{j^*}R_{j^*}(\bar{S}^{(j^*)})\geq 1/(\log K+2)\cdot \theta_{j^*}R_{j^*}(S_{j^*}^*)\geq1/m(\log K+2)\cdot\sum_{j\in\mathcal{M}}\theta_jR_j(S_j^*).
\end{equation*}
Since~\cref{alg:multi_segment_general}~outputs $\{S_j^{(\ell)}\}_{j\in\mathcal{M}}$ that maximizes the expected revenue over all $\ell\in\mathcal{M}$, the expected revenue of the output of~\cref{alg:multi_segment_general}~is greater or equal to the expected revenue of $\{S_j^{(j^*)}\}_{j\in\mathcal{M}}$. We conclude that~\cref{alg:multi_segment_general}~outputs a $1/m(\log K+2)$-approximation to~\eqref{prob:deterministic_multi_segment}.
\end{proof}

{  We remark that the approximation ratio established in Theorem~\ref{thm:algo-multi-segement-general-m}~scales inversely proportional to the number of segments $m$. When the number of segments is large $m$, Theorem~\ref{thm:algo-multi-segement-general-m}~is unable to provide a strong performance guarantee for the solution returned by Algorithm~\ref{alg:multi_segment_general}. Nevertheless, recall that we have proven in Theorem~\ref{thm:hardness_multi_segment}~that it is NP-hard to approximate \eqref{prob:deterministic_multi_segment} within a factor of $\Omega(1/m^{1-\epsilon})$ for general $m$, so the approximation ratio obtained in Theorem~\ref{thm:algo-multi-segement-general-m}~matches the corresponding hardness result up to a logarithmic factor in $K$. }

\section{Numerical Experiments}\label{sec:numerical}

% In this section, we conduct numerical experiments using MNL models calibrated from real data. Specifically, we study the following questions. First, how much revenue does the seller lose by introducing covering constraints compared to the unconstrained optimal expected revenue? Second, how large can the gap in optimal expected revenue between~\eqref{prob:randomized_single_segment}~and~\eqref{prob:deterministic_single_segment}~be in practice? Third, how many assortments does the optimal solution to~\eqref{prob:randomized_single_segment}~randomize over in practice?

In this section, we conduct numerical experiments using MNL models calibrated from real data. Specifically, we study the following questions: First, how much revenue does the seller lose by introducing covering constraints compared to the unconstrained optimal expected revenue? Second, how large can the gap in optimal expected revenue between~\eqref{prob:randomized_single_segment} and~\eqref{prob:deterministic_single_segment} be in practice? Third, how many assortments does the optimal solution to~\eqref{prob:randomized_single_segment} randomize over in practice?

\subsection{Data}
We use the dataset \cite{ECommerceData} for our numerical experiments which includes purchase records from a large electronics online store from April 2020 to November 2020. Each purchase record in the dataset includes the date and time of the purchase, the ID of the purchased product, the type of the purchased product (\eg, smartphones, tablets, printers, tables) as well as its price and brand.

We conduct numerical experiments on selected product types that appear in the dataset. The product types we conduct numerical experiments on, as well as their numbers of products and brands, are provided in \cref{table:product_type_list}. For each product type, we separately calibrate the MNL model. Since it is meaningless to introduce covering constraints to very small brands, we first discard all products from brands with fewer than $10$ products to ensure a certain level of flexibility within each brand.

We divide the time horizon of the dataset into time intervals of two weeks and assume that the offered assortment is fixed during each interval. The assortment offered during each time interval is assumed to be the set of products that were sold at least once during that period. Since the dataset lacks records of customers leaving without a purchase, which is crucial for calibrating the MNL model, we add artificial no-purchase records to supplement the purchase records. Specifically, the number of no-purchase records in each time interval is set to be a fixed proportion $\alpha$ of the total number of purchases during that interval. In our experiments, we set the ratio $\alpha$ to $\{0.05,\,0.1,\,0.2,\,0.3\}$. Finally, we calibrate the preference weight of each product within the corresponding product type using the maximum likelihood estimator, based on the offered assortment in each time interval and the number of purchases of each product as well as the no-purchase option under the corresponding offered assortments.

\begin{table}[t]
\centering{\scriptsize
\begin{tabular}{|c|c|c|}
\hline
Type&\# Products&\# Brands\\
\hline
Smartphone&618&9\\
Headphone&604&17\\
Notebook&780&14\\
Tablet&226&7\\
TV&445&10\\
Mouse&302&11\\
Keyboard&102&7\\
Monitor&124&4\\
Printer&156&3\\
Clock&193&5\\
Table&394&11\\
Kettle&175&11\\
Refrigerator&482&16\\
\hline
\end{tabular}
\caption{Number of products and brands for each product type}
\label{table:product_type_list}}
\end{table}

\subsection{Experimental Setup}\label{sec:experimental_setup}

We categorize products by their prices and brands. For price categories, we determine four price ranges based on the $25\%$, $50\%$, and $75\%$ percentiles of the prices of all products within the corresponding product type, and products within the same price range form a single category. For brand categories, products from the same brand are considered as a single category. The covering constraints over the constructed categories are set to enforce a uniform number of representatives $\ell$ to be included in the assortment, i.e., $\ell_k = \ell$ for all $k \in \{1, 2, \dots, K\}$. We vary the number of representatives $\ell$ from $1$ to $5$.

We conduct numerical experiments in both the deterministic and randomized settings. The optimal expected revenue of~\eqref{prob:randomized_single_segment} can be computed exactly using the linear program~\eqref{prob:lp_randomized} as shown in Theorem~\ref{thm:lp_randomized_single_segment}. Since products are categorized based only on two features, prices and brands, we know by Lemma~\ref{lemma:tu_constraints} in Appendix~\ref{sec:tu_constraints} that the covering constraints of~\eqref{prob:deterministic_single_segment} are totally unimodular in our experimental setting. Thus, the optimal expected revenue can be computed exactly via an equivalent linear program as shown in \cite{sumida2021revenue}. Additionally, we compute the optimal revenue in the unconstrained setting (i.e., $R^*= \max_{S \subseteq {\cal N}} R(S)$), which can be solved in polynomial time under the MNL model, as the optimal assortment is revenue-ordered in the unconstrained assortment problem under the MNL model.

For each value of $\alpha \in \{0.05,0.1, 0.2. 0.3\}$, each value of  $ \ell \in \{1,2,3,4,5\}$ and each product type, we compute  the revenue loss  between  \eqref{prob:randomized_single_segment} and the unconstrained optimum  which is defined by the ratio  ${\sf Loss_r}= \frac{R^* - z^*_{\sf r}}{R^*} $ (in \%), where $R^*$ is the optimal revenue in the unconstrained setting and $z^*_{\sf r}$ is the optimal revenue of \eqref{prob:randomized_single_segment}. We report the obtained results in \cref{table:revenue_loss_randomized}. Additionally, we plot the revenue loss as a function of the number of representatives $\ell$  in \cref{fig:revenue_loss_randomized}. We also report the number of assortments that we randomize over in each instance of  \eqref{prob:randomized_single_segment} in  \cref{table:num_assortment}. Similarly, for each value of $\alpha \in \{0.05,0.1, 0.2. 0.3\}$, each value of  $ \ell \in \{1,2,3,4,5\}$ and each product type, we compute  the revenue loss  between  \eqref{prob:deterministic_single_segment} and the unconstrained optimum  which is defined by the ratio  ${\sf Loss_d}= \frac{R^* - z^*_{\sf d}}{R^*} $ (in \%), where $R^*$ is the optimal revenue in the unconstrained setting and $z^*_{\sf d}$ is the optimal revenue of \eqref{prob:deterministic_single_segment}. We report the obtained results in \cref{table:revenue_loss_deterministic}.

%We also compute he optimal revenue  in the unconstrained setting which can be solved polynomially under the classical MNL model (optimal assortment is revenue ordered).

% We categorize products by their prices and brands. For categories on prices, we determine four price ranges based on the $25\%$, $50\%$ and $75\%$ percentiles of prices of all products in the corresponding product type, and products within the same price range form a single category. For categories on brands, products from the same brand is also considered as a single category. The covering constraints over the constructed categories are set to enforce a uniform number of representatives $m$ should be included in the assortment, \textit{i.e.}, $\ell_k=m$ for all $k\in\{1,2,\dots,K\}$. We vary the number of representatives $m$ from $1$ to $5$.

% We conduct numerical experiments both in the deterministic and randomized setting. The optimal expected revenue of~\eqref{prob:randomized_single_segment}~can be computed exactly using the linear program~\eqref{prob:lp_randomized}~as shown in Theorem~\ref{thm:lp_randomized_single_segment}. Since products are categorized only based on two features, prices and brands, we conclude by Lemma~\ref{lemma:tu_constraints} in appendix that the covering constraints of~\eqref{prob:deterministic_single_segment}~are totally unimodular in our experimental setting. Thus the optimal expected revenue can be computed exactly via an equivalent linear program as shown in \cite{sumida2021revenue}.

\subsection{Results}
% The revenue loss compared to unconstrained optimum in the randomized setting is provided in \cref{table:revenue_loss_randomized} and plotted in \cref{fig:revenue_loss_randomized}; the revenue loss in the deterministic setting is provided in \cref{table:revenue_loss_deterministic}; the number of assortments in the randomized setting is provided in \cref{table:num_assortment}.

\begin{table}[t]
\centering
\subfigure[$\alpha=0.05$]{\scriptsize
\begin{tabular}{|c|c|c|c|c|c|}
\hline
 \diagbox[dir=NW]{Type}{$\ell$}  & 1      & 2      & 3      & 4      & 5 \\
\hline
Smartphone & 0.09\% & 0.19\% & 0.47\% & 0.79\% & 1.19\% \\
Headphone & 0.37\% & 0.80\%  & 1.37\% & 2.28\% & 3.28\% \\
Notebook  & 0.77\% & 1.70\% & 2.65\% & 3.71\% & 4.83\% \\
Tablet   & 0.29\% & 0.78\% & 1.59\% & 2.59\% & 4.64\% \\
TV   & 0.31\% & 1.05\% & 2.93\% & 6.13\% & 9.86\% \\
Mouse   & 0.37\% & 0.91\% & 1.57\% & 2.43\% & 3.68\% \\
Keyboard   & 0.76\% & 1.60\% & 2.84\% & 4.69\% & 9.79\% \\
Monitor   & 0.86\% & 1.87\% & 3.34\% & 5.37\% & 7.50\%  \\
Printer   & 0.17\% & 0.38\% & 0.67\% & 1.01\% & 1.34\% \\
Clock   & 0.99\% & 2.26\% & 3.62\% & 4.99\% & 6.41\% \\
Table   & 0.08\% & 0.28\% & 0.56\% & 1.02\% & 1.52\% \\
Kettle   & 0.22\% & 0.51\% & 0.87\% & 1.56\% & 2.80\%  \\
Refrigerator   & 0.03\% & 0.16\% & 0.40\%  & 0.68\% & 1.22\% \\
\hline
\end{tabular}
}
\ 
\subfigure[$\alpha=0.1$]{\scriptsize
\begin{tabular}{|c|c|c|c|c|c|c|}
\hline
 \diagbox[dir=NW]{Type}{$\ell$}    & 1      & 2      & 3      & 4      & 5 \\
\hline
Smartphone   & 0.06\% & 0.12\% & 0.30\% & 0.52\% & 0.80\% \\
Headphone   & 0.16\% & 0.37\% & 0.62\% & 1.06\% & 1.54\% \\
Notebook   & 0.37\% & 0.85\% & 1.36\% & 1.94\% & 2.56\% \\
Tablet   & 0.14\% & 0.36\% & 0.90\% & 1.57\% & 2.93\% \\
TV   & 0.11\% & 0.38\% & 1.14\% & 2.56\% & 4.36\% \\
Mouse   & 0.10\%  & 0.30\%  & 0.53\% & 0.79\% & 1.55\% \\
Keyboard   & 0.33\% & 0.72\% & 1.29\% & 2.20\%  & 4.87\% \\
Monitor   & 0.26\% & 0.56\% & 1.07\% & 1.86\% & 2.66\% \\
Printer   & 0.09\% & 0.19\% & 0.35\% & 0.53\% & 0.73\% \\
Clock   & 0.36\% & 0.89\% & 1.47\% & 2.09\% & 2.85\% \\
Table   & 0.04\% & 0.12\% & 0.22\% & 0.45\% & 0.71\% \\
Kettle   & 0.09\% & 0.21\% & 0.38\% & 0.70\%  & 2.32\% \\
Refrigerator   & 0.01\% & 0.06\% & 0.15\% & 0.27\% & 0.52\% \\
\hline
\end{tabular}
}

\subfigure[$\alpha=0.2$]{\scriptsize
\begin{tabular}{|c|c|c|c|c|c|c|}
\hline
\diagbox[dir=NW]{Type}{$\ell$}     & 1      & 2      & 3      & 4      & 5 \\
\hline
Smartphone   & 0.03\% & 0.06\% & 0.16\% & 0.27\% & 0.42\% \\
Headphone   & 0.06\% & 0.13\% & 0.22\% & 0.34\% & 0.47\% \\
Notebook   & 0.14\% & 0.33\% & 0.53\% & 0.77\% & 1.03\% \\
Tablet   & 0.03\% & 0.11\% & 0.36\% & 0.66\% & 1.25\% \\
TV   & 0.04\% & 0.15\% & 0.47\% & 1.12\% & 1.90\%  \\
Mouse   & 0.01\% & 0.03\% & 0.08\% & 0.14\% & 0.28\% \\
Keyboard   & 0.15\% & 0.36\% & 0.72\% & 1.19\% & 2.64\% \\
Monitor   & 0.07\% & 0.17\% & 0.31\% & 0.54\% & 0.72\% \\
Printer   & 0.03\% & 0.07\% & 0.12\% & 0.19\% & 0.27\% \\
Clock   & 0.26\% & 0.53\% & 0.85\% & 1.24\% & 1.73\% \\
Table   & 0.01\% & 0.04\% & 0.06\% & 0.13\% & 0.21\% \\
Kettle   & 0.05\% & 0.11\% & 0.19\% & 0.34\% & 0.59\% \\
Refrigerator  & 0.00\%      & 0.02\% & 0.04\% & 0.07\% & 0.15\% \\
\hline
\end{tabular}
}
\ 
\subfigure[$\alpha=0.3$]{\scriptsize
\begin{tabular}{|c|c|c|c|c|c|}
\hline
\diagbox[dir=NW]{Type}{$\ell$}      & 1      & 2      & 3      & 4      & 5 \\
\hline
Smartphone   & 0.01\% & 0.04\% & 0.10\%  & 0.18\% & 0.28\% \\
Headphone   & 0.02\% & 0.05\% & 0.08\% & 0.14\% & 0.21\% \\
Notebook   & 0.09\% & 0.23\% & 0.37\% & 0.54\% & 0.73\% \\
Tablet   & 0.01\% & 0.05\% & 0.22\% & 0.41\% & 0.76\% \\
TV   & 0.02\% & 0.09\% & 0.25\% & 0.66\% & 1.12\% \\
Mouse   & 0.01\% & 0.02\% & 0.04\% & 0.08\% & 0.15\% \\
Keyboard   & 0.09\% & 0.24\% & 0.47\% & 0.74\% & 1.54\% \\
Monitor   & 0.03\% & 0.07\% & 0.13\% & 0.25\% & 0.39\% \\
Printer   & 0.02\% & 0.03\% & 0.6\% & 0.9\% & 0.13\% \\
Clock   & 0.19\% & 0.38\% & 0.61\% & 0.85\% & 1.12\% \\
Table   & 0.01\% & 0.02\% & 0.03\% & 0.07\% & 0.11\% \\
Kettle   & 0.04\% & 0.08\% & 0.14\% & 0.26\% & 0.44\% \\
Refrigerator &  0.00\% & 0.00\% & 0.00\% &             0.01\% & 0.04\% \\
\hline
\end{tabular}
}
\caption{Revenue loss compared to unconstrained optimum in the randomized setting}
\label{table:revenue_loss_randomized}
\end{table}

\begin{table}[t]
\centering
\subfigure[$\alpha=0.05$]{\scriptsize
\begin{tabular}{|c|c|c|c|c|c|}
\hline
 \diagbox[dir=NW]{Type}{$\ell$}    & 1      & 2      & 3      & 4      & 5 \\
\hline
Smartphone  & 1 & 1 & 1 & 1 & 1 \\
Headphone  & 3 & 2 & 3 & 1 & 1 \\
Notebook  & 2 & 5 & 3 & 4 & 3 \\
Tablet &  1 & 1 & 1 & 1 & 1 \\
TV  & 1 & 1 & 1 & 1 & 3 \\
Mouse & 1 & 1 & 1 & 1 & 1 \\
Keyboard & 1 & 1 & 1 & 1 & 1 \\
Monitor & 1 & 1 & 1 & 1 & 6 \\
Printer & 1 & 1 & 1 & 2 & 1 \\
Clock & 2 & 3 & 4 & 3 & 3 \\
Table & 1 & 1 & 1 & 1 & 1 \\
Kettle & 1 & 1 & 1 & 1 & 2 \\
Refrigerator & 1 & 1 & 1 & 1 & 1 \\
\hline
\end{tabular}
}
\ 
\subfigure[$\alpha=0.1$]{\scriptsize
\begin{tabular}{|c|c|c|c|c|c|}
\hline
 \diagbox[dir=NW]{Type}{$\ell$}    & 1      & 2      & 3      & 4      & 5     \\
\hline
Smartphone & 1 & 1 & 1 & 1 & 1 \\
Headphone & 1 & 3 & 2 & 2 & 2 \\
Notebook & 2 & 2 & 1 & 2 & 1 \\
Tablet & 1 & 1 & 1 & 1 & 1 \\
TV & 1 & 1 & 1 & 1 & 4 \\
Mouse & 1 & 1 & 1 & 1 & 1 \\
Keyboard & 1 & 1 & 1 & 1 & 1 \\
Monitor & 1 & 1 & 1 & 1 & 1 \\
Printer & 1 & 1 & 1 & 1 & 1 \\
Clock & 2 & 1 & 1 & 1 & 1 \\
Table & 1 & 1 & 1 & 1 & 1 \\
Kettle & 1 & 1 & 1 & 1 & 1 \\
Refrigerator & 1 & 1 & 1 & 1 & 1 \\
\hline
\end{tabular}
}

\subfigure[$\alpha=0.2$]{\scriptsize
\begin{tabular}{|c|c|c|c|c|c|}
\hline
 \diagbox[dir=NW]{Type}{$\ell$}   & 1      & 2      & 3      & 4      & 5     \\
\hline
Smartphone & 1  & 1  & 1  & 1  & 1  \\
Headphone & 1  & 1  & 1  & 2  & 3  \\
Notebook & 2 & 2 & 1 & 2 & 1  \\
Tablet & 1 & 1 & 1 & 1 & 1 \\
TV & 1 & 1 & 1 & 1 & 2 \\
Mouse  & 1 & 1 & 1 & 1 & 2 \\
Keyboard & 1 & 1 & 1 & 1 & 1 \\
Monitor & 1 & 1 & 1 & 1 & 1 \\
Printer & 1 & 1 & 1 & 2 & 1 \\
Clock & 2 & 1 & 1 & 1 & 1 \\
Table & 1 & 1 & 1 & 1 & 1 \\
Kettle & 1 & 1 & 1 & 1 & 1 \\
Refrigerator & 2 & 1 & 1 & 1 & 1 \\
\hline
\end{tabular}
}
\ 
\subfigure[$\alpha=0.3$]{\scriptsize
\begin{tabular}{|c|c|c|c|c|c|}
\hline
 \diagbox[dir=NW]{Type}{$\ell$}   & 1      & 2      & 3      & 4      & 5     \\
\hline
Smartphone & 1 & 1 & 1 & 1 & 1 \\
Headphone & 2 & 1 & 1 & 3 & 2 \\
Notebook & 2 & 2 & 1 & 2 & 1 \\
Tablet & 1 & 1 & 1 & 1 & 1 \\
TV & 1 & 1 & 1 & 1 & 2 \\
Mouse & 1 & 1 & 1 & 1 & 1 \\
Keyboard & 1 & 1 & 1 & 1 & 1 \\
Monitor & 1 & 1 & 1 & 1 & 1 \\
Printer & 1 & 1 & 1 & 1 & 1 \\
Clock & 2 & 1 & 1 & 1 & 1 \\
Table & 1 & 1 & 1 & 1 & 1 \\
Kettle & 1 & 1 & 1 & 1 & 1 \\
Refrigerator & 1 & 1 & 1 & 1 & 1 \\
\hline
\end{tabular}}
\caption{Number of assortments in the randomized setting}
\label{table:num_assortment}
\end{table}

\begin{table}[t]
\centering
\subfigure[$\alpha=0.05$]{\scriptsize
\begin{tabular}{|c|c|c|c|c|c|}
\hline
 \diagbox[dir=NW]{Type}{$\ell$}    & 1      & 2      & 3      & 4      & 5 \\
\hline
Smartphone  & 0.09\% & 0.19\% & 0.47\% & 0.79\% & 1.19\% \\
Headphone  & 0.37\% & 0.80\%  & 1.37\% & 2.28\% & 3.28\% \\
Notebook & 0.77\% & 1.71\% & 2.67\% & 3.72\% & 4.84\% \\
Tablet & 0.29\% & 0.78\% & 1.59\% & 2.59\% & 4.64\% \\
TV & 0.31\% & 1.05\% & 2.93\% & 6.13\% & 10.16\% \\
Mouse & 0.37\% & 0.91\% & 1.57\% & 2.43\% & 3.68\% \\
Keyboard & 0.76\% & 1.60\%  & 2.84\% & 4.69\% & 9.79\% \\
Monitor & 0.86\% & 1.87\% & 3.34\% & 5.37\% & 8.51\% \\
Printer & 0.17\% & 0.38\% & 0.67\% & 1.01\% & 1.34\% \\
Clock & 1.50\%  & 2.53\% & 3.65\% & 4.99\% & 6.46\% \\
Table & 0.08\% & 0.28\% & 0.56\% & 1.02\% & 1.52\% \\
Kettle & 0.22\% & 0.51\% & 0.87\% & 1.56\% & 2.80\%  \\
Refrigerator & 0.03\% & 0.16\% & 0.40\%  & 0.68\% & 1.22\% \\
\hline
\end{tabular}
}
\ 
\subfigure[$\alpha=0.1$]{\scriptsize
\begin{tabular}{|c|c|c|c|c|c|}
\hline
 \diagbox[dir=NW]{Type}{$\ell$}    & 1      & 2      & 3      & 4      & 5      \\
\hline
Smartphone & 0.06\% & 0.12\% & 0.30\% & 0.52\% & 0.80\% \\
Headphone & 0.16\% & 0.37\% & 0.62\% & 1.06\% & 1.54\% \\
Notebook & 0.37\% & 0.85\% & 1.36\% & 1.94\% & 2.56\% \\
Tablet & 0.14\% & 0.36\% & 0.90\% & 1.57\% & 2.97\% \\
TV & 0.11\% & 0.38\% & 1.14\% & 2.56\% & 4.28\% \\
Mouse & 0.10\%  & 0.30\%  & 0.53\% & 0.79\% & 1.55\% \\
Keyboard & 0.33\% & 0.72\% & 1.29\% & 2.20\%  & 4.87\% \\
Monitor & 0.26\% & 0.56\% & 1.07\% & 1.86\% & 2.66\% \\
Printer & 0.09\% & 0.19\% & 0.35\% & 0.53\% & 0.73\% \\
Clock & 0.34\% & 0.89\% & 1.47\% & 2.09\% & 2.85\% \\
Table & 0.04\% & 0.12\% & 0.22\% & 0.45\% & 0.71\% \\
Kettle & 0.09\% & 0.21\% & 0.38\% & 0.70\%  & 1.32\% \\
Refrigerator & 0.01\% & 0.06\% & 0.15\% & 0.27\% & 0.52\% \\
\hline
\end{tabular}
}

\subfigure[$\alpha=0.2$]{\scriptsize
\begin{tabular}{|c|c|c|c|c|c|}
\hline
 \diagbox[dir=NW]{Type}{$\ell$}   & 1      & 2      & 3      & 4      & 5      \\
\hline
Smartphone & 0.03\% & 0.06\% & 0.16\% & 0.27\% & 0.42\% \\
Headphone & 0.06\% & 0.13\% & 0.22\% & 0.34\% & 0.47\% \\
Notebook & 0.14\% & 0.33\% & 0.53\% & 0.77\% & 1.03\% \\
Tablet & 0.03\% & 0.11\% & 0.36\% & 0.66\% & 1.25\% \\
TV & 0.04\% & 0.14\% & 0.47\% & 1.12\% & 1.90\%  \\
Mouse & 0.01\% & 0.03\% & 0.08\% & 0.14\% & 0.28\% \\
Keyboard & 0.15\% & 0.36\% & 0.72\% & 1.19\% & 2.64\% \\
Monitor & 0.07\% & 0.17\% & 0.31\% & 0.54\% & 0.72\% \\
Printer & 0.03\% & 0.07\% & 0.12\% & 0.19\% & 0.27\% \\
Clock & 0.26\% & 0.53\% & 0.85\% & 1.24\% & 1.73\% \\
Table & 0.01\% & 0.04\% & 0.06\% & 0.13\% & 0.21\% \\
Kettle & 0.05\% & 0.11\% & 0.19\% & 0.34\% & 0.59\% \\
Refrigerator & 0.00\% & 0.02\% & 0.04\% & 0.07\% & 0.15\% \\
\hline
\end{tabular}
}
\ 
\subfigure[$\alpha=0.3$]{\scriptsize
\begin{tabular}{|c|c|c|c|c|c|}
\hline
\diagbox[dir=NW]{Type}{$\ell$}  & 1      & 2      & 3      & 4      & 5 \\
\hline
Smartphone & 0.01\% & 0.04\% & 0.10\%  & 0.18\% & 0.28\% \\
Headphone & 0.02\% & 0.05\% & 0.08\% & 0.14\% & 0.21\% \\
Notebook & 0.09\% & 0.23\% & 0.37\% & 0.54\% & 0.73\% \\
Tablet & 0.01\% & 0.05\% & 0.22\% & 0.41\% & 0.76\% \\
TV & 0.02\% & 0.09\% & 0.25\% & 0.66\% & 1.12\% \\
Mouse & 0.01\% & 0.02\% & 0.04\% & 0.08\% & 0.15\% \\
Keyboard & 0.09\% & 0.24\% & 0.47\% & 0.74\% & 1.54\% \\
Monitor & 0.03\% & 0.07\% & 0.13\% & 0.25\% & 0.39\% \\
Printer & 0.02\% & 0.03\% & 0.06\% & 0.09\% & 0.13\% \\
Clock & 0.19\% & 0.38\% & 0.61\% & 0.85\% & 1.12\% \\
Table & 0.01\% & 0.02\% & 0.03\% & 0.07\% & 0.11\% \\
Kettle & 0.04\% & 0.08\% & 0.14\% & 0.26\% & 0.44\% \\
Refrigerator & 0.00\%   & 0.00\%  & 0.00\%  & 0.01\% & 0.04\% \\
\hline
\end{tabular}
}
\caption{Revenue loss compared to unconstrained optimum in the deterministic setting}
\label{table:revenue_loss_deterministic}
\end{table}

\begin{figure}[t]
\centering
\subfigure[$\alpha=0.05$]{\includegraphics[height=6cm]{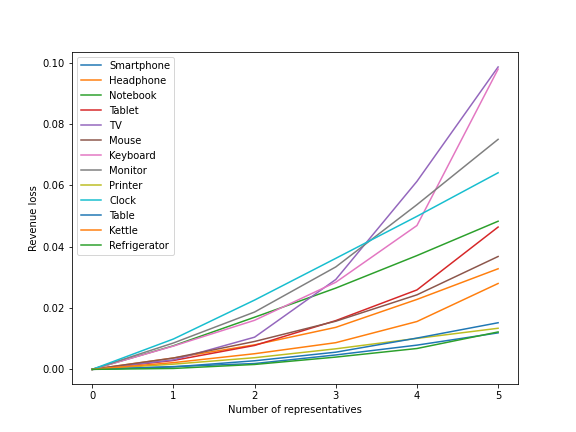}}
\subfigure[$\alpha=0.1$]{\includegraphics[height=6cm]{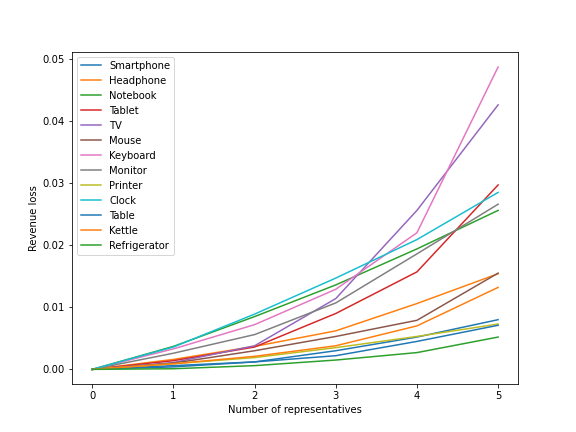}}
\subfigure[$\alpha=0.2$]{\includegraphics[height=6cm]{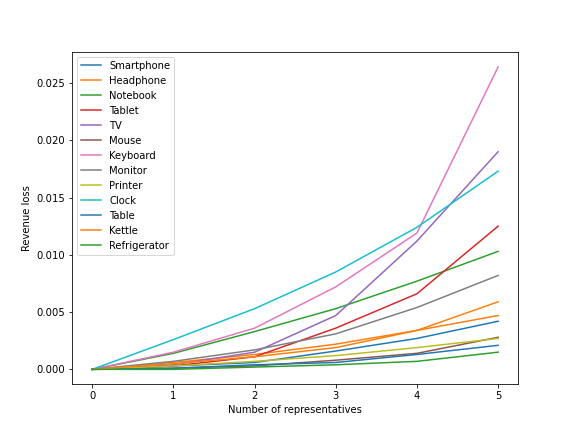}}
\subfigure[$\alpha=0.3$]{\includegraphics[height=6cm]{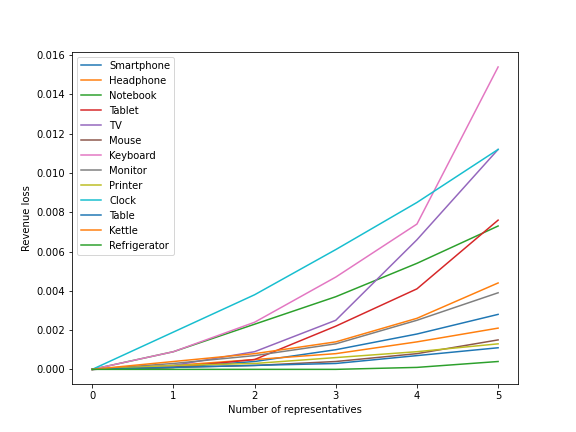}}
\caption{Revenue loss compared to unconstrained optimum in the randomized setting}
\label{fig:revenue_loss_randomized}
\end{figure}

We summarize the results in the following three aspects.

 \noindent\textbf{Price of Visibility:} 
% We observe from Tables~\ref{table:revenue_loss_randomized} and \ref{table:revenue_loss_deterministic} that the revenue loss due to introducing covering constraints is very small for most of the cases. It insteses slowsly as we increace the threshold $\ell$ from 1 to 5. In particular, for most product types, the revenue loss both in the deterministic and randomized settings is  smaller than $5\%$, and even smaller thant 1-2\% for small value of $\ell$ like 1,2,3.  The few exceptions where the loss was higher than $5\%$ are for TVs, keyboards, monitors and clocks, and even for these product types, revenue loss exceeds only for the case of $\alpha=0.05$ and the number of representatives being $4$ or $5$. The largest revenue loss is $10.16\%$, occurring for TVs in the deterministic setting when $\alpha=0.05$ and number of representatives being $5$.
We observe from Tables~\ref{table:revenue_loss_randomized} and \ref{table:revenue_loss_deterministic} that the revenue loss due to the introduction of covering constraints is minimal in most cases. It increases gradually as the threshold $\ell$ rises from 1 to 5. Specifically, for most product types, the revenue loss in both deterministic and randomized settings remains below $5\%$, and even below 1-2\% for smaller values of $\ell$, such as 1, 2, and 3. The few exceptions where the revenue loss exceeds $5\%$ involve product types like TVs, keyboards, monitors, and clocks, particularly when $\alpha = 0.05$ and the number of representatives is 4 or 5. The highest observed revenue loss is $10.16\%$, which occurs for TVs in the deterministic setting when $\alpha = 0.05$ and the number of representatives is 5.

The reasons why these four product types exhibit larger revenue loss compared to other product types can be explained as follows. As shown in \cref{table:product_type_list}, there are only $102$ available keyboards, making keyboards the product type with the smallest number of products among all tested product types. However, the number of brands for keyboard is $7$, a relatively large number of brands relative to the total number of products. Therefore adding representatives uniformly for all brands results in a larger portion of products being added to the assortment, and further causing a larger revenue loss. For TVs, monitors and clocks, these product types all contain a small number of extremely expensive products while the majority of products are cheap. Specifically, the most expensive TV costs more than $50,000$ dollars while most of the TVs cost less than $1000$ dollars; the most expensive monitor and clock cost more than $6000$ dollars, while a majority of monitors and clocks cost less than $500$ dollars. As a result, for these product types a large portion of revenue is gained from a few extremely expensive products, and introducing covering constraints results in adding a large number of cheap products, significantly reducing the purchase probability of the expensive products and thus causing relatively large revenue loss.

One can also see from \cref{fig:revenue_loss_randomized} that the revenue loss exhibits convexity in the number of representatives for each category. The revenue loss is minimal when introducing covering constraints with small number of representatives. Only when the number of representatives becomes larger does the revenue loss increase faster and faster.

\noindent\textbf{Number of Assortments in the Randomized Setting:} We observe from \cref{table:num_assortment} that the optimal solution of~\eqref{prob:randomized_single_segment}~is randomizes over only a small number of assortments for all cases. The maximum number of assortments is $6$, occurring for monitors when the number of representatives is $5$. The results in \cref{table:num_assortment} are consistent with Proposition~\ref{prop:num_assortment}, as the total number of categories is at least $7$ for all product types. Furthermore, one can also see from \cref{table:num_assortment} that the optimal solution in the randomized setting deterministically offers a single assortment for a majority of cases. Among the cases where the optimal solution does randomize over multiple assortments, the number of assortments randomized over is mostly $2$ or $3$, while the optimal solution rarely randomizes over more than $4$ assortments. From this observation one can conclude that in practice the optimal solution in the randomized setting frequently offers a single assortment deterministically, while for other cases the optimal solution randomizes over only a few assortments.

\noindent\textbf{Comparison between the Deterministic and Randomized Setting:} We have shown in Proposition~\ref{prop:integrality_gap} that the ratio in optimal expected revenue between the deterministic and randomized setting can be arbitrarily large. However, by comparing \cref{table:revenue_loss_randomized} and \cref{table:revenue_loss_deterministic} we observe that the optimal expected revenue in the randomized and deterministic settings are very close to each other. The difference in revenue loss between~\eqref{prob:randomized_single_segment}~and~\eqref{prob:deterministic_single_segment}~is at most $0.3\%$ for all cases, and for many cases the optimal expected revenues of the two settings are exactly the same. The reason for this is that the optimal solution in the randomized setting offers a single assortment deterministically in many cases, thus the optimal revenue in the randomized setting can also be achieved using a single assortment. For most of the other cases, the optimal solution for~\eqref{prob:randomized_single_segment} randomizes over a few assortments, and thus the gain from  randomization is limited. Furthermore, in the example constructed in Proposition~\ref{prop:integrality_gap}, the difference in both revenue and preference weights among products are extremely large, while in practice both revenue and preference weights of products are usually within a reasonable range, thus we expect large gaps between the deterministic and randomized settings to be rarely seen in practice.

{  

\noindent\textbf{Impact of Covering Constraints on Different Price Ranges:} To understand the impact of covering constraints on the expected revenue from each category level, we study how the expected revenue generated from each price category changes as we enforces stronger covering constraints. In this experiment, we focus on keyboards and the setting with $\alpha=0.05$ (similar phenomena can also be observed in other product types and other values of $\alpha$). Recall that in our experiments, besides categorizing products by brands, we construct price categories by dividing products into four price ranges. We refer to these four price categories as ``high", ``medium high", ``medium low", and ``low" from the highest revenues to the lowest revenues. We vary the minimum number of representatives for each category $\ell$ from $1$ to $5$. Normalizing the unconstrained optimal expected revenue as $1$, we present the expected revenue generated from each price category under different values of $\ell$ in Table~\ref{table:revenue_category}.

\begin{table}[!ht]
\centering{\scriptsize
\begin{tabular}{|c|c|c|c|c|c|}
\hline
\diagbox[dir=NW]{Category}{$\ell$}&1&2&3&4&5\\
\hline
High&0.7817&0.7063&0.6716&0.6491&0.5790\\
Medium High&0.2070&0.2714&0.2899&0.2802&0.2645\\
Medium Low&0.0031&0.0054&0.0083&0.0212&0.0548\\
Low&0.0006&0.0009&0.0017&0.0025&0.0037\\
\hline
Total&0.9924&0.9840&0.9716&0.9531&0.9021\\
\hline
\end{tabular}
\caption{Normalized expected revenue from each price category for keyboards when $\alpha=0.05$}
\label{table:revenue_category}}
\end{table}

One can see from Table~\ref{table:revenue_category} that, as the minimum number of representatives increases from $1$ to $5$, the expected revenue from high revenue products steadily decreases. The reason for such decrease is that as $\ell$ increases, more products with lower revenue are included in the assortment, which reduces the purchase probabilities of high revenue products. One can also see a significant increase in expected revenue for medium low and low revenue products as $\ell$ gets large. The reason for such increase is that as $\ell$ increases, more products with medium low and low revenue are included in the assortment due to covering constraints, resulting in an increase in the total expected revenue from the two categories. As for medium high revenue products, we see an increase in the expected revenue from the corresponding category when $\ell$ increases from $1$ to $3$, but the expected revenue from this category decreases as $\ell$ increases from $3$ to $5$. These observations suggest that as the seller enforce stronger covering constraints, the expected revenue generated by high revenue products will decrease, while the expected revenue generated by low revenue products will increase.

}

\section{Conclusion}

We have considered assortment optimization under the MNL model with covering constraints, in both deterministic and randomized settings. For deterministic assortment optimization with covering constraints, we proved that the problem is NP-hard to approximate within a factor of $(1+\epsilon)/\log K$ for any $\epsilon>0$, and presented a $1/(\log K + 2)$-approximation algorithm that matches the hardness of approximation. For randomized assortment optimization with covering constraints, we showed that the problem can be reformulated as an equivalent linear program, making it solvable in polynomial time. Additionally, we extended our analysis to a deterministic multi-segment model, establishing a $(1-\epsilon)/(\log K + 2)$-approximation algorithm for a constant number of customer segments $m$ for any $\epsilon>0$. For general $m$, we proved that the problem is NP-hard to approximate within a factor of $\Omega(1/m^{1-\epsilon})$ for any $\epsilon > 0$, and provided a $1/m(\log K + 2)$-approximation algorithm. Our numerical experiments using an e-commerce purchase dataset showed that revenue loss due to covering constraints is minimal, the optimal expected revenues for deterministic and randomized settings are very close, and that the randomized setting’s optimal solution typically involves only a small number of assortments. {  In Appendix H, we consider a joint assortment optimization and product framing problem with covering constraints, where the seller not only decides which products to offer, but also decide which position each offered product should be placed at. We establish an approximation algorithm such that, when the browsing probabilities are monotonically decreasing and there are at least $3n$ available positions, the algorithm returns a $1/2\log_2(8n)$-approximation to the problem.}

For future studies, we suggest further exploration in the following directions. First, this work focuses on assortment optimization with covering constraints under the MNL model. Future research could explore the application of covering constraints to other choice models, such as the generalized attraction model, nested logit model, or Markov chain choice model.
Second, we addressed assortment optimization only with covering constraints. In practice, sellers often face additional constraints. A significant area for future research is considering assortment optimization under both covering and packing constraints. Finally, our work focused on static problems where assortments are pre-determined for each customer segment under covering constraints. A valuable extension would be studying the online problem, where customers arrive sequentially, the decision maker observes their type, and offers an assortment. In this setting, the covering constraint might not need to be satisfied at each time step but should hold cumulatively at the end.

{

\bibliographystyle{ormsv080}
\bibliography{ref}
}

\newpage
\begin{appendices}{\Large \noindent\textbf{Appendix}}

\section{Special Cases of Totally Unimodular Constraints}\label{sec:tu_constraints}

In Lemma~\ref{lemma:tu_constraints}, we prove that if the categories can be divided into two groups such  that any two categories of the same group are disjoint, then the covering constraints are totally unimodular.

\begin{lemma}\label{lemma:tu_constraints}
Consider categories $C_1,C_2,\dots,C_K\subseteq\mathcal{N}$. Suppose there exists $k_0\in\{1,2,\dots,K\}$ such that $C_i\cap C_j=\varnothing$ if either $i<j\leq k_0$ or $k_0<i<j$. Then the constraints in~\eqref{prob:deterministic_single_segment}~are totally unimodular.
\end{lemma}

\begin{proof}
We define the constraint matrix $A\in\mathbb{R}^{n\times K}$ as
\begin{equation*}
A_{ik}=\begin{cases}
1,\ &\text{if}\ i\in C_k,\\
0,\ &\text{if}\ i\notin C_k.
\end{cases}
\end{equation*}
To prove that the covering constraints are totally unimodular, it suffices to prove that $A$ is a totally unimodular matrix. Denote $B$ the submatrix formed by the first $k_0$ rows of $A$, and denote $C$ the submatrix formed by the last $K-k_0$ rows of $A$. In particular, the submatrix $B$ corresponds to covering constraints over categories $C_1,C_2,\dots,C_{k_0}$, and the submatrix $C$ corresponds to covering constraints over categories $C_{k_0+1},\dots,C_K$. Since $C_i\cap C_j=\varnothing$ for all $i<j\leq k_0$, then for each $i\in\{1,2,\dots,n\}$ there exists at most one index $k \in \{1, \ldots, k_0 \}$ such that $i\in C_k$. Therefore each column of $B$ contains at most one $1$ and the rest are zeros. Similarly each column of $C$ also contains at most one $1$ and the rest are zeros. 

To prove that $A$ is totally unimodular, by the result of Theorem 5 of \cite{tamir1975totally}, it suffices to prove that every row submatrix of $A$ has an equitable bicoloring. Here, an equitable bicoloring of a matrix is defined as a bipartition of rows such that the difference in sums of rows of the two parts is a vector whose entries are $0,\pm 1$. Consider any row submatrix $A'$ of $A$. We define  $B'$ as the matrix formed by rows of $A'$ in block $B$, and $C'$ as the matrix formed by rows of $B'$ in block $C$. Then each column of $B'$ and $C'$ has at most a $1$, thus the sums of rows of $B'$ and $C'$ are vectors with entries either $0$ or $1$, and the difference of sum of rows of $B'$ and $C'$ is a vector whose entries are $0,\pm 1$. Therefore each row submatrix $A$ has a equitable bicoloring, and we conclude that $A$ is totally unimodular.
\end{proof}

\section{Proof of \cref{thm:hardness_single_segment}} \label{sec:thm:hardness_single_segment}

We start by presenting the minimum set cover problem.

    \noindent\textbf{Minimum set cover problem.} Given elements $\{1,2,\dots,n\}$ and sets $\{C_1,C_2,\dots,C_K\}$ with $C_k\subseteq\{1,2,\dots,n\}$ for each $k\in\{1,2,\dots,K\}$, we say that element $i$ covers set $C_k$ if $i\in C_k$. The goal of the minimum set cover problem is to find a subset of elements with minimum size such that each set $C_1,C_2,\dots,C_K$ is covered with at least one of the elements in the subset. This problem is NP-hard to approximate within a factor of $(1-\epsilon)\log K$ for any $\epsilon>0$, unless $P=NP$ (see \citealt{dinur2014analytical}).

	Consider an instance of the minimum set cover problem with the universe of elements $\{1,2,\dots,n\}$ and sets $C_1,C_2,\dots,C_K$. We assume that $n\geq 2$. We set the universe of products as $\{1,\dots,n,n+1\}$. We construct an instance of~\eqref{prob:deterministic_single_segment}~as follows. Consider any $0<\epsilon<1/2$. We define $\epsilon_1$ as the positive number that satisfies $(1+\epsilon_1)^2/(1-\epsilon_1)=1+\epsilon$, then we get $0<\epsilon_1<1/2$. We set $\delta=\epsilon_1/n$ and $M=n/\epsilon_1$, and set the revenue and preference weights as 
    \begin{equation*}
    r_{n+1}=M/\delta,\ v_{n+1}=\delta,\ r_i=\delta/M,\ v_i=M,\ \forall i\in\{1,2,\dots,n\}.
    \end{equation*}
    We further set $\ell_k=1$ for all $k\in\{1,2,\dots,K\}$. We define $\mathcal{S}$ as the set of all feasible subsets of $\{1,2,\dots,n\}$ that satisfy the covering constraints.
		
    Since $r_i=\delta/M=\epsilon_1^2/n^2$ for all $i\in\{1,2,\dots,n\}$, the expected revenue of any assortment without product $n+1$ is strictly less than $\epsilon_1^2/n^2$. On the other hand, the expected revenue of any assortment $S$ containing product $n+1$ is at least
    \begin{equation*}
    R(S) \geq \dfrac{r_{n+1}v_{n+1}}{1+\sum_{i\in S} v_i} = \frac{M}{1+\delta + M (|S|-1)}  \geq\dfrac{M}{1+\delta+Mn}\geq\dfrac{1}{2n}\geq\dfrac{\epsilon_1^2}{n^2}.
    \end{equation*}
    Here, the second inequality is due to $|S|\leq n+1$, the third inequality is due to $Mn\geq M\geq 1/\epsilon_1\geq 2\geq 1+\delta$. Therefore we conclude that the optimal assortment must contain product $n+1$, and that $R(S\cup\{n+1\})\geq R(S)$ for any assortment $S\subseteq\{1,2,\dots,n+1\}$.
				
	For every $S\subseteq\{1,2,\dots,n\}$, we have
	\begin{equation*}
	M\leq r_{n+1}v_{n+1}+\sum_{i\in S}r_iv_i=M+\epsilon_1|S|/n\leq (1+\epsilon_1)M.
	\end{equation*}
    Here, the first inequality is due to $r_{n+1}v_{n+1}=M$ and the last inequality is due to $|S|/n\leq 1\leq M$. Similarly, we get
    \begin{equation*}
    M|S|\leq 1+v_{n+1}+\sum_{i\in S}v_i=1+\delta+M|S|\leq (1+\epsilon_1)M|S|.
    \end{equation*}
%    Here, the first inequality is due to $v_i=M$ for all $i\in \{1,2,\dots,n\}$ and 
Here, the third inequality is due to $\epsilon_1M|S|\geq \epsilon_1M\geq n\geq 2\geq 1+\delta$.
	Therefore for every $S\subseteq\{1,2,\dots,n\}$, we have
	\begin{equation}
		\dfrac{1}{(1+\epsilon_1)|S|}\leq R(S\cup\{n+1\})\leq\dfrac{1+\epsilon_1}{|S|}.\label{eq:revenue_versus_cover}
	\end{equation}

    Consider any feasible solution $S$ of the set cover problem, we have that $S\cup\{n+1\}$ is a feasible solution to \eqref{prob:deterministic_single_segment}, and by \eqref{eq:revenue_versus_cover} we get
    \begin{equation}
    R(S\cup\{n+1\})\geq \dfrac{1}{(1+\epsilon_1)|S|}.\label{eq:opt_lower_bound}
    \end{equation}
    Define $S^*_{\text{assort}}$ as the optimal assortment to~\eqref{prob:deterministic_single_segment}, and $S^*_{\text{cover}}$ as the optimal solution to the set cover problem, then by~\eqref{eq:opt_lower_bound}~we have
    \begin{equation*}
    R(S^*_{\text{assort}})=\max_{S\in\mathcal{S}}R(S\cup\{n+1\})\geq\max_{S\in\mathcal{S}}\dfrac{1}{(1+\epsilon_1)|S|}=\dfrac{1}{(1+\epsilon_1)\min_{S\in\mathcal{S}}|S|}=\dfrac{1}{(1+\epsilon_1)|S^*_{\text{cover}}|}.
    \end{equation*}
    Here, the first equation is due to our conclusion that the optimal assortment must contain product $n+1$. Consider any feasible solution $S$ of \eqref{prob:deterministic_single_segment}, we have that $S\cap\{1,2,\dots,n\}$ is a feasible solution to the set cover problem, and by~\eqref{eq:revenue_versus_cover} we get
    \begin{equation*}
    R(S)\leq R(S\cup\{n+1\})\leq \dfrac{1+\epsilon_1}{|S\cap\{1,\dots,n\}|}.
    \end{equation*}
    Suppose $S$ is an $\alpha$-approximation to~\eqref{prob:deterministic_single_segment}, then
    \begin{equation*}
    \dfrac{|S\cap\{1,\dots,n\}|}{|S^*_{\text{cover}}|}\leq \dfrac{1+\epsilon_1}{R(S)}\cdot (1+\epsilon_1)R(S^*_{\text{assort}})=(1+\epsilon_1)^2\dfrac{R(S_{\text{assort}}^*)}{R(S)}\leq \dfrac{(1+\epsilon_1)^2}{\alpha}.
    \end{equation*}
    Therefore $S\cap\{1,2,\dots,n\}$ is a $(1+\epsilon_1)^2/\alpha$ approximation to the set cover problem. Specifically, if we set $\alpha=(1+\epsilon)/\log K$, then by definition of $\epsilon_1$ we get
    \begin{equation*}
    \dfrac{(1+\epsilon_1)^2}{\alpha}=\dfrac{(1+\epsilon_1)^2}{1+\epsilon}\cdot\log K=\dfrac{(1-\epsilon_1)(1+\epsilon_1)^2}{(1+\epsilon_1)^2}\cdot\log K=(1-\epsilon_1)\log K.
    \end{equation*}
    Thus we conclude that if $S$ is a $(1+\epsilon)/\log K$ approximation to~\eqref{prob:deterministic_single_segment}, then $S\cap\{1,2,\dots,n\}$ is a $(1-\epsilon_1)\log K$ approximation to the set cover problem.
		
	By~\cite{dinur2014analytical}, the set cover problem is NP-hard to approximate within a factor of $(1-\epsilon_1)\log K$ for any $\epsilon_1>0$. Thus~\eqref{prob:deterministic_single_segment}~is also NP-hard to approximate within a factor of $(1+\epsilon)/\log K$ for any $\epsilon>0$.

{ 

\section{DAOC with Cardinality Constraint}\label{sec:daoc_cardinality}

In this section, we consider the assortment optimization problem with covering and cardinality constraints. Specifically, the seller deterministically offers an assortment $S$ to maximize the total expected revenue, subject to the covering constraints $|S\cap C_k|\geq\ell_k$ for all $k\in\{1,2,\dots,K\}$ as well as a cardinality constraint $|S|\leq L$. We refer to this problem as \eqref{prob:daoc_cardinality}, and its formal definition is given by

\begin{equation}\tag{DAOC-Cardinality}\label{prob:daoc_cardinality}
\begin{aligned}
&\max_{S\subseteq\mathcal{N}}&&R(S)\\
&\text{s.t.}&&|S\cap C_k|\geq \ell_k,\ \forall k\in\{1,2,\dots,K\},\\
&&&|S|\leq L.
\end{aligned}
\end{equation}

In Theorem~\ref{thm:hardness_single_segment}, we have proven that \eqref{prob:deterministic_single_segment} is NP-hard to approximate within a factor of $(1+\epsilon)/\log K$ for any $\epsilon>0$. This suggests that \eqref{prob:daoc_cardinality} is NP-hard to approximate within a factor of $(1+\epsilon)/\log K$ for any $\epsilon>0$, even in the special case of $L=n$. Furthermore, minimizing $|S|$ subject to the covering constraints is also known to be NP-hard to approximate within a factor of $(1-\epsilon)\log K$ for all $\epsilon>0$ (see \citealt{dinur2014analytical}). This also suggests that it is NP-hard to determine whether \eqref{prob:daoc_cardinality} is feasible. Therefore, the best one can hope for is a bicriteria approximation algorithm that, when the problem is feasible, returns an approximate solution $S$ that satisfies the covering constraints, attains an expected revenue at least $\Omega(1/\log K)$ fraction of the optimal expected revenue, and satisfies the cardinality constraint up to a factor of $O(\log K)$. In other words, the assortment returned by the algorithm has a cardinality at most $L\cdot O(\log K)$, and violates the cardinality constraint by a factor of at most $O(\log K)$.

We present an algorithm that finds such a bicriteria approximate solution. Similar to Algorithm \ref{alg:single_segment}, our algorithm also consists of two steps: in the first step, we find a approximate solution $\hat{S}$ of a weighted cover problem using greedy algorithm; in the second step, we optimally expand the assortment $\hat{S}$ obtained from the previous step subject to some cardinality constraint. 

As suggested in Lemma~\ref{lemma:optimal_expansion}, expected revenue of the optimal expansion of an assortment $\hat{S}$ is closely related to the sum of the preference weights of $\hat{S}$. Therefore we would like the sum of the preference weights $\sum_{i\in \hat{S}}v_i$ to be small. On the other hand, with the cardinality constraint, the cardinality of $\hat{S}$ should not be excessively large. A natural approach of limiting the cardinality of $\hat{S}$ is simply to add a cardinality constraint to the weighted cover problem, but in this case the problem with both covering constraints and cardinality constraint could be challenging to solve. Therefore we consider an alternative approach that limits the cardinality of $\hat{S}$ by a penalty term. Specifically, in the first step, we approximately solve the following weighted cover problem

\begin{equation}
\label{prob:weighted_cover_cardinality}
\begin{aligned}
&\max_{S\subseteq\mathcal{N}}&&\gamma|S|+\sum_{i\in S}v_i\\
&\text{s.t.}&&|S\cap C_k|\geq \ell_k,\ \forall k\in\{1,2,\dots,K\}.
\end{aligned}
\end{equation}
In Problem~\eqref{prob:weighted_cover_cardinality}, the objective function is the sum of preference weights plus the penalty term $\gamma|S|$, where $\gamma>0$ is the penalty parameter for cardinality of $S$. Since the objective function in Problem~\eqref{prob:weighted_cover_cardinality} is still a linear function, Problem~\eqref{prob:weighted_cover_cardinality}~can be approximately solved using greedy algorithm. In our algorithm, we need to run a grid search to find an appropriate penalty parameter $\gamma$. Letting $L_\gamma=\lceil\log_{1+\epsilon}(v_{\max}/v_{\min})\rceil$, the set of grid points for $\gamma$ is given by
\begin{equation*}
\Gamma_\epsilon=\Big\{v_{\min}\cdot (1+\epsilon)^\ell\Big\}_{\ell=0}^{L_\gamma}.
\end{equation*}
In the second step, we expand the assortment $\hat{S}$ optimally such that at most $L$ products are added to the assortment. In other words, we solve the following problem
\begin{equation}\label{prob:optimal_expansion_cardinality}
\bar{S}=\argmax_{S\supseteq\hat{S},\,|S\backslash \hat{S}|\leq L} R(S).
\end{equation}
Since the constraints in Problem~\eqref{prob:optimal_expansion_cardinality} is totally unimodular, Problem~\eqref{prob:optimal_expansion_cardinality}~can be solved in polynomial time using the algorithm in \cite{sumida2021revenue}. The complete algorithm is provided in Algorithm~\ref{alg:daoc_cardinality}.

\begin{algorithm}[!ht]
\caption{Bicriteria Approximation algorithm for~\eqref{prob:daoc_cardinality}}
\SingleSpacedXI
\label{alg:daoc_cardinality}
\begin{algorithmic}
\For{$\gamma\in \Gamma_\epsilon$}
\State Initialize $\hat{S}_\gamma\leftarrow\varnothing$
\While{there exists $k\in\{1,2,\dots,K\}$ such that $|\hat{S}_\gamma\cap C_k|<\ell_k$}
\State Set $c_i=|\{k\in\{1,2,\dots,K\}:\,i\in C_k,\,|C_k\cap \hat{S}_\gamma|<\ell_k\}|$ for all $i\in\mathcal{N}\backslash\hat{S}$
\State Set $i^*=\arg\min_{i\in \mathcal{N}\backslash\hat{S}} (v_i+\gamma)/c_i$, update $\hat{S}_\gamma\leftarrow \hat{S}_\gamma\cup\{i^*\}$
\EndWhile
\If{$|\hat{S}_\gamma|>(2\log K+2)L$}
\State Set $\bar{S}_\gamma=\varnothing$
\Else
\State Set $\bar{S}_\gamma$ as
\begin{equation*}
\bar{S}_{\gamma}=\underset{S\supseteq\hat{S}_{\hat{\gamma}},\,|S\backslash\hat{S}_{\hat{\gamma}}|\leq L}{\mathrm{argmax}} R(S)
\end{equation*}
\EndIf
\EndFor
\State Return $\bar{S}_\gamma$ that maximizes the expected revenue over all $\gamma\in \Gamma_\epsilon$
\end{algorithmic}
\end{algorithm}

In the following theorem, we show that the assortment returned by Algorithm~\ref{alg:daoc_cardinality} satisfies the cardinality constraints up to a factor of $O(\log K)$, and attains an expected revenue at least $\Omega(1/\log K)$ fraction of the optimal value of Problem~\eqref{prob:daoc_cardinality}.

\begin{theorem}
If \eqref{prob:daoc_cardinality} is feasible, then the assortment $S$ returned by Algorithm~\ref{alg:daoc_cardinality} satisfies the covering constraints, satisfies $|S|\leq (2\log K+3)L$, and generates an expected revenue at least $1/((2+\epsilon)\log K+3+\epsilon)$ fraction of the optimal value of \eqref{prob:daoc_cardinality}.

\end{theorem}

\begin{proof}

It is easy to verify that for all $\hat{\gamma}\in\Gamma_\epsilon$, if $\bar{S}_{\hat{\gamma}}\neq \varnothing$, then $\bar{S}_{\hat{\gamma}}$ satisfies the covering constraints, and \mbox{$|\bar{S}_{\hat{\gamma}}|\leq (2\log K+3)L$.} Therefore it suffices to show that if~\eqref{prob:daoc_cardinality}~is feasible, then there exists $\hat{\gamma}\in\Gamma_\epsilon$ such that $R(\bar{S}_{\hat{\gamma}})$ is at least $1/((2+\epsilon)\log K+3+\epsilon)$ fraction of the optimal value of~\eqref{prob:daoc_cardinality}.

Suppose \eqref{prob:daoc_cardinality} is feasible, and let $S^*$ be an optimal solution. It is easy to verify that $S^*\neq \varnothing$, thus $|S^*|>0$. Then we have \mbox{$|S^*|v_{\min}\leq \sum_{i\in S^*}v_i\leq |S^*|v_{\max}$,} thus \mbox{$v_{\min}\leq \sum_{i\in S^*}v_i/|S^*|\leq v_{\max}$.} Therefore there exists $\hat{\gamma}\in\Gamma_\epsilon$ such that $\sum_{i\in S^*}v_i/|S^*|\leq \hat{\gamma}\leq (1+\epsilon)\sum_{i\in S^*}v_i/|S^*|$, which implies 
\begin{equation}\label{eq:condition_gamma}
\sum_{i\in S^*}v_i\leq \hat{\gamma}|S^*|\leq (1+\epsilon)\sum_{i\in S^*}v_i.
\end{equation}

For the rest of the proof, we proceed in two steps. In the first step, we show that \mbox{$|\hat{S}_{\hat{\gamma}}|\leq (2\log K+2)L$.} This step is to ensure that $\bar{S}_{\hat{\gamma}}\neq\varnothing$. In the second step, we show that
\begin{equation*}
R(\bar{S}_{\hat{\gamma}})\geq \dfrac{R(S^*)}{(2+\epsilon)\log K+3+\epsilon}.
\end{equation*}

\noindent\underline{\bf Step 1. $|\hat{S}_{\hat{\gamma}}|\leq (2\log K+2)L$:} Let $\tilde{S}_{\hat{\gamma}}$ be an optimal solution to Problem~\eqref{prob:weighted_cover_cardinality} when $\gamma=\hat{\gamma}$. By \cite{lovasz1975ratio} we have that $\hat{S}_{\hat{\gamma}}$ is a $(\log K+1)$-approximation to Problem~\eqref{prob:weighted_cover_cardinality}. Therefore we get
\begin{equation*}
\sum_{i\in \hat{S}_{\hat{\gamma}}} (v_i+\hat{\gamma})\leq (\log K+1)\cdot \sum_{i\in \tilde{S}_{\hat{\gamma}}} (v_i+\hat{\gamma}).
\end{equation*}
Furthermore, since $S^*$ satisfies the covering constraints, $S^*$ is a feasible solution to Problem~\eqref{prob:weighted_cover_cardinality}, therefore $\sum_{i\in \tilde{S}_{\hat{\gamma}}} (v_i+\hat{\gamma})\leq \sum_{i\in S^*} (v_i+\hat{\gamma})$. In this case, we have
\begin{equation*}
\sum_{i\in \hat{S}_{\hat{\gamma}}} (v_i+\hat{\gamma})\leq (\log K+1)\cdot \sum_{i\in S^*}(v_i+\hat{\gamma}).
\end{equation*}
By \eqref{eq:condition_gamma} we have $\sum_{i\in S^*}v_i\leq \hat{\gamma}|S^*|$, therefore
\begin{equation*}
\hat{\gamma}|\hat{S}_{\hat{\gamma}}|\leq \sum_{i\in \hat{S}_{\hat{\gamma}}} (v_i+\hat{\gamma})\leq (\log K+1)\cdot \sum_{i\in S^*} (v_i+\hat{\gamma})\leq (2\log K+2)\hat{\gamma}|S^*|.
\end{equation*}
Since $S^*$ is feasible to Problem~\eqref{prob:weighted_cover_cardinality}, we have $|S^*|\leq L$, thus $|\hat{S}_{\hat{\gamma}}|\leq (2\log K+2)|S^*|\leq (2\log K+2)L$. 

\noindent\underline{\bf Step 2. $R(\bar{S}_{\hat{\gamma}})\geq {R(S^*)}/{((2+\epsilon)\log K+3+\epsilon)}$:} By~\eqref{eq:condition_gamma} we have $\hat{\gamma}|S^*|\leq (1+\epsilon)\sum_{i\in S^*}v_i$. Therefore
\begin{equation*}
\sum_{i\in \hat{S}_{\hat{\gamma}}}v_i\leq \sum_{i\in \hat{S}_{\hat{\gamma}}}(v_i+\hat{\gamma})\leq (\log K+1)\cdot\sum_{i\in S^*}(v_i+\hat{\gamma})\leq ((2+\epsilon)\log K+2+\epsilon)\cdot \sum_{i\in S^*}v_i.
\end{equation*}
Let $\underline{S}$ be an optimal solution to the following problem
\begin{equation}\label{prob:knapsack_cardinality}
\begin{aligned}
&\max_{S\subseteq\mathcal{N}}&&\sum_{i\in S}r_iv_i\\
&\text{s.t.}&&\sum_{i\in S}v_i\leq \sum_{i\in S^*}v_{i}\\
&&&|S|\leq L.
\end{aligned}
\end{equation}
It is easy to verify that $S^*$ is a feasible solution to~\eqref{prob:knapsack_cardinality}. Then we get $\sum_{i\in \underline{S}}r_iv_i\geq \sum_{i\in S^*}r_iv_i$. Therefore
\begin{align*}
R(\hat{S}_{\hat{\gamma}}\cup \underline{S})=&\dfrac{\sum_{i\in \hat{S}_{\hat{\gamma}}\cup \underline{S}}r_iv_i}{1+\sum_{i\in \hat{S}_{\hat{\gamma}}\cup \underline{S}}v_i}\geq \dfrac{\sum_{i\in \underline{S}}r_iv_i}{1+\sum_{i\in \hat{S}_{\hat{\gamma}}}v_i+\sum_{i\in \underline{S}}v_i}\\
\geq &\dfrac{\sum_{i\in S^*}r_iv_i}{1+((2+\epsilon)\log K+2+\epsilon)\sum_{i\in S^*}v_i+\sum_{i\in S^*}v_i}\\
\geq &\dfrac{\sum_{i\in S^*}r_iv_i}{((2+\epsilon)\log K+3+\epsilon)(1+\sum_{i\in S^*}v_i)}=\dfrac{R(S^*)}{(2+\epsilon)\log K+3+\epsilon}.
\end{align*}
Since $\hat{S}_{\hat{\gamma}}\cup\bar{S}\supseteq \hat{S}_{\hat{\gamma}}$ and $|(\hat{S}_{\hat{\gamma}}\cup \underline{S})\backslash \hat{S}_{\hat{\gamma}}|\leq |\underline{S}|\leq L$, by definition of $\bar{S}_{\hat{\gamma}}$ we have $$R(\bar{S}_{\hat{\gamma}})\geq R(\hat{S}_{\hat{\gamma}}\cup \underline{S})\geq R(S^*)/((2+\epsilon)\log K+3+\epsilon).$$ Finally, since the assortment returned by the algorithm generates an expected revenue greater than or equal to $R(\bar{S}_{\hat{\gamma}})$, we have that the assortment returned by the algorithm generates an expected revenue greater than or equal to $1/((2+\epsilon)\log K+3+\epsilon)$ fraction of the optimal value of~\eqref{prob:daoc_cardinality}.
\end{proof}

}

\section{Randomized Multi-Segment Assortment Optimization}\label{sec:multi_segment_randomized}

In this section, we consider the randomized multi-segment assortment optimization problem with covering constraints. In this problem, the seller customizes assortments over multiple customer segments indexed by $\mathcal{M}=\{1,2,\dots,m\}$. Each customer segment $j\in\mathcal{M}$ is associated with an arrival probability $\theta_j$ and a MNL model with preference weights $v_{ij}$ for all $i\in\mathcal{N}$. Products can be classified into categories $C_1,C_2,\dots,C_K$, and each category $C_k$ is associated with a minimum threshold $\ell_k$ for all $k\in\{1,2,\dots,K\}$. Similar to~\cref{sec:multi_segment}, the expected number of products from category $C_k$ {  offered} to customers should be at least $\ell_k$.

The decision of the seller in the randomized setting is to decide for each customer segment $j \in {\cal M}$ a distribution over assortments $q_j(\cdot)$ to offer. In this setting, the expected number of products from category $C_k$ {  offered} to customers is given by $\sum_{j\in\mathcal{M}}\sum_{S\subseteq\mathcal{N}}\theta_jq_j(S)|S\cap C_k|$, thus the covering constraints can be written as $\sum_{j\in\mathcal{M}}\sum_{S\subseteq\mathcal{N}}\theta_jq_j(S)|S\cap C_k|\geq \ell_k$.
We refer to this problem as~\eqref{prob:randomized_multi_segment}, and its formal definition is provided by
\begin{equation}\tag{m-RAOC}\label{prob:randomized_multi_segment}
\begin{aligned}
&\max_{\{q_j(\cdot)\}_{j\in\mathcal{M}}}&&{\sum_{j\in\mathcal{M}}\sum_{S\subseteq\mathcal{N}}\theta_jR_j(S)q_j(S)}\\
&\text{s.t.}&&{\sum_{j\in\mathcal{M}}\sum_{S\subseteq\mathcal{N}}\theta_j|S\cap C_k|q_j(S)\geq \ell_k,\ \forall k\in\{1,2,\dots,K\},}\\
&&&{\sum_{S\subseteq\mathcal{N}}q_j(S)=1,\ \forall j\in\mathcal{M},}\\
&&&{q_j(S)\geq 0,\ \forall S\subseteq\mathcal{N}.}\\
\end{aligned}
\end{equation}
In the following, we prove that~\eqref{prob:randomized_multi_segment}~can be rewritten as an equivalent linear program. The proof follows similar ideas as in the single segment randomized case discussed in Section \ref{sec:randomized_setting}. Our candidate linear program is given by
\begin{equation}\label{prob:lp_randomized_multi_segment}\tag{m-RAOC-LP}
\begin{aligned}
&\max_{\mathbf{x},\mathbf{y}}&&{\sum_{j\in\mathcal{M}}\sum_{i\in\mathcal{N}}\theta_jr_iv_{ij}x_{ij}}\\
&\text{s.t.}&&{x_{0j}+\sum_{i\in\mathcal{N}}v_{ij}x_{ij}=1,\ \forall j\in\mathcal{M},}\\
&&&{x_{ij}\leq x_{0j},\ \forall i\in\mathcal{N},\,j\in\mathcal{M},}\\
&&&{0\leq y_{ii'j}\leq x_{ij},\ \forall i,i'\in\mathcal{N},\,j\in\mathcal{M},}\\
&&&{y_{ii'j}\leq x_{i'j},\ \forall i,i'\in\mathcal{N},\,j\in\mathcal{M},}\\
&&&{\sum_{i\in C_k}\sum_{j\in\mathcal{M}}\theta_j\left(x_{ij}+\sum_{i'\in\mathcal{N}}v_{i'j}y_{ii'j}\right)\geq \ell_k,\ \forall k\in\{1,2,\dots,K\}.}
\end{aligned}
\end{equation}
Consider an optimal solution $(\mathbf{x}^*,\mathbf{y}^*)$ of \eqref{prob:lp_randomized_multi_segment}. For each $j \in {\cal M}$, we define $\{i_{1j},\dots,i_{nj}\}$ as a permutation of $\{1,\dots,n\}$ such that $x^*_{i_{1j}j}\geq x^*_{i_{2j}j}\geq\dots x^*_{i_{nj}j}$, and $S_{\ell j}$ is denoted as $\{i_{1j},i_{2j},\dots,i_{\ell j}\}$. We construct a distribution over assortments $q_j^*(\cdot)$ for customer segment $j$ by
\begin{equation}
q^*_j(S)=\begin{cases}
\left(1+\sum_{i'=1}^\ell v_{i'j}\right)(x^*_{i_{i'j}j}-x^*_{i_{i'+1,j}j}),\ &\text{if}\ S=S_{\ell j}\ \text{for some}\ \ell\in\{1,2,\dots,n\},\\
1-\sum_{\ell=1}^n q_j(S_{\ell j}),\ &\text{if}\ S=\varnothing,\\
0,\ &\text{otherwise},
\end{cases}\label{eq:solution_recover_multi_segment}
\end{equation}
with the notation $x^*_{i_{n+1,j}j}=0$.
In the following, we show that we can obtain an optimal solution to~\eqref{prob:randomized_multi_segment}~using an optimal solution to~\eqref{prob:lp_randomized_multi_segment}.

\begin{theorem}\label{thm:randomized_multi_segment_lp}
Let $(\mathbf{x}^*,\mathbf{y}^*)$ be the optimal solution to~\eqref{prob:lp_randomized_multi_segment}, and $\{q_j^*(\cdot)\}_{j\in\mathcal{M}}$ is defined in~\eqref{eq:solution_recover_multi_segment}, then $\{q_j^*(\cdot)\}_{j\in\mathcal{M}}$ is the optimal solution to~\eqref{prob:randomized_multi_segment}.
\end{theorem}

To prove~\cref{thm:randomized_multi_segment_lp}, we first prove the following lemma.

\begin{lemma}\label{lemma:nested_support_multi_segment}
There exists an optimal solution $\{q_j^*(\cdot)\}_{j\in\mathcal{M}}$ to~\eqref{prob:randomized_multi_segment} 
such that for every $j\in\mathcal{M}$, there exists a sequence of assortments $\{S_{ij}\}_{i=1}^m$ such that $S_{1j}\supseteq S_{2j}\supseteq \dots \supseteq S_{mj}$ and $q_j^*(S)=0$ if $S\notin \{S_{ij}\}_{i=1}^m$. 
\end{lemma}

\begin{proof}
We rewrite~\eqref{prob:randomized_multi_segment} as follows
\begin{equation}\label{prob:randomized_multi_segment_2}
\begin{aligned}
&\max_{\{q_j(\cdot)\}_{j\in\mathcal{M}},\mathbf{p}}&&{\sum_{j\in\mathcal{M}}\sum_{S\subseteq\mathcal{N}}\theta_jR_j(S)q_j(S)}\\
&\text{s.t.}&&{\sum_{S\ni i}q_j(S)\geq p_{ij},\ \forall i\in\mathcal{N},\,j\in\mathcal{M},}\\
&&&{\sum_{j\in\mathcal{M}}\sum_{i\in C_k}\theta_jp_{ij}\geq \ell_k,\ \forall k\in\{1,2,\dots,K\},}\\
&&&{\sum_{S\subseteq\mathcal{N}}q_j(S)=1,\ \forall j\in\mathcal{M},}\\
&&&{q_j(S)\geq 0,\ \forall S\subseteq\mathcal{N},\,j\in\mathcal{M}.}
\end{aligned}
\end{equation}
Let $(\{q^*_j(\cdot)\}_{j\in\mathcal{M}},\mathbf{p}^*)$ be an optimal solution of \eqref{prob:randomized_multi_segment_2}. We fix $\mathbf{p}$ in~\eqref{prob:randomized_multi_segment_2}~to be $\mathbf{p}^*$ and only optimizes over $\{q_j(\cdot)\}_{j\in\mathcal{M}}$. We derive the following problem for each $j\in\mathcal{M}$
\begin{equation}\label{prob:randomized_multi_segment_3}
\begin{aligned}
&\max_{q_j(\cdot)}&&{\sum_{S\subseteq\mathcal{N}}\theta_jR_j(S)q_j(S)}\\
&\text{s.t.}&&{\sum_{S\ni i}q_j(S)\geq p^*_{ij},\ \forall i\in\mathcal{N},}\\
&&&{\sum_{S\subseteq\mathcal{N}}q_j(S)=1,}\\
&&&{q_j(S)\geq 0,\ \forall S\subseteq\mathcal{N}.}
\end{aligned}
\end{equation}
By Corollary 1 of \cite{lu2023simple}, for each $j\in\mathcal{M}$ there exists an optimal solution $\hat{q}_j(\cdot)$ to~\eqref{prob:randomized_multi_segment_3}~such that the assortments offered with positive probability are nested. Thus $(\{\hat{q}_j(\cdot)\}_{j\in\mathcal{M}},\mathbf{p}^*)$ is an optimal solution to~\eqref{prob:randomized_multi_segment_2}, and $\{\hat{q}_j(\cdot)\}_{j\in\mathcal{M}}$ is an optimal solution to~\eqref{prob:randomized_multi_segment} that satisfies the conditions stated in the lemma.
\end{proof}

Building on Lemma~\ref{lemma:nested_support_multi_segment}, we give a proof of~\cref{thm:randomized_multi_segment_lp}.

\begin{proof}[Proof of~\cref{thm:randomized_multi_segment_lp}]
We prove the theorem in two steps: first we prove that the optimal objective value of~\eqref{prob:randomized_multi_segment}~is smaller or equal to that of~\eqref{prob:lp_randomized_multi_segment}, then we prove that the optimal objective value of~\eqref{prob:randomized_multi_segment}~is greater or equal to that of~\eqref{prob:lp_randomized_multi_segment}.

\underline{\textbf{Step 1.} \eqref{prob:randomized_multi_segment}$\leq$\eqref{prob:lp_randomized_multi_segment}:} By Lemma~\ref{lemma:nested_support_multi_segment}, there exists an optimal solution $\{q^*_j(\cdot)\}_{j\in\mathcal{M}}$ such that for every $j\in\mathcal{M}$, $q^*_j(\cdot)$ is supported on a sequence of nested assortment $S_{1j}\supseteq S_{2j}\supseteq \dots \supseteq S_{mj}$. For any $i\in\mathcal{N}$ and $j\in\mathcal{M}$, we define $x^*_{ij}$ as
\begin{equation*}
x^*_{ij}=\sum_{\ell=1}^m q^*_j(S_{\ell j})\dfrac{\bm{1}[i\in S_{\ell j}]}{1+\sum_{s\in S_{\ell j}}v_{sj}}.
\end{equation*}
We define $y^*_{ii'j}=\min\{x^*_{ij},x^*_{i'j}\}$, and define $x_{0j}$ as
\begin{equation*}
x^*_{0j}=\sum_{\ell=1}^m q^*_j(S_{\ell j})\dfrac{1}{1+\sum_{s\in S_{\ell j}}v_{sj}}.
\end{equation*}
Using the same argument as in Step 1 of the proof of \cref{thm:lp_randomized_single_segment} we have
\begin{align*}
&x^*_{ij}\leq x^*_{0j},\ \forall i\in\mathcal{N},\, \forall j\in\mathcal{M},\\
&x^*_{0j}+\sum_{i\in\mathcal{N}}v_{ij}x^*_{ij}=1,\ \forall j\in\mathcal{M},\\
&\sum_{j\in\mathcal{M}}\sum_{i\in\mathcal{N}}\theta_jr_iv_{ij}x^*_{ij}=\sum_{j\in\mathcal{M}}\sum_{S\subseteq\mathcal{N}}\theta_jR_j(S)q^*_j(S),\\
&x^*_{ij}+\sum_{i'\in\mathcal{N}}v_{i'j}y^*_{ii'j}=\sum_{S\subseteq\mathcal{N}}q^*_j(S)\cdot \bm{1}[i\in S],\ \forall i\in\mathcal{N},\, \forall j\in\mathcal{M}.
\end{align*}
By the last equation we have that for any $k\in\{1,2,\dots,K\}$
\begin{equation*}
\sum_{j\in\mathcal{M}}\sum_{i\in C_k}\theta_j\left(x^*_{ij}+\sum_{i'\in\mathcal{N}}v_{i'j}y^*_{ii'j}\right)=\sum_{j\in\mathcal{M}}\sum_{S\subseteq\mathcal{N}}\theta_j|S\cap C_k|q^*_j(S)\geq \ell_k.
\end{equation*}
Therefore we conclude that $(\mathbf{x}^*,\mathbf{y}^*)$ is feasible to~\eqref{prob:lp_randomized_multi_segment}, and the optimal objective value of~\eqref{prob:lp_randomized_multi_segment}~is greater or equal to~\eqref{prob:randomized_multi_segment}.

\underline{\textbf{Step 2.} \eqref{prob:randomized_multi_segment}$\geq$\eqref{prob:lp_randomized_multi_segment}}
Let $(\mathbf{x}^*,\mathbf{y}^*)$ be an optimal solution to~\eqref{prob:lp_randomized_multi_segment}. For each $j\in\mathcal{M}$, we define $q^*_j(\cdot)$ using~\eqref{eq:solution_recover_multi_segment}. Since setting $y^*_{ii'j}=\min\{x^*_{ij},x^*_{i'j}\}$ does not change the objective value and feasibility, we assume without loss of generality that $y^*_{ii'j}=\min\{x^*_{ij},x^*_{i'j}\}$. Using the same argument as in Step 2 of proof of~\cref{thm:lp_randomized_single_segment}, we have that $q^*_j(\cdot)$ is a probability mass function for each $j\in\mathcal{M}$, and
\begin{align*}
&\sum_{\ell=1}^n q^*_j(S_{\ell j})\cdot \bm{1}[i\in S_{\ell j}]=x^*_{ij}+\sum_{i'\in\mathcal{N}}v_{i'j}y^*_{ii'j},\ \forall i\in\mathcal{N},\, \forall j\in\mathcal{M},\\
&\sum_{j\in\mathcal{M}}\sum_{i\in\mathcal{N}}\theta_jr_iv_{ij}x^*_{ij}=\sum_{j\in\mathcal{M}}\sum_{S\subseteq\mathcal{N}}\theta_jR_j(S)q^*_j(S).
\end{align*}
By the first equation above we have that for any $k\in\{1,2,\dots,K\}$
\begin{equation*}
\sum_{j\in\mathcal{M}}\sum_{S\subseteq\mathcal{N}}\theta_jq^*_j(S)|S\cap C_k|=\sum_{j\in\mathcal{M}}\theta_j\sum_{i\in\mathcal{N}}\left(x^*_{ij}+\sum_{i'\in\mathcal{N}}v_{i'j}y^*_{ii'j}\right)\geq\ell_k.
\end{equation*}
Thus we conclude that $\{q_j(\cdot)\}_{j\in\mathcal{M}}$ is feasible to~\eqref{prob:randomized_multi_segment}, and the optimal objective value of~\eqref{prob:randomized_multi_segment}~is greater or equal to that of~\eqref{prob:lp_randomized_multi_segment}. 

Combining the  two steps we conclude that the optimal objective value of~\eqref{prob:randomized_multi_segment}~equals to that of~\eqref{prob:lp_randomized_multi_segment}. %With an optimal solution $(\mathbf{x},\mathbf{y})$ to~\eqref{prob:lp_randomized_multi_segment}, an optimal solution to~\eqref{prob:randomized_multi_segment} is given as~\eqref{eq:solution_recover_multi_segment}.
\end{proof}

{ 

\section{Assortment Customization over Finite Number of Customers}\label{sec:assortment_customization}

In this section, we consider the assortment customization problem over $T$ homogeneous customers. We assume that there are $T$ customers indexed by $1,2,\dots,T$, and the choice of each customer follows the same MNL model. The decision of the seller is to customize an assortment $S_t$ to each customer $t\in\{1,2,\dots,T\}$. We assume that products are categorized into $K$ categories $C_1,\dots,C_K$, and each category $C_k$ is associated with a minimum threshold $\ell_k$. In this problem, we require that for every $k\in\{1,2,\dots,K\}$, the total number of times products in $C_k$ are offered to customers exceeds $\ell_k$, i.e., $\sum_{t=1}^T|S_t\cap C_k|\geq \ell_k$. The formal definition of the problem is given by
\begin{equation}
\label{prob:assortment_customization}
\begin{aligned}
&\max_{\{S_t\}_{t=1}^T}&&\sum_{t=1}^TR(S_t)\\
&\text{s.t.}&& \sum_{t=1}^T|S_t\cap C_k|\geq \ell_k,\ \forall k\in\{1,2,\dots,K\}.
\end{aligned}
\end{equation}

We present a rounding-based approximation algorithm for Problem~\eqref{prob:assortment_customization}. We first solve the following problem
\begin{equation}
\label{prob:relaxed_assortment_customization}
\begin{aligned}
&\max_{q(\cdot)}&&\sum_{S\subseteq\mathcal{N}} q(S)R(S)\\
&\text{s.t.}&&\sum_{S\subseteq\mathcal{N}}q(S)=1,\\
&&&q(S)\geq 0,\ \forall S\subseteq\mathcal{N},\\
&&&\sum_{S\subseteq\mathcal{N}}q(S)|S\cap C_k|\geq \ell_k/T,\ \forall k\in\{1,2,\dots,K\}.
\end{aligned}
\end{equation}
By Theorem~\ref{thm:lp_randomized_single_segment}, Problem~\eqref{prob:relaxed_assortment_customization}~can be solved in polynomial time via solving an equivalent linear program. Furthermore, by Proposition~\ref{lemma:nested_support}, there exists an optimal solution $q^*$ to Problem~\eqref{prob:relaxed_assortment_customization}~that only offers a nested sequence of assortments $S^{(1)}\supseteq S^{(2)}\supseteq\dots\supseteq S^{(n_0)}$ with positive probability for some $n_0\in\{1,2,\dots,n\}$. 

Based on $q^*$, we construct a feasible solution $\{\tilde{S}_{t}\}_{t=1}^T$ to Problem~\eqref{prob:assortment_customization} as follows. For any $i\in \{1,2,\dots,n_0\}$, we define $\tau_i$ as
\begin{equation*}
\tau_i=\left\lceil \sum_{j=1}^i Tq^*(S^{(j)})\right\rceil.
\end{equation*}
We further define $\tau_0=0$. Then for any $t\in\{1,2,\dots,T\}$, we define $\tilde{S}_t=S^{(i)}$ if $\tau_{i-1}<t\leq \tau_{i}$.

We use the following theorem to show that $\{\tilde{S}_t\}_{t=1}^T$ is feasible to Problem~\eqref{prob:assortment_customization}, and the solution is asymptotically optimal as the number of customers $T$ gets large.

\begin{theorem}\label{thm:approximation_customization}
The sequence of assortments $\{\tilde{S}_{t}\}$ is feasible to Problem~\eqref{prob:assortment_customization}, and attains at least $1-(nv_{\max}+1)/Tv_{\min}$ fraction of the optimal solution of Problem~\eqref{prob:assortment_customization}.
\end{theorem}

Theorem~\ref{thm:approximation_customization}~states that the approximation ratio of $\{\tilde{S}_t\}_{t=1}^T$ converges to $1$ as the number of customers $T$ increases. This indicates that $\{\tilde{S}_t\}_{t=1}^T$ is asymptotically optimal as $T$ gets large. Since in practice the number of customers $T$ the seller customize assortments over is usually large, the asymptotic optimality of $\{\tilde{S}_t\}_{t=1}^T$ suggests that the solution can perform well in practice.

To prove Theorem~\ref{thm:approximation_customization}, we first prove the following lemma, which shows that the optimality gap of $\{\tilde{S}_t\}_{t=1}^T$ is uniformly bounded by $r_{\max}$.

\begin{lemma}\label{lemma:integrality_gap}
The sequence of assortments $\{\tilde{S}_{t}\}_{t=1}^T$ is a feasible solution to Problem~\eqref{prob:assortment_customization}, and its expected revenue differs from the optimal value of Problem~\eqref{prob:assortment_customization}~by at most $r_{\max}$.
\end{lemma}

\begin{proof}
We first verify the feasibility of $\{\tilde{S}_t\}_{t=1}^T$. For notational brevity we define $S^{(n_0+1)}=\varnothing$. For any $k\in\{1,2,\dots,K\}$, we have
\begin{align*}
&\sum_{t=1}^T|\tilde{S}_t\cap C_k|=\sum_{i=1}^{n_0} (\tau_i-\tau_{i-1})|S^{(i)}\cap C_k|=\sum_{i=1}^{n_0} \tau_i|S^{(i)}\cap C_k|-\sum_{i=1}^{n_0}\tau_{i-1}|S^{(i)}\cap C_k|\\
&=\sum_{i=1}^{n_0} \tau_i|S^{(i)}\cap C_k|-\sum_{i=1}^{n_0}\tau_i|S^{(i+1)}\cap C_k|=\sum_{i=1}^{n_0} \tau_i(|S^{(i)}\cap C_k|-|S^{(i+1)}\cap C_k|)\\
&\geq\sum_{i=1}^{n_0} (|S^{(i)}\cap C_k|-|S^{(i+1)}\cap C_k|)\sum_{j=1}^i Tq^*(S^{(j)}).
\end{align*}
Here the inequality is due to $S^{(i)}\supseteq S^{(i+1)}$ and $\tau_i\geq \sum_{j=1}^i Tq(S^{(i)})$. We further have
\begin{align*}
&\sum_{i=1}^{n_0} (|S^{(i)}\cap C_k|-|S^{(i+1)}\cap C_k|)\sum_{j=1}^i Tq^*(S^{(j)})=T\sum_{j=1}^{n_0}\sum_{i=j}^{n_0} (|S^{(i)}\cap C_k|-|S^{(i+1)}\cap C_k|)q^*(S^{(j)})\\
&=T\sum_{j=1}^{n_0} |S^{(j)}\cap C_k|q^*(S^{(j)})=T\sum_{S\subseteq\mathcal{N}}|S\cap C_k|q^*(S)\geq T\cdot \ell_k/T=\ell_k.
\end{align*}
Therefore we have $\sum_{t=1}^T|\tilde{S}_t\cap C_k|\geq \ell_k$ for all $k\in\{1,2,\dots,K\}$, thus $\{\tilde{S}_t\}_{t=1}^T$ is feasible to Problem~\eqref{prob:assortment_customization}.

Then we prove an upper bound on the optimality gap of $\{\tilde{S}_t\}_{t=1}^T$. We first prove that the optimal value of Problem~\eqref{prob:assortment_customization}~is upper bounded by $T$ times that of Problem~\eqref{prob:relaxed_assortment_customization}. Let $\{S^*_t\}_{t=1}^T$ be an optimal solution to Problem~\eqref{prob:assortment_customization}. For any $S\subseteq\mathcal{N}$, we define $n(S)$ as the number of times assortment $S$ is offered in $\{S_t^*\}_{t=1}^T$, i.e.,
\begin{equation*}
n(S)=|\{t\in\{1,2,\dots,T\}:S_t^*=S\}|.
\end{equation*}
We further define $\tilde{q}(S)=n(S)/T$ for all $S\subseteq\mathcal{N}$. Since by definition $\sum_{S\subseteq\mathcal{N}}n(S)=T$, we have \mbox{$\sum_{S\subseteq\mathcal{N}}\tilde{q}(S)=1$.} For any $k\in\{1,2,\dots,K\}$, we have
\begin{equation*}
\sum_{S\subseteq\mathcal{N}}\tilde{q}(S)|S\cap C_k|=\sum_{S\subseteq\mathcal{N}}n(S)|S\cap C_k|/T=\sum_{t=1}^T |S_k^*\cap C_k|/T\geq \ell_k/T.
\end{equation*}
Therefore $\tilde{q}$ is a feasible solution to Problem~\eqref{prob:relaxed_assortment_customization}. Since $q^*$ is an optimal solution to Problem~\eqref{prob:relaxed_assortment_customization}, we have
\begin{equation}\label{eq:customization_upper_bound}
\sum_{t=1}^T R(S_t^*)=\sum_{S\subseteq\mathcal{N}}n(S)R(S)=T\sum_{S\subseteq\mathcal{N}}\tilde{q}(S)R(S)\leq T\sum_{S\subseteq\mathcal{N}} q^*(S)R(S).
\end{equation}
We also have
\begin{align*}
&\sum_{t=1}^T R(\tilde{S}_t)=\sum_{i=1}^{n_0} R(S^{(i)})(\tau_{i}-\tau_{i-1})=\sum_{i=1}^{n_0} \tau_iR(S^{(i)})-\sum_{i=1}^{n_0} \tau_{i-1}R(S^{(i)})\\
&=\sum_{i=1}^{n_0} \tau_iR(S^{(i)})-\sum_{i=1}^{n_0-1}\tau_iR(S^{(i+1)})=\tau_{n_0}R(S^{(n_0)})-\sum_{i=1}^{n_0-1}\tau_i(R(S^{(i+1)})-R(S^{(i)}))\\
&=TR(S^{(n_0)})-\sum_{i=1}^{n_0-1}\tau_i(R(S^{(i+1)})-R(S^{(i)})).
\end{align*}
Here the last equality is due to $\sum_{i=1}^{n_0} q^*(S^{(i)})=1$, thus $\tau_n=\lceil T\sum_{i=1}^{n_0} q^*(S^{(i)})\rceil=T$. Similarly we get
\begin{align*}
T\sum_{S\subseteq\mathcal{N}}q^*(S)R(S)=&T\sum_{i=1}^{n_0} q^*(S^{(i)})R(S^{(i)})=T\sum_{i=1}^{n_0} R(S^{(i)})\Big(\sum_{j=1}^i q^*(S^{(j)})-\sum_{j=1}^{i-1}q^*(S^{(j)})\Big)\\
=&T\sum_{i=1}^{n_0} R(S^{(i)})\sum_{j=1}^i q^*(S^{(j)})-T\sum_{i=2}^{n_0} R(S^{(i)})\sum_{j=1}^{i-1}q^*(S^{(j)})\\
=&T\sum_{i=1}^{n_0} R(S^{(i)})\sum_{j=1}^i q^*(S^{(j)})-T\sum_{i=1}^{n_0-1} R(S^{(i+1)})\sum_{j=1}^{i}q^*(S^{(j)})\\
=&TR(S^{(n_0)})\sum_{j=1}^{n}q^*(S^{(j)})-\sum_{i=1}^{n_0-1}(R(S^{(i+1)})-R^{(i)})\cdot \sum_{j=1}^i Tq^*(S^{(j)}).
\end{align*}
Therefore we have
\begin{equation}\label{eq:optimality_gap_bound}
T\sum_{S\subseteq\mathcal{N}}q^*(S)R(S)-\sum_{t=1}^T R(\tilde{S}_t)=\sum_{i=1}^{n_0-1}(R(S^{(i+1)})-R(S^{(i)}))\Big(\tau_{i}-\sum_{j=1}^i q^*(S^{(j)})\Big).
\end{equation}

Then we prove that $R(S^{(i+1)})\geq R(S^{(i)})$ for all $i\in\{1,2,\dots,n_0-1\}$. Suppose there exists \mbox{$i_0\in\{1,2,\dots,n_0-1\}$} such that $R(S^{(i_0+1)})<R(S^{(i_0)})$. Then we define a distribution over assortments $\hat{q}$ as $\hat{q}(S^{(i_0)})=q^*(S^{(i_0)})+q^*(S^{(i_0+1)})$, $\hat{q}(S^{(i_0+1)})=0$, and $\hat{q}(S)=q^*(S)$ for all $S\neq S^{(i_0)}, S^{(i_0+1)}$. It is easy to verify that $\hat{q}$ is a valid probability mass function over assortments. Since in the definition of $\hat{q}$, we are replacing $S^{(i_0+1)}$ in $q^*$ with a larger assortment $S^{(i_0)}$, $\hat{q}$ should still satisfy the covering constraints in Problem~\eqref{prob:relaxed_assortment_customization}. Therefore $\hat{q}$ is also feasible to Problem~\eqref{prob:relaxed_assortment_customization}. We further have
\begin{equation*}
\begin{split}
&\sum_{S\subseteq\mathcal{N}}\hat{q}(S)R(S)-\sum_{S\subseteq\mathcal{N}}q^*(S)R(S)\\
&=(q^*(S^{(i_0)})+q^*(S^{(i_0+1)}))R(S^{(i_0)})-q^*(S^{(i_0)})R(S^{(i_0)})-q^*(S^{(i_0+1)})R(S^{(i_0+1)})\\
&=q^*(S^{(i_0+1)})(R(S^{(i_0)})-R(S^{(i_0+1)}))>0.
\end{split}
\end{equation*}
This contradicts with the optimality of $q^*$. Therefore $R(S^{(i+1)})\geq R(S^{(i)})$ for all \mbox{$i\in\{1,\dots,n_0-1\}$.}

By definition we have $0\leq \tau_i-T\sum_{j=1}^i q^*(S^{(j)})\leq 1$, and we have proven that $R(S^{(i+1)})\geq R(S^{(i)})$ for all $i\in\{1,2,\dots,n_0-1\}$. Therefore by~\eqref{eq:optimality_gap_bound}~we have
\begin{equation}\label{eq:customization_upper_bound_2}
T\sum_{S\subseteq\mathcal{N}}q^*(S)R(S)-\sum_{t=1}^T R(\tilde{S}_t)\leq \sum_{i=1}^{n_0-1}(R(S^{(i+1)})-R(S^{(i)}))=R(S^{(n_0)})-R(S^{(1)})\leq R(S^{(n_0)})\leq r_{\max}.
\end{equation}
Combining \eqref{eq:customization_upper_bound} and \eqref{eq:customization_upper_bound_2} we have
\begin{equation*}
\sum_{t=1}^TR(S_t^*)-\sum_{t=1}^T R(\tilde{S}_t)\leq r_{\max}.
\end{equation*}
Therefore we conclude that the optimality gap of $\{\tilde{S}_t\}_{t=1}^T$ is upper bounded by $r_{\max}$.
\end{proof}

With Lemma~\ref{lemma:integrality_gap}, we are able to finish the proof of Theorem~\ref{thm:approximation_customization}.

\begin{proof}[Proof of Theorem~\ref{thm:approximation_customization}]
Assume without loss of generality that product $1$ is the product with the highest revenue. Since under the MNL model, adding the product with the highest revenue to the assortment always increases the expected revenue and does not violate the covering constraints, we have $1\in S_t^*$ for all $t\in\{1,2,\dots,T\}$. Therefore
\begin{equation*}
\sum_{t=1}^T R(S_t^*)\geq \sum_{t=1}^T r_{\max}\phi(1,S_t^*)\geq Tr_{\max}\dfrac{v_{\min}}{1+nv_{\max}}.
\end{equation*}
Here the second inequality is because for any $i\in S$, $\phi(i,S)\geq v_{\min}/(1+nv_{\max})$. We have proven in Lemma~\ref{lemma:integrality_gap}~that $\sum_{t=1}^T R(\tilde{S}_t)\geq \sum_{t=1}^T R(S_t^*)-r_{\max}$, therefore we have
\begin{equation*}
\dfrac{\sum_{t=1}^T R(S_t^*)-\sum_{t=1}^T R(\tilde{S}_t)}{\sum_{t=1}^T R(S_t^*)}\leq \dfrac{r_{\max}}{Tr_{\max}v_{\min}/(1+nv_{\max})}=\dfrac{1+nv_{\max}}{Tv_{\min}}.
\end{equation*}
This further implies
\begin{equation*}
\dfrac{\sum_{t=1}^T R(\tilde{S}_t)}{\sum_{t=1}^T R(S_t^*)}\geq 1-\dfrac{1+nv_{\max}}{Tv_{\min}},
\end{equation*}
and we conclude that $\{\tilde{S}_t\}_{t=1}^T$ is a $1-(1+nv_{\max})/Tv_{\min}$-approximation to Problem~\eqref{prob:assortment_customization}.
\end{proof}

}

\section{Proof of \cref{thm:hardness_multi_segment}} \label{sec:thm:hardness_multi_segment}

We prove the theorem by a reduction from the maximum independent set problem. We start by presenting the maximum independent set problem.

\noindent\textbf{Maximum independent set.} Given a graph $G=(V,E)$ with vertex set indexed by $V=\{1,2,\dots,n\}$. We say a subset of vertices $S\subseteq V$ is an independent set if any two vertices in $S$ are not adjacent to each other. The goal is to find an independent set in $G$ with maximum size. The problem is NP-hard to approximate within a factor of $\Omega(1/n^{1-\epsilon})$ for all $\epsilon>0$ unless $NP\subseteq BPP$.
 
 Consider any instance $G=(V,E)$ of the maximum independent set problem. We index $V$ by $\{1,2,\dots,m\}$. We construct an instance of~\eqref{prob:deterministic_multi_segment} as follows. We index the set of customer segments by $\mathcal{M}=\{1,2,\dots,m\}$ and the set of products by $\mathcal{N}=\{1,2,\dots,m+1\}$. Revenue $r_i$ and preference weights $v_{ij}$ are defined as
\begin{align*}
		&r_{m+1}=2,\ r_i=m^{-2}/2,\ \forall i\in\{1,2,\dots,m\},\\
		&v_{m+1,j}=1,\ v_{jj}=4m^2,\ v_{ij}=0,\ \forall j\in\{1,2,\dots,m\},\ i\neq j.
	\end{align*}
We set $K=|E|$, and denote $E=\{(i_k,j_k)\}_{k=1}^K$. We set the product categories as follows: $C_k=\{i_k,j_k\}$, and the thresholds $\ell_k=2-1/m$ for all $k\in\{1,2,\dots,K\}$. We further set the arrival probabilities  as $\theta_j=1/m$ for all $j\in\{1,2,\dots,m\}$.

% In this setting, if $m+1 \in S$ and $j\in S$, then
% \begin{equation*}
% 	R_j(S)=\dfrac{2+(m^{-2}/2)\cdot 4m^2}{1+1+4m^2}\leq \dfrac{1}{m^2}.
% \end{equation*}
% Otherwise, if $m+1 \in S$ and $j \notin S$, then $R_j(S)=1$.
		
For any independent set $A\subseteq V$, we set $S_j=\mathcal{N}\backslash\{j\}$ if $j\in A$, and $S_j=\mathcal{N}$ if $j\notin A$. We claim that $\{S_j\}_{j=1}^m$ is a feasible solution to \eqref{prob:deterministic_multi_segment} and $\sum_{j=1}^m \theta_jR_j(S_j)\geq |A|/m$.
		
To prove feasibility of $\{S_j\}_{j=1}^m$, it suffices to prove that for all $(i,i')\in E$,
\begin{equation*}
	\sum_{j=1}^m \bm{1}[i\in S_j]+\sum_{j=1}^m\bm{1}[i'\in S_j]\geq 2m-1.
\end{equation*}
Since $i\in S_j$ for all $j\neq i$ and $i'\in S_j$ for all $j\neq i'$, the inequality above is equivalent to \mbox{$\bm{1}[i\in S_i]+\bm{1}[i'\in S_{i'}]\geq 1$}, or equivalently, either $i\in S_i$ or $i'\in S_{i'}$. Since by definition $A$ is an independent set, and $(i,i') \in E$, then at least one of $i$ and $i'$ does not belong to $A$, which implies that $i\in S_i$ or $i'\in S_{i'}$ and therefore
%$A^c=\{i\in V:i\in S_i\}$ is a vertex cover, therefore for any $(i,i')\in E$, either $i\in S_i$ or $i'\in S_{i'}$, thus 
$\{S_j\}_{j=1}^m$ is a feasible solution to~\eqref{prob:deterministic_multi_segment}.
		
Furthermore, we have
\begin{equation*}
	\sum_{j=1}^n\theta_j R_j(S_j)\geq \sum_{j\in A}R_j(S_j)/m=  \sum_{j\in A}R_j( {\cal N} \setminus j)/m =   |A|/m.
\end{equation*}
%therefore $\sum_{j=1}^m R_j(S_j)\geq |A|$.
		
Conversely, we claim that for any feasible solution $\{S_j\}_{j=1}^m$ to~\eqref{prob:deterministic_multi_segment}, there exists an independent set $A\subseteq V$ such that $|A|\geq\lfloor m\sum_{j=1}^m \theta_jR_j(S_j)\rfloor$. Consider any feasible solution $\{S_j\}_{j=1}^m$ to~\eqref{prob:deterministic_multi_segment}. We set $A=\{i\in V:i\notin S_i\}$. %We claim that $A$ is an independent set and $|A|\geq \lfloor m\sum_{j=1}^m\theta_j R_j(S_j)\rfloor$.
		
First we prove that $A$ is an independent set. Since $\{S_j\}_{j=1}^m$ is a feasible solution, for every $(i,i')\in E$,
\begin{equation*}
	\sum_{j=1}^m \bm{1}[i\in S_j]+\sum_{j=1}^m \bm{1}[i'\in S_j]\geq 2m-1,
\end{equation*}
which implies that $\bm{1}[i\in S_i]+\bm{1}[i\in S_{i'}]\geq 1$. In other words, for every $(i,i')\in E$, we can't have $i \notin S_i$ and $i' \notin S_{i'}$ at the same time, which means that at least one of $i$ and $i'$ does not belong to $A$,
thus $A$ is an independent set.
		
We further prove that $|A|\geq \lfloor m\sum_{j=1}^m \theta_jR_j(S_j)\rfloor$.
For $j\in A$, we have $R_j(S_j)=1$ if $m+1\in S_j$, and $R_j(S_j)=0$ if $m+1\notin S_j$. Thus $R_j(S_j)\leq 1$ for all $j\in A$. 

For $j\notin A$, if $m+1\notin S_j$, we have $R_j(S_j)<m^{-2}/2<m^{-2}$ because  the revenue of all products except $m+1$ is equal to $m^{-2}/2$.  If $m+1\in S_j$, then we have
\begin{equation*}
R_j(S_j)=\dfrac{r_{m+1}v_{m+1}+r_jv_j}{1+v_{m+1}+v_j}=\dfrac{2+m^{-2}/2\cdot 4m^2}{1+1+4m^2}\leq \dfrac{1}{m^2}.
\end{equation*}
Therefore
\begin{equation*}
	m\sum_{j=1}^m \theta_jR_j(S_j)=\sum_{j\in A}R_j(S_j)+\sum_{j\notin A}R_j(S_j)\leq|A|+m\cdot\dfrac{1}{m^2}=|A|+\dfrac{1}{m},
\end{equation*}
which implies $|A|\geq \lfloor m\sum_{j=1}^m R_j(S_j)\rfloor$.
		
Therefore, if $A^*$ is the optimal independent set and $R^*$ is the optimal value of~\eqref{prob:deterministic_multi_segment}, then we get $\lfloor mR^*\rfloor\leq |A^*|\leq mR^*$. Consequently, an $\alpha$-approximation to~\eqref{prob:deterministic_multi_segment}~implies a $\Omega(\alpha)$-approximation to the maximum independent set problem.
By~\cite{feige1996interactive}, it is NP-hard to approximate the maximum set cover problem within a factor of $\Omega(1/m^{1-\epsilon})$ for all $\epsilon>0$, it is also NP-hard to approximate~\eqref{prob:deterministic_multi_segment}~within a factor of $\Omega(1/m^{1-\epsilon})$ for all $\epsilon>0$.

{ 

\section{Numerical Experiments on Synthetic Instances}\label{sec:numerical_synthetic}

In this section, we consider conduct numerical experiments using synthetic instances to test the performance of our algorithm for \eqref{prob:deterministic_single_segment}.

\noindent\underline{\bf Experimental Setup:} We use the following steps to generate our instances. In all of our instances, we set the number of product to be $n=200$. The revenue $r_i$ of each product $i$ is drawn from an exponential distribution with mean $1$, and the preference weight $v_i$ of each product $i$ is drawn from a uniform distribution over $[1,5]$. To generate covering constraints, we first randomly generate $K=3K_0$ categories, where we vary $K_0$ in the experiments. The categories are divided into three types, each type containing $K_0$ categories. The first type of categories are indexed by $\{1,2,\dots,K_0\}$, and every product is added to each type-one category with probability $\alpha$. The second type of categories are indexed by $\{K_0+1,K_0+2,\dots,2K_0\}$. Every product whose revenue is higher than median revenue is added to each type-two category with probability $\alpha$, while products with revenue lower than median revenue are not added to type-two categories. The third type of categories are indexed by $\{2K_0+1,2K_0+2,\dots,3K_0\}$. Every product whose revenue is lower than median revenue is added to each type-three category with probability $\alpha$, while products with revenue higher than median revenue are not added to type-three categories. We vary the parameter $\alpha$ in the experiments. Here we generate type-two and type-three categories because we would like products to be categorized based on revenue, i.e., some categories consist solely of high-revenue products while others consist solely of low-revenue products. Then we construct minimum thresholds $\ell_k$ for each category $C_k$, given by $\ell_k=\lceil \beta U_k|C_k|\rceil$, where $U_k$ is drawn from a uniform distribution over $[0,1]$ for all $k\in\{1,2,\dots,K\}$, and we also vary the parameter $0<\beta<1$ in the experiments.

\noindent\underline{\bf Benchmark Heuristics:} We compare the performance of our algorithm with the following two heuristics.

\noindent\textbf{Heuristic 1:} Let $\tilde{S}_k$ be the set of $\ell_k$ products from $C_k$ with the highest revenues. The heuristic first finds the unconstrained optimal assortment $S_{\text{OPT}}$, then returns $\bigcup_{k=1}^K \tilde{S}_k\cup S_{\text{OPT}}$.

\noindent\textbf{Heuristic 2:} The heuristic first finds $\tilde{S}=\bigcup_{k=1}^K \tilde{S}_k$, where $\tilde{S}_k$ is defined in Heuristic 1 above, and then returns an optimal expansion of $\tilde{S}$.

We measure the performance of our algorithm and the two heuristics by the approximation ratio compared to the exact optimal value of \eqref{prob:deterministic_single_segment}. In our experiments, the optimal value of \eqref{prob:deterministic_single_segment} is computed by solving the following equivalent integer program

\begin{equation*}
\begin{aligned}
&\max_{\mathbf{x}\in\mathbb{R}^{n+1}}&&\sum_{i\in\mathcal{N}}r_iv_ix_i\\
&\text{s.t.}&& x_0+\sum_{i\in\mathcal{N}}v_ix_i=1,\\
&&& x_i\in\{0,x_0\},\ \forall i\in\mathcal{N},\\
&&& \sum_{i\in C_k}x_i\geq \ell_kx_0,\ \forall k\in\{1,2,\dots,K\}.
\end{aligned}
\end{equation*}

Varying $K_0\in\{10,20\}$, $\alpha\in\{0.2,0.4,0.6\}$, $\beta\in\{0.2,0.5\}$, we obtain $12$ parameter configurations. Under each parameter configuration, we randomly generate $100$ instances, and compute the mean approximation ratios of our algorithm as well as two heuristics under the generated instances.

\begin{table}[!ht]
\centering
\begin{tabular}{|c|c c c||c|c c c|}
\hline
$(K_0,\alpha,\beta)$&ALG&Heur.1&Heur.2&$(K_0,\alpha,\beta)$&ALG&Heur.1&Heur.2\\
\hline
$(10,0.2,0.2)$&0.893&0.774&0.787&$(10,0.2,0.5)$&0.900&0.791&0.797\\
$(10,0.4,0.2)$&0.886&0.778&0.780&$(10,0.4,0.5)$&0.897&0.787&0.788\\
$(10,0.6,0.2)$&0.884&0.788&0.789&$(10,0.6,0.5)$&0.902&0.794&0.794\\
\hline
$(20,0.2,0.2)$&0.887&0.767&0.770&$(20,0.6,0.5)$&0.900&0.779&0.780\\
$(20,0.4,0.2)$&0.883&0.780&0.779&$(20,0.4,0.5)$&0.899&0.779&0.779\\
$(20,0.6,0.2)$&0.885&0.790&0.790&$(20,0.6,0.5)$&0.902&0.786&0.786\\
\hline
\end{tabular}
\caption{Mean approximation ratios of our algorithm and two heuristics under synthetic instances}
\label{table:mean_approx_ratio}
\end{table}

\underline{\textbf{Results:}} We present the mean approximation ratios of our algorithm and two heuristics under the synthetic instances in Table~\ref{table:mean_approx_ratio}. From Table~\ref{table:mean_approx_ratio}~one can see that in all tested parameter settings, the mean approximation ratio of our algorithm is around $0.9$, while the corresponding approximation ratio guarantees established in Theorem~\ref{thm:single_segment} are $1/(\log 30+2)\approx 0.18$ when $K=30$, and $1/(\log 60+2)\approx 0.16$ when $K=60$. This suggests that under the tested instances, Algorithm~\ref{alg:single_segment}~can often achieve a much better approximation ratio than its theoretical approximation ratio guarantee established in Theorem~\ref{thm:single_segment}. When comparing the performance of our algorithm with the two heuristics, one can observe that under all tested parameter settings, the assortments returned by Algorithm~\ref{alg:single_segment} attains an expected revenue more than $10\%$ higher than those returned by the two heuristics. This suggests that Algorithm~\ref{alg:single_segment}~is able to achieve a better empirical performance than the two heuristics under synthetic instances. 

}

{ 

\section{Joint Assortment Optimization and Product Framing with Covering Constraints}

In this section, we consider the joint assortment optimization and product framing problem under the MNL model with covering constraints. 

\subsection{Problem Formulation}
In this problem, the seller has $G$ positions indexed by $\mathcal{G}=\{1,2,\dots,G\}$, each of which can include one product. Customer's browsing probabilities are given by $\{\beta_g\}_{g\in\{1,2,\dots,G\}}$. With probability $\beta_g$, a customer browses from position $1$ through position $g$, and makes a purchase decision following the MNL model under the assortment offered in the first $g$ positions.

The decision of the seller is to decide a product framing $X:\mathcal{G}\to\mathcal{N}\cup\{0\}$ that maps from positions $\mathcal{G}$ to products. Under product framing $X$, the product placed at position $g$ is $X(g)$ for each $g\in\mathcal{G}$. When no product is placed at position $g$, we set $X(g)=0$. We use $X(g_1:g_2)$ to denote the set of products offered in positions $g_1$ through $g_2$ under product framing $X$. With probability $\beta_g$, a customer browses through positions $1$ through $g$, observes assortment $X(1:g)$, and generates an expected revenue of $R(X(1:g))$. Therefore the total expected revenue of a product framing $X$ is given by 
\begin{equation*}
\mathcal{R}(X)=\sum_{g=1}^G \beta_gR(X(1:g)). 
\end{equation*}
We further require that the products that appear in the product framing should satisfy the covering constraints, i.e., $|X(1:G)\cap C_k|\geq \ell_k$ for all $k\in\{1,2,\dots,K\}$. The formal definition of the problem is given by

\begin{equation}
\label{prob:framing}
\begin{aligned}
&\max_{X:\mathcal{G}\to\mathcal{N}\cup\{0\}} && \mathcal{R}(X)\\
&\text{s.t.} && |X(1:G)\cap C_k|\geq \ell_k,\ \forall k\in\{1,2,\dots,K\}.
\end{aligned}
\end{equation}

\subsection{Approximation Algorithm}

Before introducing our algorithm for Problem \eqref{prob:framing}, we first make the following assumptions.

\begin{assumption}\label{assump:monotone_browsing_distribution}
The browsing probabilities satisfy $\beta_g\leq \beta_{g'}$ for all $g\geq g'$.
\end{assumption}

\begin{assumption}\label{assump:enough_position}
The total number of positions $G$ satisfies $G\geq 3n$.
\end{assumption}

Assumption~\ref{assump:monotone_browsing_distribution}~states that the browsing probability $\beta_g$ is monotonically decreasing, while Assumption~\ref{assump:enough_position}~states that there are at least $3n$ positions available for product framing. Both assumptions are reasonable in practice. Customers' attention are usually limited to top-ranked products, resulting in a decreasing browsing probability. Furthermore, online platforms usually have a large number of positions available for products.

We construct $n$ candidate solutions $\{X_i\}_{i=1}^n$ for Problem~\eqref{prob:framing} as follows. For any $i\in\{1,2,\dots,n\}$, we define $S_i^*$ as an optimal solution to the following constrained assortment optimization problem
\begin{equation}
\label{prob:mnl_cardinality}
\begin{aligned}
&\max_{S\subseteq\mathcal{N}}&&R(S)&&\text{s.t.}&& |S|\leq i.
\end{aligned}
\end{equation}
\cite{rusmevichientong2010dynamic} shows that Problem~\eqref{prob:mnl_cardinality}~can be solved in polynomial time. We further define $\hat{S}$ as a approximate solution of \eqref{prob:deterministic_single_segment} we obtain using Algorithm \ref{alg:single_segment}. Then for any $i\in\{1,2,\dots,n\}$, we define a product framing $X_i$ by positioning $S_i^*$ in the first $|S_i^*|$ positions (in arbitrary order) and $\hat{S}$ in the last $|\hat{S}|$ positions, and leaving other positions empty. We further define $Y$ as the candidate solution with the maximum expected revenue, given by
\begin{equation}\label{eq:best_candidate_solution}
Y=\underset{X\in\{X_i\}_{i=1}^n}{\mathrm{argmax}}\,\mathcal{R}(X).
\end{equation}

\begin{theorem}
Under Assumptions~\ref{assump:monotone_browsing_distribution}~and~\ref{assump:enough_position}, $Y$ defined in \eqref{eq:best_candidate_solution} provides a $1/2\log_2(8n)$-approximation to Problem~\eqref{prob:framing}.
\end{theorem}

\begin{proof}

Let $X^*:\mathcal{G}\to\mathcal{N}\cup\{0\}$ be an optimal solution and $\text{OPT}$ be the optimal value of Problem~\eqref{prob:framing}. We define $\hat{X}$ as
\begin{equation*}
\hat{X}(g)=\begin{cases}
X^*(g),\ &\text{if}\ X^*(i)\notin X^*(1:g-1);\\
0,\ &\text{otherwise}.
\end{cases}
\end{equation*}
In other words, in the definition of $\hat{X}$, we remove duplicate products by only keeping each products in the position where it first appears. By definition of $\hat{X}$ we have $i\in \hat{X}(1:g)$ implies $i\in X^*(1:g)$. Furthermore, if $i\in X^*(1:g)$, then the first position where product $i$ appears is among positions $1$ through $g$, thus $i\in \hat{X}(1:g)$. Therefore we conclude that $X^*(1:g)=\hat{X}(1:g)$ for all $g\in\{1,2,\dots,G\}$, thus $\text{OPT}=\mathcal{R}(X^*)=\mathcal{R}(\hat{X})$. We further have
\begin{equation}\label{eq:framing_1}
\begin{split}
\text{OPT}&=\sum_{g=1}^G\beta_gR(\hat{X}(1:g))=\sum_{g=1}^G\beta_g\sum_{i\in \hat{X}(1:g)}r_i\phi(i,\hat{X}(1:g))\\
&=\sum_{g=1}^G\beta_g\sum_{g'=1}^g r_{\hat{X}(g')}\phi(\hat{X}(g'),\hat{X}(1:g))\\
&=\sum_{g'=1}^G\sum_{g=g'}^G \beta_g r_{\hat{X}(g')}\phi(\hat{X}(g'),\hat{X}(1:g)).
\end{split}
\end{equation}
Here for notational brevity we define $r_0=0$, and the third equality holds because by definition of $\hat{X}$, every product appears at most once in the product framing $\hat{X}$. We set $\ell_{\max}=\lceil\log_2(n)\rceil$, and define
\begin{equation*}
g_0=0,\ g_\ell=2^{\ell-1},\ \forall \ell\in\{1,2,\dots,\ell_{\max}\},\ g_{\ell_{\max}+1}=n,\ g_{\ell_{\max}+2}=G.
\end{equation*}

For any $\ell\in\{0,1,\dots,\ell_{\max}+1\}$, we define $\hat{X}_\ell$ as the product framing where positions $g_\ell+1$ through $g_{\ell+1}$ are kept same as $\hat{X}$ while all other positions are left empty, i.e.,
\begin{equation*}
\hat{X}_\ell(g)=\begin{cases}
\hat{X}(g),\ &\text{if}\ g_{\ell}+1\leq g\leq g_{\ell+1},\\
0,\ &\text{otherwise}.
\end{cases}
\end{equation*}
Then by~\eqref{eq:framing_1}~we get
\begin{align*}
\text{OPT}=&\sum_{g'=1}^G\sum_{g=g'}^G \beta_g r_{\hat{X}(g')}\phi(\hat{X}(g'),\hat{X}(1:g))=\sum_{\ell=0}^{\ell_{\max}+1}\sum_{g'=g_{\ell}+1}^{g_{\ell+1}}\sum_{g=g'}^G \beta_g r_{\hat{X}(g')}\phi(\hat{X}(g'),\hat{X}(1:g))\\
\leq&\sum_{\ell=0}^{\ell_{\max}+1}\sum_{g'=g_\ell+1}^{g_{\ell+1}}\sum_{g=g'}^G\beta_g r_{\hat{X}(g')}\phi(\hat{X}(g'),\hat{X}(g_\ell+1:\min\{g,g_{\ell+1}\}))=\sum_{\ell=0}^{\ell_{\max}+1}\mathcal{R}(\hat{X}_\ell).
\end{align*}
Here the inequality is because under the MNL model, $\phi(i,S)\geq \phi(i,S')$ if $i\in S$ and $S\subseteq S'$. We define $\ell^*=\mathrm{argmax}_{\ell}\,\mathcal{R}(\hat{X}_\ell)$, then we have
\begin{equation*}
\mathcal{R}(\hat{X}_{\ell^*})\geq \dfrac{1}{\ell_{\max}+2}\sum_{\ell=0}^{\ell_{\max}+1}\mathcal{R}(\hat{X}_\ell)\geq \dfrac{\text{OPT}}{\ell_{\max}+2}.
\end{equation*}

Then we prove that there exists $i\in\{1,2,\dots,n\}$ such that $\mathcal{R}(X_i)\geq \mathcal{R}(\hat{X}_{\ell^*})/2$. If $\ell^*\neq 0$ and $\ell^*\neq \ell_{\max}+1$, we set $i=g_{\ell^*+1}-g_{\ell^*}$, then
\begin{equation*}
\mathcal{R}(\hat{X}_{\ell^*})=\sum_{g=g_{\ell^*}+1}^{G}\beta_g R(\hat{X}(g_{\ell^*}+1:\min\{g,g_{\ell^*+1}\}))\leq \sum_{g=g_{\ell^*}+1}^{G}\beta_g R(S_i^*)\leq \sum_{g=i}^G\beta_gR(S_i^*).
\end{equation*}
Here the first inequality is due to $|\hat{X}(g_{\ell^*}+1:\min\{g,g_{\ell^*+1}\})|\leq g_{\ell^*+1}-g_{\ell^*}=i$, thus by definition of $S_i^*$ we have $R(S_i^*)\geq R(\hat{X}(g_{\ell^*}+1:g))$ for all $g_{\ell^*}+1\leq g\leq G$. The second inequality is due to the assumption that $\ell^*\neq \ell_{\max}+1$ and $\ell^*\neq 0$, thus $i=g_{\ell^*+1}-g_{\ell^*}\leq g_{\ell^*}$. Furthermore, since the approximate solution $\hat{S}$ of \eqref{prob:deterministic_single_segment} satisfies $|\hat{S}|\leq n$, under product framing $X_i$, if a customer browses up to $g$ for some $i\leq g\leq G-n$, the set of products offered to the customer is $S_i^*$. Therefore
\begin{equation*}
\mathcal{R}(X_i)\geq \sum_{g=i}^{G-n}\beta_g R(S_i^*).
\end{equation*}
We further have
\begin{equation*}
\sum_{g=G-n+1}^G \beta_g\leq \sum_{g=G-2n+1}^{G-n}\beta_g\leq \sum_{g=i}^{G-n}\beta_g.
\end{equation*}
Here the first inequality is because by Assumption~\ref{assump:monotone_browsing_distribution}~we get $\beta_g\leq \beta_{g-n}$ for all $G-n+1\leq g\leq G$. The second inequality is because we have assumed $\ell^*\neq \ell_{\max}+1$, thus $i=g_{\ell^*+1}-g_{\ell^*}\leq g_{\ell^*+1}\leq n$, and by Assumption~\ref{assump:enough_position}~we have $G-2n+1\geq n\geq i$. Therefore we have $\sum_{g=i}^{G-n}\beta_g\leq \sum_{g=i}^G\beta_g/2$. Then we get
\begin{equation*}
\mathcal{R}(X_i)\geq \sum_{g=i}^{G-n}\beta_g R(S_i^*)\geq \dfrac{1}{2}\sum_{g=i}^G \beta_gR(S_i^*)\geq \dfrac{1}{2}\mathcal{R}(\hat{X}_{\ell^*})\geq \dfrac{\text{OPT}}{2(\ell_{\max}+2)}.
\end{equation*}

If $\ell^*=0$, since by definition $g_1=1$, we have
\begin{equation*}
\mathcal{R}(\hat{X}_0)=\sum_{g=1}^G \beta_g R(\{\hat{X}_0(1)\})=R(\{\hat{X}(1)\})\cdot\sum_{g=1}^G\beta_g=R(\{\hat{X}(1)\})\leq R(S_1^*).
\end{equation*}
On the other hand, we have
\begin{equation*}
\mathcal{R}(X_1)\geq \sum_{g=1}^{G-n}\beta_gR(S_1^*).
\end{equation*}
By Assumption~\ref{assump:monotone_browsing_distribution}~we also have
\begin{equation*}
\sum_{g=G-n+1}^G\beta_g\leq \sum_{g=G-2n+1}^{G-n}\beta_g\leq \sum_{g=1}^{G-n}\beta_g.
\end{equation*}
Therefore we have $\sum_{g=1}^{G-n}\beta_g\geq \sum_{g=1}^G\beta_g/2$, therefore
\begin{equation*}
\mathcal{R}(X_1)\geq \sum_{g=1}^{G-n}\beta_g R(S_1^*)\geq \dfrac{1}{2}\sum_{g=1}^G \beta_gR(S_1^*)\geq \mathcal{R}(\hat{X}_0)\geq \dfrac{\text{OPT}}{2(\ell_{\max}+2)}.
\end{equation*}

If $\ell^*=\ell_{\max}+1$, then we have
\begin{equation*}
\mathcal{R}(\hat{X}_{\ell^*})=\sum_{g=n+1}^G \beta_gR(\hat{X}(n:g))\leq \sum_{g=n+1}^G\beta_gR(S_n^*).
\end{equation*}
Here the inequality is because $S_n^*$ is the unconstrained optimal assortment, thus $R(S_n^*)\geq R(S)$ for all $S\subseteq\mathcal{N}$. On the other hand, we have
\begin{equation*}
\mathcal{R}(X_n)\geq \sum_{g=n}^{G-n}\beta_g R(S_n^*).
\end{equation*}
We also have
\begin{equation*}
\sum_{g=G-n+1}^{G}\beta_g\leq \sum_{g=G-2n+1}^{G-n}\beta_g\leq \sum_{g=n}^{G-n}\beta_g.
\end{equation*}
Here the first inequality is because by Assumption~\ref{assump:monotone_browsing_distribution}~we get $\beta_g\leq \beta_{g-n}$ for all $G-n+1\leq g\leq G$. The second inequality is because by Assumption~\ref{assump:enough_position} we have $G-2n+1\geq n$. Therefore we have $\sum_{g=n}^{G-n}\beta_g\geq \sum_{g=n}^G \beta_g/2$, thus
\begin{equation*}
\mathcal{R}(X_n)\geq \sum_{g=n}^{G-n}\beta_g R(S_n^*)\geq \dfrac{1}{2}\sum_{g=n+1}^G \beta_gR(S_n^*)\geq \dfrac{1}{2}\mathcal{R}(\hat{X}_{\ell^*})\geq \dfrac{\text{OPT}}{2(\ell_{\max}+2)}.
\end{equation*}
Therefore in all three cases, there exists $i\in\{1,2,\dots,n\}$ such that $\mathcal{R}(X_i)\geq \text{OPT}/2(\ell_{\max}+2)$. Since we define $Y=\mathrm{argmax}_{X\in\{X_i\}_{i=1}^n}\mathcal{R}(X)$, we have
\begin{equation*}
\mathcal{R}(Y)\geq \dfrac{\text{OPT}}{2(\ell_{\max}+2)}\geq \dfrac{\text{OPT}}{2\log_2(8n)}.
\end{equation*}

Finally we show that for any $i\in\{1,2,\dots,n\}$, $X_i$ satisfies the covering constraints. Since $\hat{S}$ is always offered at the last $|\hat{S}|$ positions of $X_i$, we get $X_i(1:G)\supseteq \hat{S}$, thus $|X_i(1:G)\cap C_k|\geq |\hat{S}\cap C_k|\geq \ell_k$ for all $k\in\{1,2,\dots,K\}$. Therefore $Y$ satisfies the covering constraints, thus is a $1/2\log(8n)$-approximation to Problem~\eqref{prob:framing}.
\end{proof}

}

\end{appendices}

\end{document}